\documentclass{amsart}

\usepackage{amsfonts,amssymb,stmaryrd,amscd,amsmath,latexsym,amsbsy,mathrsfs,xypic,eucal}
\renewcommand{\mathcal}{}\let\mathcal=\CMcal

\usepackage{verbatim}

\usepackage{amsthm}

\numberwithin{equation}{section}
\newtheorem{thm}{Theorem}[section]
\newtheorem{prop}[thm]{Proposition}

\newtheorem{defi}[thm]{Definition}
\newtheorem{lem}[thm]{Lemma}
\newtheorem{cor}[thm]{Corollary}
\newtheorem{eg}[thm]{{Example}}

\newtheorem{rema}[thm]{Remark}

\newcommand{\ad}{{\mbox{\upshape{ad}}}}
\newcommand{\Ad}{{\mbox{\upshape{Ad}}}}

\newcommand{\Aut}{\mathrm{Aut}}

\newcommand{\barB}{\overline{\phantom{a,}}^B}
\newcommand{\bc}{{\mathbf{c}}}

\newcommand{\bs}{{\mathbf{s}}}

\newcommand{\bfrak}{{\mathfrak b}}

\newcommand{\C}{{\mathbb C}}

\newcommand{\N}{{\mathbb N}}

\newcommand{\cA}{{\mathcal A}}

\newcommand{\cB}{{\mathcal B}}
\newcommand{\cC}{{\mathcal C}}
\newcommand{\ccf}{{\mathsf f}}

\newcommand{\cK}{{\mathcal K}}
\newcommand{\cM}{{\mathcal M}}
\newcommand{\cO}{{\mathcal O}}
\newcommand{\calR}{{\mathcal R}}
\newcommand{\cS}{{\mathcal S}}

\newcommand{\cV}{\mathcal{V}ect}

\newcommand{\cZ}{{\mathcal Z}}

\newcommand{\dkm}{{{\mathrm{dkm}}}}

\newcommand{\End}{\mbox{End}}

\newcommand{\field}{{\mathbb K}}
\newcommand{\fieldtwo}{k}
\newcommand{\flip}{{\mathrm{flip}}}
\newcommand{\For}{\mathcal{F}or}
\newcommand{\free}{{}'{\mathbf f}}

\newcommand{\Ftwo}{{\For^{(2)}}}
\newcommand{\Ftwoop}{{\For^{(2)\mathrm{op}}}}
\newcommand{\gfrak}{{\mathfrak g}}
\newcommand{\glfrak}{{\mathfrak{gl}}}

\newcommand{\hfrak}{{\mathfrak h}}

\newcommand{\Hom}{{\mathrm{Hom}}}

\newcommand{\hght}{\mathrm{ht}}

\newcommand{\id}{{\mathrm{id}}}

\newcommand{\kfrak}{{\mathfrak k}}
\newcommand{\kow}{{\varDelta}}
\newcommand{\lact}{{\triangleright}}

\newcommand{\alphatil}{\tilde{\alpha}}

\newcommand{\lambdatil}{\tilde{\lambda}}

\newcommand{\Oint}{{{\mathcal O}_{int}}}

\newcommand{\op}{\mathrm{op}}
\newcommand{\ot}{\otimes}

\newcommand{\Qvext}{Q^\vee_{ext}}

\newcommand{\rh}{\hat{R}}
\newcommand{\ri}{r_i}
\newcommand{\ir}{{}_ir}

\newcommand{\RtX}{R^{(\tau,X)}}
\newcommand{\Q}{\mathbb Q}

\newcommand{\rank}{\mathrm{rank}}

\newcommand{\scrU}{\mathscr{U}}

\newcommand{\slfrak}{{\mathfrak{sl}}}

\newcommand{\sU}{{\mathscr{U}}}

\newcommand{\tw}{{\mathrm{tw}}}
\newcommand{\twn}{{tw}}

\newcommand{\uc}{\mathbf{c}}

\newcommand{\uqg}{{U_q(\mathfrak{g})}}

\newcommand{\uqgp}{{U_q(\mathfrak{g'})}}

\newcommand{\uqislzi}{U_{q_i}(\mathfrak{sl}_2)_i}

\newcommand{\vep}{\varepsilon}

\newcommand{\uX}{\underline{X}}

\newcommand{\Z}{{\mathbb Z}}

\newcommand{\coid}{{B_{\mathbf{c},\mathbf{s}}}}
\newcommand{\Xfrak}{\mathfrak{X}}

\newcommand{\g}{\mathfrak{g}}
\newcommand{\oci}{{\overline{c_i}}}

\title{Universal K-matrix for quantum symmetric pairs}
\thanks{Research supported by EPSRC grant EP/K025384/1}

\author{Martina Balagovi\'c}
\author{Stefan Kolb}
\address{Martina Balagovi\'c and Stefan Kolb, School of Mathematics and Statistics, Newcastle University, Newcastle upon Tyne NE1 7RU, UK}
\email{martina.balagovic@newcastle.ac.uk}
\email{stefan.kolb@newcastle.ac.uk}


\begin{document}

\begin{abstract}
Let $\gfrak$ be a symmetrizable Kac-Moody algebra and let $\uqg$ denote the corresponding quantized enveloping algebra. In the present paper we show that quantum symmetric pair coideal subalgebras $\coid$ of $\uqg$ have a universal K-matrix if $\gfrak$ is of finite type. By a universal K-matrix for $\coid$ we mean an element in a completion of $\uqg$ which commutes with $\coid$ and provides solutions of the reflection equation in all integrable $\uqg$-modules in category $\cO$. The construction of the universal K-matrix for $\coid$ bears significant resemblance to the construction of the universal R-matrix for $\uqg$. Most steps in the construction of the universal K-matrix are performed in the general Kac-Moody setting. 

In the late nineties T.~tom Dieck and R.~H{\"a}ring-Oldenburg developed a program of representations of categories of ribbons in a cylinder. Our results show that quantum symmetric pairs provide a large class of examples for this program. 
\end{abstract}

\maketitle

\section{Introduction}
%%%%%%%%%%%%%%%%%%%%%%%%%%%%%%%%%%
\subsection{Background}
%%%%%%%%%%%%%%%%%%%%%%%%%%%%%%%%%% 
  Let $\gfrak$ be a symmetrizable Kac-Moody algebra and $\theta:\gfrak\rightarrow \gfrak$ an involutive Lie algebra automorphism. Let $\kfrak=\{x\in \gfrak\,|\,\theta(x)=x\}$ denote the fixed Lie subalgebra. We call the pair of Lie algebras $(\gfrak,\kfrak)$ a symmetric pair. Assume that $\theta$ is of the second kind, which means that the standard Borel subalgebra $\bfrak^+$ of $\gfrak$ satisfies $\dim(\theta(\bfrak^+)\cap\bfrak^+)<\infty$.  In this case the universal enveloping algebra $U(\kfrak)$ has a quantum group analog $\coid=\coid(\theta)$ which is a right coideal subalgebra of the Drinfeld-Jimbo quantized enveloping algebra $\uqg$, see \cite{a-Letzter99a}, \cite{MSRI-Letzter}, \cite{a-Kolb14}. We call $(\uqg,\coid)$ a quantum symmetric pair.
  
The theory of quantum symmetric pairs was first developed by M.~Noumi, T.~Sugitani, and M.~Dijkhuizen for all classical Lie algebras in \cite{a-Noumi96}, \cite{a-NS95}, \cite{a-NDS97}, \cite{a-Dijk96}. The aim of this program was to perform harmonic analysis on quantum group analogs of compact symmetric spaces. This allowed an interpretation of Macdonald polynomials as quantum zonal spherical functions. Independently, G.~Letzter developed a comprehensive theory of quantum symmetric pairs for all semisimple $\gfrak$ in \cite{a-Letzter99a}, \cite{MSRI-Letzter}. Her approach uses the Drinfeld-Jimbo presentation of quantized enveloping algebras and hence avoids casework. Letzter's theory also aimed at applications in harmonic analysis for quantum group analogs of symmetric spaces \cite{a-Letzter04}, \cite{a-Letzter-memoirs}. The algebraic theory of quantum symmetric pairs was extended to the setting of Kac-Moody algebras in \cite{a-Kolb14}. 

Over the past two years it has emerged that quantum symmetric pairs play an important role in a much wider representation theoretic context. In a pioneering paper H.~Bao and W.~Wang proposed a program of canonical bases for quantum symmetric pairs \cite{a-BaoWang13p}. They performed their program for the symmetric pairs $(\slfrak_{2N},\mathfrak{s}(\glfrak_N\times \glfrak_N))$ and $(\slfrak_{2N+1},\mathfrak{s}(\glfrak_N\times \glfrak_{N+1}))$ and applied it to establish Kazhdan-Lusztig theory for the category $\cO$ of the ortho-symplectic Lie superalgebra $\mathfrak{osp}(2n{+}1\,|\,2m)$. Bao and Wang developed the theory for these two examples in astonishing similarity to Lusztig's exposition of quantized enveloping algebras in \cite{b-Lusztig94}. In a closely related program M.~Ehrig and C.~Stroppel showed that quantum symmetric pairs for $(\glfrak_{2N}, \glfrak_N\times\glfrak_N)$ and $(\glfrak_{2N+1}, \glfrak_N\times \glfrak_{N+1})$ appear via categorification using parabolic category $\cO$ of type $D$ \cite{a-EhrigStroppel13p}. The recent developments as well as the previously known results suggest that quantum symmetric pairs allow as deep a theory as quantized enveloping algebras themselves. It is reasonable to expect that most results about quantized enveloping algebras have analogs for quantum symmetric pairs. 

One of the fundamental properties of the quantized enveloping algebra $\uqg$ is the existence of a universal R-matrix which gives rise to solutions of the quantum Yang-Baxter equation for suitable representations of $\uqg$. The universal R-matrix is at the heart of the origins of quantum groups in the theory of quantum integrable systems \cite{inp-Drinfeld1}, \cite{a-Jimbo1} and of the applications of quantum groups to invariants of knots, braids, and ribbons  \cite{a-RT90}. Let $\kow:\uqg\rightarrow \uqg\ot \uqg$ denote the coproduct of $\uqg$ and let $\kow^{\mathrm{op}}$ denote the opposite coproduct obtained by flipping tensor factors. The universal R-matrix $R^U$ of $\uqg$ is an element in a completion $\sU^{(2)}_0$ of $\uqg\ot \uqg$, see Section \ref{sec:U-coproduct}. It has the following two defining properties:
\begin{enumerate}
  \item In $\sU^{(2)}_0$ the element $R^U$ satisfies the relation $\kow(u) R^U=R^U\kow^{\mathrm{op}}(u)$ for all $u\in \uqg$.
  \item The relations
    \begin{align*}
      (\kow\ot \id)(R^U)= R^U_{23} R^U_{13},\qquad(\id \ot \kow)(R^U)= R^U_{12} R^U_{13}
    \end{align*}
    hold. Here we use the usual leg notation for threefold tensor products. 
\end{enumerate}
The universal R-matrix gives rise to a family $\rh=(R_{M,N})$ of commutativity isomorphisms $\rh_{M,N}:M\ot N\rightarrow N\ot M$ for all category $\cO$ representations $M,N$ of $\uqg$. 
In our conventions one has $\rh_{M,N}=R^U\circ \flip_{M,N}$ where $\flip_{M,N}$ denotes the flip of tensor factors. The family $\rh$ can be considered as an element in an extension $\sU^{(2)}$ of the completion $\sU^{(2)}_0$ of $\uqg\ot \uqg$, see Section \ref{sec:Rmatrix} for details. In $\sU^{(2)}$ property (1) of $R^U$ can be rewritten as follows:
\begin{enumerate}
  \item[(1')] In $\sU^{(2)}$ the element $\rh$ commutes with $\kow(u)$ for all $u\in \uqg$.
\end{enumerate} 
By definition the family of commutativity isomorphisms $\rh=(\rh_{M,N})$ is natural in $M$ and $N$. The above relations mean that $\rh$ turns category $\cO$ for $\uqg$ into a braided tensor category. 

The analog of the quantum Yang-Baxter equation for quantum symmetric pairs is known as the boundary quantum Yang-Baxter equation or (quantum) reflection equation. It first appeared in I.~Cherednik's investigation of factorized scattering on the half line \cite{a-Cher84} and in E.~Sklyanin's investigation of quantum integrable models with non-periodic boundary conditions \cite{a-Skl88}, \cite{a-KuSkl92}. In \cite[6.1]{a-KuSkl92} an element providing solutions of the reflection equation in all representations was called a `universal K-matrix'. Explicit examples of universal K-matrices for $U_q(\slfrak_2)$ appeared in \cite[(3.31)]{a-CremGer92} and \cite[(2.20)]{a-KSS93}. 

A categorical framework for solutions of the reflection equation was proposed by T.~tom Dieck and R.~H{\"a}ring-Oldenburg under the name braided tensor categories with a cylinder twist \cite{a-tD98}, \cite{a-tDHO98}, \cite{a-HaOld01}. Their program provides an extension of the graphical calculus for braids and ribbons in $\C\times [0,1]$ as in \cite{a-RT90} to the setting of braids and ribbons in the cylinder $\C^\ast\times[0,1]$, see \cite{a-HaOld01}. It hence corresponds to an extension of the theory from the classical braid group of type $A_{N-1}$ to the braid group of type $B_N$. Tom Dieck and H{\"a}ring-Oldenburg called the analog of the universal R-matrix in this setting a universal cylinder twist. They determined a family of universal cylinder twists for $U_q(\slfrak_2)$ by direct calculation \cite[Theorem 8.4]{a-tDHO98}. This family essentially coincides with the universal K-matrix in \cite[(2.20)]{a-KSS93} where it was called a universal solution of the reflection equation. 
%%%%%%%%%%%%%%%%%%%%%%%%%
\subsection{Universal K-matrix for coideal subalgebras}
%%%%%%%%%%%%%%%%%%%%%%%%%
Special solutions of the reflection equation were essential ingredients in the initial construction of quantum symmetric pairs by Noumi, Sugitani, and Dijkhuizen \cite{a-Noumi96}, \cite{a-NS95}, \cite{a-NDS97}, \cite{a-Dijk96}. For this reason it is natural to expect that quantum symmetric pairs give rise to universal K-matrices. The fact that quantum symmetric pairs $\coid$ are coideal subalgebras of $\uqg$ moreover suggests to base the concept of a universal K-matrix on a coideal subalgebra of a braided (or quasitriangular) Hopf algebra. 

Recall that a subalgebra $B$ of $\uqg$ is called a right coideal subalgebra if
\begin{align*}
  \kow(B) \subset B\ot \uqg.
\end{align*}
In the present paper we introduce the notion of a universal K-matrix for a right coideal subalgebra $B$ of $\uqg$. A universal K-matrix for $B$ is an element $\cK$ in a suitable completion $\sU$ of $\uqg$ with the following properties:
\begin{enumerate}
  \item In $\sU$ the universal K-matrix $\cK$ commutes with all $b\in B$.
  \item The relation
    \begin{align}\label{eq:kowcK}
      \kow(\cK)= (\cK\ot 1)\cdot \rh \cdot (\cK\ot 1) \cdot \rh 
    \end{align}
    holds in the completion $\sU^{(2)}$ of $\uqg\ot \uqg$.
\end{enumerate}
See Definition \ref{def:U-cylinder-braided} for details. By definition of the completion $\sU$, a universal K-matrix is a family $\cK=(K_M)$ of linear maps $K_M:M\rightarrow M$ for all integrable $\uqg$-modules in category $\cO$. Moreover, this family is natural in $M$. The defining properties (1) and (2) of $\cK$ are direct analogs of the defining properties (1') and (2) of the universal R-matrix $R^U$. The fact that $\rh$ commutes with $\kow(\cK)$ immediately implies that $\cK$ satisfies the reflection equation
\begin{align*}
\rh \cdot (\cK\ot 1)\cdot \rh \cdot (\cK\ot 1) =(\cK\ot 1)\cdot \rh \cdot (\cK\ot 1) \cdot \rh 
\end{align*}
in $\sU^{(2)}$. By \eqref{eq:kowcK} and the naturality of $\cK$ a universal K-matrix for $B$ gives rise to the structure of a universal cylinder twist on the braided tensor category of integrable $\uqg$-modules in category $\cO$. Universal K-matrices, if they can be found, hence provide examples for the theory proposed by tom Dieck and H{\"a}ring-Oldenburg. The new ingredient in our definition is the coideal subalgebra $B$. We will see in this paper that $B$ plays a focal role in finding a universal K-matrix.

The notion of a universal K-matrix can be defined for any coideal subalgebra of a braided bialgebra $H$ with universal R-matrix $R^H\in H\ot H$. This works in complete analogy to the above definition for $B$ and $\uqg$, and it avoids completions, see Section \ref{sec:cylinder-braided} for details. Following the terminology of \cite{a-tD98}, \cite{a-tDHO98} we call a coideal subalgebra $B$ of $H$ cylinder-braided if it has a universal K-matrix.

A different notion of a universal K-matrix for a braided Hopf algebra $H$ was previously introduced by J.~Donin, P.~Kulish, and A.~Mudrov in \cite{a-DKM03}. Let $R^H_{21}\in H\ot H$ denote the element obtained from $R^H$ by flipping the tensor factors. Under some technical assumptions the universal K-matrix in \cite{a-DKM03} is just the element $R^H R^H_{21}\in H\ot H$. Coideal subalgebras only feature indirectly in this setting. We explain this in Section \ref{sec:characters}.

In a dual setting of coquasitriangular Hopf algebras the relations between the constructions in \cite{a-DKM03}, the notion of a universal cylinder twist \cite{a-tD98}, \cite{a-tDHO98}, and the theory of quantum symmetric pairs was already discussed by J.~Stokman and the second named author in \cite{a-KolbStok09}. In that paper universal K-matrices were found for quantum symmetric pairs corresponding to the symmetric pairs $(\slfrak_{2N},\mathfrak{s}(\glfrak_N\times \glfrak_N))$ and $(\slfrak_{2N+1},\mathfrak{s}(\glfrak_N\times \glfrak_{N+1}))$. However, a general construction was still outstanding.
%%%%%%%%%%%%%%%%%%%%%%%%%%%%%%%%%%%
\subsection{Main results}
%%%%%%%%%%%%%%%%%%%%%%%%%%%%%%%%%%%
The main result of the present paper is the construction of a universal K-matrix for every quantum symmetric pair coideal subalgebra $\coid$ of $\uqg$ for $\gfrak$ of finite type. This provides an analog of the universal R-matrix for quantum symmetric pairs. Moreover, it shows that important parts of Lusztig's book \cite[Chapters 4 \& 32]{b-Lusztig94} translate to the setting of quantum symmetric pairs. 

The construction in the present paper is significantly inspired by the examples $(\slfrak_{2N},\mathfrak{s}(\glfrak_N\times \glfrak_N))$ and $(\slfrak_{2N+1},\mathfrak{s}(\glfrak_N\times \glfrak_{N+1}))$ considered by Bao and Wang in \cite{a-BaoWang13p}. The papers \cite{a-BaoWang13p} and \cite{a-EhrigStroppel13p} both observed the existence of a bar involution for quantum symmetric pair coideal subalgebras $\coid$ in this special case. Bao and Wang then constructed an intertwiner $\Upsilon\in \sU$ between the new bar involution and Lusztig's bar involution. The element $\Upsilon$ is hence an analog of the quasi R-matrix in Lusztig's approach to quantum groups, see \cite[Theorem 4.1.2]{b-Lusztig94}. Similar to the construction of the commutativity isomorphisms in \cite[Chapter 32]{b-Lusztig94} Bao and Wang construct a $\coid$-module homomorphism $\EuScript{T}_M:M\rightarrow M$ for any finite-dimensional representation $M$ of $U_q(\slfrak_N)$. If $M$ is the vector representation they show that $\EuScript{T}_M$ satisfies the reflection equation and they establish Schur-Jimbo duality between the coideal subalgebra and a Hecke algebra of type $B_N$ acting on $V^{\ot N}$.

In the present paper we consider quantum symmetric pairs in full generality and formulate results in the Kac-Moody setting whenever possible. The existence of the bar involution 
\begin{align*}
  \barB:\coid\rightarrow \coid, \qquad x\mapsto \overline{x}^B
\end{align*}
for the quantum symmetric pair coideal subalgebra $\coid$ was already established in \cite{a-BalaKolb14p}. Following \cite[Section 2]{a-BaoWang13p} closely we now prove the existence of an intertwiner between the two bar involutions. More precisely, we show in Theorem \ref{thm:Xfrak} that there exists a nonzero element $\Xfrak\in \sU$ which satisfies the relation
\begin{align}\label{eq:X-cond}
   \overline{x}^B\, \Xfrak = \Xfrak\, \overline{x} \qquad \mbox{for all $x\in\coid$.}
\end{align}
We call the element $\Xfrak$ the quasi K-matrix for $\coid$. It corresponds to the intertwiner $\Upsilon$ in the setting of \cite{a-BaoWang13p}.  

Recall from \cite[Theorem 2.7]{a-Kolb14} that the involutive automorphism $\theta:\gfrak\rightarrow \gfrak$ is determined by a pair $(X,\tau)$ up to conjugation. Here $X$ is a subset of the set of nodes of the Dynkin diagram of $\gfrak$ and $\tau$ is a diagram automorphism. The Lie subalgebra $\gfrak_X\subset \gfrak$ corresponding to $X$ is required to be of finite type. Hence there exists a longest element $w_X$ in the parabolic subgroup $W_X$ of the Weyl group $W$. The Lusztig automorphism $T_{w_X}$ may be considered as an element in the completion $\sU$ of $\uqg$, see Section \ref{sec:completion}. We define
\begin{align}\label{eq:K'-def}
  \cK' = \Xfrak\, \xi \, T_{w_X}^{-1} \in \sU
\end{align}
where $\xi \in \sU$ denotes a suitably chosen element which acts on weight spaces by a scalar. The element $\cK'$ defines a linear isomorphism 
\begin{align}\label{eq:KM'}
  \cK'_M:M\rightarrow M
\end{align}
for every integrable $\uqg$-module $M$ in category $\cO$. In Theorem \ref{thm:Bc-hom} we show that $\cK'_M$ is a $\coid$-module homomorphism if one twists the $\coid$-module structure on both sides of \eqref{eq:KM'} appropriately. The element $\cK'$ exists in the general Kac-Moody case.

For $\gfrak$ of finite type there exists a longest element $w_0\in W$ and a corresponding family of Lusztig automorphisms $T_{w_0}=(T_{w_0,M})\in \sU$. In this case we define
\begin{align}\label{eq:K-def-intro}
  \cK =\Xfrak\, \xi\, T_{w_X}^{-1}\, T_{w_0}^{-1}\in \sU.
\end{align} 
For the symmetric pairs $(\slfrak_{2N},\mathfrak{s}(\glfrak_N\times \glfrak_N))$ and $(\slfrak_{2N+1},\mathfrak{s}(\glfrak_N\times \glfrak_{N+1}))$ the construction of $\cK$ coincides with the construction of the $\coid$-module homomorphisms $\EuScript{T}_M$ in \cite{a-BaoWang13p} up to conventions.
The longest element $w_0$ induces a diagram automorphism $\tau_0$ of $\gfrak$ and of $\uqg$. Any $\uqg$-module $M$ can be twisted by an algebra automorphism $\varphi:\uqg \rightarrow \uqg$ if we define $u\lact m=\varphi(u)m$ for all $u\in \uqg$, $m\in M$. We denote the resulting twisted module by $M^\varphi$. We show in Corollary \ref{cor:K} that the element $\cK$ defines a $\coid$-module isomorphism
\begin{align}\label{eq:KMMtt0}
  \cK_M: M\rightarrow M^{\tau \tau_0}
\end{align}
for all finite-dimensional $\uqg$-modules $M$. Alternatively, this can be written as
\begin{align*}
  \cK b =\tau_0(\tau(b)) \cK \qquad \mbox{for all $b\in \coid$.}
\end{align*}

The construction of the bar involution for $\coid$, the intertwiner $\Xfrak$, and the $\coid$-module homomorphism $\cK$ are three expected key steps in the wider program of canonical bases for quantum symmetric pairs proposed in \cite{a-BaoWang13p}. The existence of the bar involution was explicitly stated without proof and reference to the parameters in \cite[0.5]{a-BaoWang13p} and worked out in detail in \cite{a-BalaKolb14p}. Weiqiang Wang has informed us that he and Huanchen Bao have constructed $\Xfrak$ and $\cK'_M$ independently in the case $X=\emptyset$, see \cite{a-BaoWang15in-prep}.

In the final Section \ref{sec:coproductK} we address the crucial problem to determine the coproduct $\kow(\cK)$ in $\sU^{(2)}$. The main step to this end is to determine the coproduct of the quasi K-matrix $\Xfrak$ in Theorem \ref{thm:kowX}. Even for the symmetric pairs $(\slfrak_{2N},\mathfrak{s}(\glfrak_N\times \glfrak_N))$ and $(\slfrak_{2N+1},\mathfrak{s}(\glfrak_N\times \glfrak_{N+1}))$, this calculation goes beyond what is contained in \cite{a-BaoWang13p}. It turns out that if $\tau\tau_0=\id$ then the coproduct $\kow(\cK)$ is given by formula \eqref{eq:kowcK}. Hence, in this case $\cK$ is a universal K-matrix as defined above for the coideal subalgebra $\coid$. If $\tau\tau_0\neq \id$ then we obtain a slight generalization of the properties (1) and (2) of a universal K-matrix. Motivated by this observation we introduce the notion of a $\varphi$-universal K-matrix for $B$ if $\varphi$ is an automorphism of a braided bialgebra $H$ and $B$ is a right coideal subalgebra, see Section \ref{sec:cylinder-braided}. With this terminology it hence turns out in Theorem \ref{thm:deltaK} that in general $\cK$ is a $\tau\tau_0$-universal K-matrix for $\coid$. The fact that $\tau\tau_0$ may or may not be the identity provides another conceptual explanation for the occurrence of two distinct reflection equations in the Noumi-Sugitani-Dijkhuizen approach to quantum symmetric pairs. 
%%%%%%%%%%%%%%%%%%%%%%%%%%%%%%%%%%%%%%%%%%%%%
\subsection{Organization}
%%%%%%%%%%%%%%%%%%%%%%%%%%%%%%%%%%%%%%%%%%%%%
Sections \ref{sec:prelim}--\ref{sec:QSP} are of preparatory nature. In Section \ref{sec:prelim} we fix notation for Kac-Moody algebras and quantized enveloping algebras, mostly following \cite{b-Kac1}, \cite{b-Lusztig94}, and \cite{b-Jantzen96}. In Section \ref{sec:completion} we discuss the completion $\sU$ of $\uqg$ and the completion $\sU^{(2)}_0$ of $\uqg\ot \uqg$. In particular, we consider Lusztig's braid group action and the commutativity isomorphisms $\rh$ in this setting. 

Section \ref{sec:CylTw+RE} is a review of the notion of a braided tensor category with a cylinder twist as introduced by tom Dieck and H{\"a}ring-Oldenburg. We extend their original definition by a twist in Section \ref{sec:TwCylTw} to include all the examples obtained from quantum symmetric pairs later in the paper. The categorical definitions lead us in Section \ref{sec:cylinder-braided} to introduce the notion of a cylinder-braided coideal subalgebra of a braided bialgebra. By definition this is a coideal subalgebra which has a universal K-matrix. We carefully formulate the analog definition for coideal subalgebras of $\uqg$ to take into account the need for completions. Finally, in Section \ref{sec:characters} we recall the different definition of a universal K-matrix from \cite{a-DKM03} and indicate how it relates to cylinder braided coideal subalgebras as defined here. 

Section \ref{sec:QSP} is a brief summary of the construction  and properties of the quantum symmetric pair coideal subalgebras $\coid$ in the conventions of \cite{a-Kolb14}. In Section \ref{sec:bar-involution} we recall the existence of the bar involution for $\coid$ following \cite{a-BalaKolb14p}. The quantum symmetric pair coideal subalgebra $\coid$ depends on a choice of parameters, and the existence of the bar involution imposes additional restrictions. In Section \ref{sec:Assumptions} we summarize our setting, including all restrictions on the parameters $\bc, \bs$.

The main new results of the paper are contained in Sections \ref{sec:quasiK}--\ref{sec:coproductK}. In Section \ref{sec:quasiK} we prove the existence of the quasi K-matrix $\Xfrak$. The defining condition \eqref{eq:X-cond} gives rise to an overdetermined recursive formula for the weight components of $\Xfrak$. The main difficulty is to prove the existence of elements satisfying the recursion. To this end, we translate the inductive step into a more easily verifiable condition in Section \ref{sec:SysEq}. This condition is expressed solely in terms of the constituents of the generators of $\coid$, and it is verified in Section \ref{sec:constructXmu}. This allows us to prove the existence of $\Xfrak$ in Section \ref{sec:constructingX}. A similar argument is contained in \cite[2.4]{a-BaoWang13p} for the special examples $(\slfrak_{2N},\mathfrak{s}(\glfrak_N\times \glfrak_N))$ and $(\slfrak_{2N+1},\mathfrak{s}(\glfrak_N\times \glfrak_{N+1}))$. However, the explicit formulation of the conditions in Proposition \ref{prop:Xi} seems to be new.

In Section \ref{sec:K-construction} we consider the element $\cK'\in \sU$ defined by \eqref{eq:K'-def}. In Section \ref{sec:pseudoT} we define a twist of $\uqg$ which reduces to the Lusztig action $T_{w_0}$ if $\gfrak$ is of finite type. We also record an additional Assumption ($\tau_0$) on the parameters. In Section \ref{sec:generalK'} this assumption is used in the proof that $\cK'_M:M\rightarrow M$ is a $\coid$-module isomorphism of twisted $\coid$-modules. In the finite case this immediately implies that the element $\cK$ defined by \eqref{eq:K-def-intro} gives rise to an $\coid$-module isomorphism \eqref{eq:KMMtt0}. Up to a twist this verifies the first condition in the definition of a universal K-matrix for $\coid$.

The map $\xi$ involved in the definition of $\cK'$ is discussed in more detail in Section \ref{sec:xi-construction}. So far, the element $\xi$ was only required to satisfy a recursion which guarantees that $\cK'_M$ is a $\coid$-module homomorphism. In Section \ref{sec:choosing-xi} we choose $\xi$ explicitly and show that our choice satisfies the required recursion. In Section \ref{sec:xi-coproduct} we then determine the coproduct of this specific $\xi$ considered as an element in the completion $\sU$. Moreover, in Section \ref{sec:xi-adjoint} we discuss the action of $\xi$ on $\uqg$ by conjugation. This simplifies later calculations.

In Section \ref{sec:coproductK} we restrict to the finite case. We first perform some preliminary calculations with the quasi R-matrices of $\uqg$ and $U_q(\gfrak_X)$. This allows us in Section \ref{sec:deltaX} to determine the coproduct of the quasi K-matrix $\Xfrak$, see Theorem \ref{thm:kowX}. Combining the results from Sections \ref{sec:xi-construction} and \ref{sec:coproductK} we calculate the coproduct $\kow(\cK)$ and prove a $\tau\tau_0$-twisted version of Formula \eqref{eq:kowcK} in Section \ref{sec:deltaK}. This shows that $\cK$ is a $\tau\tau_0$-universal K-matrix in the sense of Definition \ref{def:U-cylinder-braided}.

\medskip

\noindent{\bf Acknowledgements:} The authors are grateful to Weiqiang Wang for comments and advice on referencing.

%%%%%%%%%%%%%%%%%%%%%%%%%%%%%%%%%%%%%%%%%%%%
\section{Preliminaries on quantum groups}\label{sec:prelim}
%%%%%%%%%%%%%%%%%%%%%%%%%%%%%%%%%%%%%%%%%%%%
In this section we fix notation and recall some standard results about quantum groups. We mostly follow the conventions in \cite{b-Lusztig94} and \cite{b-Jantzen96}.
%%%%%%%%%%%%%%%%%%%%%%%%%%%%%%%%%%%%%%%%%
\subsection{The root datum}\label{sec:RootDatum}
%%%%%%%%%%%%%%%%%%%%%%%%%%%%%%%%%%%%%%%%%
Let $I$ be a finite set and let $A=(a_{ij})_{i,j\in I}$ be a symmetrizable generalized Cartan matrix. By definition there exists a diagonal matrix $D=\mathrm{diag}(\epsilon_i\,|\,i\in I)$ with coprime entries $\epsilon_i\in \N$ such that the matrix $DA$ is symmetric. Let $(\mathfrak{h},\Pi,\Pi^\vee)$ be a minimal realization of $A$ as in \cite[1.1]{b-Kac1}. Here $\Pi=\{\alpha_i\,|\,i\in I\}$ and $\Pi^\vee=\{h_i\,|\,i\in I\}$ denote the set of simple roots and the set of simple coroots, respectively. We write $\mathfrak{g}=\mathfrak{g}(A)$ to denote the Kac-Moody Lie algebra corresponding to the realization $(\mathfrak{h},\Pi,\Pi^\vee)$ of $A$ as defined in \cite[1.3]{b-Kac1}.

Let $Q=\Z\Pi$  be the root lattice and define $Q^+=\mathbb{N}_0\Pi$. For $\lambda,\mu\in \hfrak^\ast$ we write $\lambda>\mu$ if $\lambda-\mu\in Q^+\setminus\{0\}$. For $\mu=\sum_i m_i \alpha_i\in Q^+$ let $\mathrm{ht}(\mu)=\sum_i m_i$ denote the height of $\mu$. For any $i\in I$ the simple reflection $\sigma_i\in \mathrm{GL}(\hfrak^\ast)$ is defined by
\begin{align*}
  \sigma_i(\alpha)= \alpha- \alpha(h_i)\alpha_i.
\end{align*}
The Weyl group $W$ is the subgroup of $\mathrm{GL}(\hfrak^\ast)$ generated by the simple reflections $\sigma_i$ for all $i\in I$. For simplicity set $r_A=|I|-\rank(A)$. Extend $\Pi^\vee$ to a basis $\Pi^\vee_{ext}=\Pi^\vee\cup\{d_s\,|\,s=1,\dots,r_A\}$ of $\hfrak$ and set $\Qvext=\Z\Pi^\vee_{ext}$. Assume additionally that $\alpha_i(d_s)\in \Z$ for all $i\in I$, $s=1,\dots,r_A$. By \cite[2.1]{b-Kac1} there exists a nondegenerate, symmetric, bilinear form $(\cdot,\cdot)$ on $\hfrak$ such that
\begin{align*}
  (h_i,h)=\alpha_i(h)/\epsilon_i \quad \forall h\in \hfrak, i\in I, \qquad (d_m,d_n)=0 \quad \forall n,m\in \{1,\dots,r_A\}.
\end{align*}
Hence, under the resulting identification of $\hfrak$ and $\hfrak^\ast$ we have $h_i=\alpha_i/\epsilon_i$. The induced bilinear form on $\hfrak^\ast$ is also denoted by the bracket $(\cdot,\cdot)$. It satisfies $(\alpha_i,\alpha_j)=\epsilon_i a_{ij}$ for all $i,j\in I$. Define the weight lattice by 
\begin{align*}
  P=\{\lambda\in \hfrak^\ast\,|\, \lambda(\Qvext)\subseteq \Z\}.
\end{align*}
%%%%%%%%%%%%%%%%%%%%%%%%%%%%%%%%%%%
\begin{rema}
  The abelian groups $Y=\Qvext$ and $X=P$ together with the embeddings $I\rightarrow Y$, $i\mapsto h_i$ and $I\rightarrow X$, $i\mapsto \alpha_i$ form an $X$-regular and $Y$-regular root datum in the sense of \cite[2.2]{b-Lusztig94}.
\end{rema}
%%%%%%%%%%%%%%%%%%%%%%%%%%%%%%%%%%%%%%%%
Define $\beta_i\in \hfrak^\ast$ by $\beta_i(h)=(d_i,h)$, set $\Pi_{ext}=\Pi\cup\{\beta_i\,|\,i=1,\dots,r_A\}$, and let $Q_{ext}=\Z \Pi_{ext}$. Then
\begin{align*}
  P^\vee=\{h\in \hfrak\,|\, Q_{ext}(h)\subseteq \Z\}
\end{align*}
is the coweight lattice. Let $\varpi_i^\vee$ for $i\in I$ denote the basis vector of $P^\vee$ dual to $\alpha_i$. Let $B$ denote the $r_A\times |I|$-matrix with entries $\alpha_j(d_i)$. Define 
an $(r_A+|I|)\times (r_A+|I|)$ matrix by 
\begin{align*}
  A_{ext}=\left(\begin{array}{cc} A & D^{-1} B^t \\ B & 0\end{array} \right).
\end{align*}
By construction one has $\det(A_{ext})\neq 0$.
The pairing $(\cdot,\cdot)$ induces $\Q$-valued pairings on $P\times P$ and $P\times P^\vee$. The above conventions lead to the following result.
%%%%%%%%%%%%%%%%%%%%%%%%%%%%%%%%%%%%%%%%
\begin{lem}\label{lem:detA}
  The pairing $(\cdot,\cdot)$ takes values in $\frac{1}{\det(A_{ext})}\Z$ on $P\times P$ and on $P\times P^\vee$. 
\end{lem}
%%%%%%%%%%%%%%%%%%%%%%%%%%%%%%%%%%%%%%%%

%%%%%%%%%%%%%%%%%%%%%%%%%%%%%%%%%%%%%%%%
  \subsection{Quantized enveloping algebras}\label{sec:QGps}
%%%%%%%%%%%%%%%%%%%%%%%%%%%%%%%%%%%%%%%%
With the above notations we are ready to introduce the quantized enveloping algebra $\uqg$. Let $d\in \N$ be the smallest positive integer such that $\frac{d}{\det(A_{ext})}\in \Z$. Let $q^{1/d}$ be an indeterminate and let $\mathbb{K}$ a field of characteristic zero. We will work with the field $\mathbb{K}(q^{1/d})$ of rational functions in $q^{1/d}$ with coefficients in $\mathbb{K}$. 
%%%%%%%%%%%%%%%%%%%%%%%%%%%%%%
\begin{rema}\label{rem:3reasons}
  The choice of ground field is dictated by two reasons. Firstly, by Lemma \ref{lem:detA} it makes sense to consider $q^{(\lambda,\mu)}$  as an element of $\field(q^{1/d})$ for any weights $\lambda, \mu\in P$. This will allow us to define the commutativity isomorphism $\rh$, see Example \ref{eg:kappa-def} and formula \eqref{eq:Rh-def}. Secondly, in the construction of the function $\xi$ in Section \ref{sec:xi-construction} we will require factors of the form $q^{\lambda(\varpi^\vee_i)}$ for $\lambda\in P$ and $i\in I$, see formula \eqref{eq:xi-def}. Again, Lemma \ref{lem:detA} shows that such factors lie in $\field(q^{1/d})$.
\end{rema}
%%%%%%%%%%%%%%%%%%%%%%%%%%%%%%%%%%%%%%%
Following \cite[3.1.1]{b-Lusztig94} the quantized enveloping algebra $\uqg$ is the associative $\mathbb{K}(q^{1/d})$-algebra generated by elements $E_i$, $F_i$, $K_h$ for all $i\in I$ and $h\in \Qvext$ satisfying the following defining relations:
\begin{enumerate}
  \item $K_0=1$ and $K_h K_{h'}=K_{h+h'}$ for all $h, h'\in \Qvext$.
  \item $K_h E_i= q^{\alpha_i(h)} E_i K_h$ for all $i\in I$, $h\in \Qvext$.
  \item $K_h F_i= q^{-\alpha_i(h)} F_i K_h$ for all $i\in I$, $h\in \Qvext$.
  \item $E_i F_i - F_i E_i=\delta_{ij} \frac{K_i-K_i^{-1}}{q_i-q_i^{-1}}$ for all $i\in I$ where $q_i=q^{\epsilon_i}$ and $K_i=K_{\epsilon_i h_i}$.
  \item \label{q-Serre} the quantum Serre relations given in \cite[3.1.1.(e)]{b-Lusztig94}. 
\end{enumerate}
We will use the notation $q_i=q^{\epsilon_i}$ and $K_i=K_{\epsilon_i h_i}$ all through this text. Moreover, for $\mu=\sum_{i\in I}n_i\alpha_i\in Q$ we will use the notation
\begin{align}\label{eq:Kmu}
  K_\mu=\prod_{i\in I} K_i^{n_i}.
\end{align}
We make the quantum-Serre relations (\ref{q-Serre}) more explicit. Let $\left[\begin{matrix}1-a_{ij}\\n\end{matrix}\right]_{q_i}$ denote the $q_i$-binomial coefficient defined in \cite[1.3.3]{b-Lusztig94}. For any $i,j\in I$ define a non-commutative polynomial $S_{ij}$ in two variables by
\begin{align}\label{eq:Fij-def}
  S_{ij}(x,y)=\sum_{n=0}^{1-a_{ij}}(-1)^n
  \left[\begin{matrix}1-a_{ij}\\n\end{matrix}\right]_{q_i}
  x^{1-a_{ij}-n}yx^n.
\end{align}
By \cite[33.1.5]{b-Lusztig94} the quantum Serre relations can be written in the form
\begin{align}\label{eq:q-Serre}
  S_{ij}(E_i,E_j)=S_{ij}(F_i,F_j)=0\qquad \mbox{for all $i,j\in I$}.
\end{align}
The algebra $\uqg$ is a Hopf algebra with coproduct $\kow$, counit $\vep$, and antipode $S$ given by
\begin{align}
  \kow(E_i)&=E_i \ot 1 + K_i\ot E_i,& \vep(E_i)&=0, & S(E_i)&=-K_i^{-1} E_i,\label{eq:E-copr}\\
  \kow(F_i)&=F_i\ot K_i^{-1} + 1 \ot F_i,&\vep(F_i)&=0, & S(F_i)&=-F_i K_i,\label{eq:F-copr}\\
  \kow(K_h)&=K_h \ot K_h,& \vep(K_h)&=1, & S(K_h)&=K_{-h}\label{eq:K-copr}
\end{align}
for all $i\in I$, $h\in \Qvext$. We denote by $\uqgp$ the Hopf subalgebra of $\uqg$ generated by the elements $E_i, F_i$, and $K_i^{\pm 1}$ for all $i\in I$.  Moreover, for any $i\in I$ let $U_{q_i}(\mathfrak{sl}_2)_i$ be the subalgebra of $\uqg$ generated by $E_i, F_i, K_i$ and $K_i^{-1}$. The Hopf algebra $U_{q_i}(\mathfrak{sl}_2)_i$ is isomorphic to $U_{q_i}(\mathfrak{sl}_2)$ up to the choice of the ground field.

As usual we write $U^+$, $U^-$, and $U^0$ to denote the $\field(q^{1/d})$-subalgebras of $\uqg$ generated by $\{E_i\,|\,i\in I\}$, $\{F_i\,|\,i\in I\}$, and $\{K_h\,|\,h\in \Qvext\}$, respectively. We also use the notation $U^{\ge}=U^+U^0$ and $U^\le=U^-U^0$ for the positive and negative Borel part of $\uqg$. For any $U^0$-module $M$ and any $\lambda\in P$ let
\begin{align*}
  M_{\lambda}=\{m\in M | K_h\lact m=q^{\lambda(h)}m \,\, \textrm{ for all } h\in \Qvext \}
\end{align*}  
denote the corresponding weight space. We can apply this notation in particular to $U^+$, $U^-$, and $\uqg$ which are $U^0$-modules with respect to the left adjoint action. We obtain algebra gradings
\begin{align}\label{eq:U-grading}
  U^+=\bigoplus_{\mu \in Q^+}U^+_{\mu}, \quad U^-=\bigoplus_{\mu \in Q^+}U^-_{-\mu}, \quad \uqg=\bigoplus_{\mu\in Q} \uqg_\mu. 
\end{align}
%%%%%%%%%%%%%%%%%%%%%%%%%%%%%%%%%%%%%%%%%%%%%%%%%%%%%%%%%%%%%%
\subsection{The bilinear pairing $\left<\cdot,\cdot \right>$}
%%%%%%%%%%%%%%%%%%%%%%%%%%%%%%%%%%%%%%%%%%%%%%%%%%%%%%%%%%%%%%
Let $k$ be any field, let $A$ and $B$ be $k$-algebras, and let $\langle\cdot,\cdot \rangle:A\times B \rightarrow k$ be a bilinear pairing. Then $\langle\cdot,\cdot\rangle$ can be extended to $A^{\otimes n} \times B^{\otimes n}$ by setting
\begin{align*}
  \langle\otimes_i a_i,\otimes_i b_i \rangle=\prod_{i=1}^n\langle a_i ,b_i\rangle.
\end{align*}
In the following we will use this convention for $k=\field(q^{1/d})$, $A=U^\le$, $B=U^\ge$,  and $n=2$ and $3$ without further remark. 

There exists a unique $\mathbb{K}(q^{1/d})$-bilinear pairing 
\begin{align}\label{eq:form}
  \left<\cdot,\cdot \right>:U^{\le}\times U^{\ge}\to \mathbb{K}(q^{1/d})
\end{align}
such that for all $x,x' \in U^{\ge}$, $y,y' \in U^{\le}$, $g,h\in \Qvext$, and $i,j\in I$ the following relations hold
\begin{align} 
\left<y, xx' \right>&=\left<\Delta(y), x'\otimes x \right>, &  \left<yy', x \right>&=\left<y\otimes y', \Delta(x)\right>,  \label{eq:formdef} \\
\left<K_g,K_h \right>&=q^{-(g,h)}, &  \left<F_i, E_j\right>&=\delta_{ij} \frac{-1}{q_i-q_i^{-1}},  \\
\left<K_h,E_i\right>&=0,  &  \left<F_i, K_h\right>&=0.  \label{eq:formdef2}
\end{align}
Here we follow the conventions of \cite[6.12]{b-Jantzen96} in the finite case. In the Kac-Moody case the existence of the pairing $\langle\cdot,\cdot\rangle$ follows from the results in \cite[Chapter 1]{b-Lusztig94}. Relations \eqref{eq:formdef}--\eqref{eq:formdef2} imply that for all $x\in U^+, y\in U^-$, and $g,h\in \Qvext$ one has
\begin{align} 
\left<yK_g,xK_h \right>&=q^{-(g,h)}\left<y ,x\right>. \label{eq:formKout} 
\end{align}
The pairing $\langle\cdot,\cdot\rangle$ respects weights in the following sense. For $\mu,\nu \in Q^+$ with $\mu\neq \nu$ the restriction of the pairing to $U^-_{-\nu}\times U^+_\mu$ vanishes identically. On the other hand, the restriction of the paring to $U^-_{-\mu}\times U^+_\mu$ is nondegenerate for all $\mu\in Q^+$. The nondegeneracy of this restriction implies the following lemma, which we will need in the proof of Theorem \ref{thm:kowX}.
%%%%%%%%%%%%%%%%%%%%%%%%%%%%%%%%%%%%%%%%%%%5
\begin{lem}\label{lem:yzXX'}
  Let $X,X' \in \prod_{\mu\in Q^+} U^+ K_\mu\ot U^+_\mu$. If 
  \begin{align*}
    \left<y\ot z, X\right> = \left<y\ot z,X'\right> \qquad \mbox{for all $y,z\in U^-$}
  \end{align*} 
  then $X=X'$.
\end{lem}
%%%%%%%%%%%%%%%%%%%%%%%%%%%%%%%%%%%%%%%%%%%
\begin{proof}
We may assume that $X'=0$. Write $X=\sum_{\mu\in Q^+}X_\mu$, with 
\begin{align*}
  X_\mu=\sum_i X_{\mu,i}^{(1)}K_\mu\otimes X_{\mu,i}^{(2)}\in U^+ K_\mu\ot U^+_\mu.
\end{align*}
Consider $\tilde{X}_\mu=\sum_i X_{\mu,i}^{(1)}\otimes X_{\mu,i}^{(2)}\in U^+ \ot U^+_\mu$. For any $y\in U^-$, $\mu \in Q^+$, and $z\in U^-_\mu$ we then have
\begin{align*}
  0=\left<y\ot z, X \right>=\left<y\ot z, X_\mu \right>\stackrel{\eqref{eq:formKout}}{=}\left<y\ot z, X_\mu' \right>.
\end{align*}
By the nondegeneracy of the pairing on $U^+_\mu\times U^-_\mu$ it follows that $ X_\mu'=0$. Consequently $ X_\mu=0$ for all $\mu\in Q^+$, and hence $X=0$ as claimed. 
\end{proof}
%%%%%%%%%%%%%%%%%%%%%%%%%%%%%%%%%%%%%%%%%%%%%%%%%%%%%%%%%%%%%%

%%%%%%%%%%%%%%%%%%%%%%%%%%%%%%%%%%%%%%%%%%%%%%%%%%%%%%%%%%%%%%
\subsection{Lusztig's skew derivations $r_i$ and ${}_ir$}\label{sec:defri}
%%%%%%%%%%%%%%%%%%%%%%%%%%%%%%%%%%%%%%%%%%%%%%%%%%%%%%%%%%%%%%
Let $\free$ be the free associative $\field(q^{1/d})$-algebra generated by elements $\ccf_i$ for all $i\in I$. The algebra $\free$ is a $U^0$-module algebra with $K_h \lact{\ccf}_i= q^{\alpha_i(h)} \ccf_i$. As in \eqref{eq:U-grading} one obtains a $Q^+$-grading
\begin{align*}
  \free=\bigoplus_{\mu\in Q^+} \free_\mu.
\end{align*}
The natural projection $\free\rightarrow U^+$, $\ccf_i\mapsto E_i$ respects the $Q^+$-grading.
There exist uniquely determined $\field(q^{1/d})$-linear maps $\ir, \ri:\free\rightarrow \free$ such that 
\begin{align}
r_i(\ccf_j)= \delta_{ij}, \qquad & r_i(xy)=q^{(\alpha_i,\nu)}r_i(x)y+xr_i(y), \label{def:ri} \\
{}_ir(\ccf_j)= \delta_{ij}, \qquad & {}_ir(xy)={}_ir(x)y+q^{(\alpha_i,\mu)}x\, {}_ir(y) \label{def:ir} 
\end{align}
for any $x\in \free_\mu$ and $y\in \free_\nu$. The above equations imply in particular that $\ir(1)=0=\ri(1)$. By \cite[1.2.13]{b-Lusztig94} the maps $\ir$ and $\ri$ factor over $U^+$, that is there exist linear maps $\ir, \ri:U^+\rightarrow U^+$, denoted by the same symbols, which satisfy
\eqref{def:ri} and \eqref{def:ir} for all $x\in U^+_\mu$, $y\in U^+_\nu$ and with $\ccf_j$ replaced by $E_j$. The maps $\ri$ and $\ir$ on $U^+$ satisfy the following three properties, each of which is equivalent to the definition given above.

\noindent {\bf (1)} For all $x\in U^+$ and all $i\in I$ one has
     \begin{equation}
        [x,F_i]=\frac{1}{(q_i-q_i^{-1})}\left(r_i(x)K_i -K_i^{-1}\, {}_ir(x) \right), \label{eq:riCommute}
     \end{equation}
     see \cite[Proposition 3.1.6]{b-Lusztig94}.
     
\noindent {\bf (2)} For all $x\in U^+_\mu$ one has 
\begin{align}
\Delta(x)&=x\otimes 1+ \sum_{i}r_i(x)K_i\otimes E_i + \mathrm{(rest)_1}, \label{eq:riDelta}\\
\Delta(x)&=K_\mu \otimes x+ \sum_{i}E_i K_{\mu-\alpha_i}\otimes {}_ir(x) + \mathrm{(rest)_2} \notag
\end{align}
 where $\mathrm{(rest)_1}\in \sum_{ \alpha \notin \Pi \cup \{0\}} U^+_{\mu-\alpha}K_\alpha\otimes U^+_\alpha$, and $\mathrm{(rest)_2}\in \sum_{ \alpha \notin \Pi \cup \{0\}}  U^+_{\alpha}K_{\mu-\alpha}\otimes U^+_{\mu-\alpha}$, see \cite[Section 6.14]{b-Jantzen96}.

\noindent{\bf (3)} For all $x\in U^+$, $y\in U^-$, and $i\in I$ one has
    \begin{equation} \label{eq:form-ri}
      \left<F_iy,x\right>=\left<F_i,E_i\right>\left<y,{}_ir(x)\right>,  \qquad
			\left<yF_i,x\right> =\left<F_i,E_i\right>\left<y,r_i(x)\right>
		\end{equation}
see \cite[1.2.13]{b-Lusztig94}

Property (3), and the original definition of $\ri$ and $\ir$ as skew derivations, are useful in inductive arguments. Properties (1) and (2), on the other hand, carry information about the algebra and the coalgebra structure of $\uqg$, respectively.

Property (3) above and the nondegeneracy of the pairing $\langle\cdot,\cdot\rangle$ imply that for any $x\in U^+_\mu$ with $\mu\in Q^+\setminus \{0\}$ one has 
\begin{equation}
x=0 \quad \Leftrightarrow \quad r_i(x)=0 \textrm{ for all } i\in I \quad \Leftrightarrow \quad {}_ir(x)=0 \textrm{ for all } i\in I, \label{eq:allrizero}
\end{equation}
see also \cite[Lemma 1.2.15]{b-Lusztig94}. Moreover, property (2) and the coassociativity of the coproduct imply that for any $i,j\in I$ one has
\begin{equation}r_i\circ {}_jr={}_jr \circ  r_i, \label{eq:rijr}\end{equation}
see \cite[Lemma 10.1]{b-Jantzen96}. Note that this includes the case $i=j$.

Similarly to the situation for $U^+$, the maps $\ri, \ir:\free\rightarrow \free$ also factor over the canonical projection $\free \rightarrow U^-$, $\ccf_i\rightarrow F_i$ which maps $\free_\mu$ to $U^-_{-\mu}$ for all $\mu\in Q^+$. The maps $\ri, \ir: U^- \rightarrow U^-$ satisfy \eqref{def:ri} and \eqref{def:ir} for all $x\in U^-_{-\mu}$, $y\in U^-_{-\nu}$ with $\ccf_j$ replaced by $F_j$. 
Moreover, the maps $\ri, \ir: U^- \rightarrow U^-$ can be equivalently described by analogs of the properties (1)-(3) above. For example, in analogy to (3) one has
\begin{equation}\label{eq:form-riE}
  		\left<y,E_ix\right>=\left<F_i,E_i\right>\left<{}_ir(y),x\right>,\qquad
			\left<y,xE_i\right>=\left<F_i,E_i\right>\left<r_i(y),x\right>
\end{equation}
for all $x\in U^-$, $y\in U^-$, and $i\in I$.

As in \cite[3.1.3]{b-Lusztig94} let $\sigma:\uqg\rightarrow \uqg$ denote the $\field(q^{1/d})$-algebra antiautomorphism determined by 
\begin{align}
 \sigma(E_i)&=E_i, & \sigma(F_i)&=F_i,& \sigma(K_h)&=K_{-h} \qquad \mbox{for all $i\in I$, $h\in \Qvext$}. \label{def:sigma}
\end{align}
The map $\sigma$ intertwines the skew derivations $r_i$ and ${}_i r$ as follows
\begin{align}
  \sigma \circ r_i = {}_i r \circ \sigma\qquad \mbox{for all $i\in I$}. \label{eq:sigmari}
\end{align}
Recall that the bar involution on $\uqg$ is the $\mathbb{K}$-algebra automorphism
\begin{align*}
  \overline{\phantom{bl}}:\uqg\rightarrow \uqg,\quad x\mapsto \overline{x}
\end{align*}
defined by
\begin{align}
  \overline{q^{1/d}}&=q^{-1/d}, & \overline{E_i}&=E_i, & \overline{F_i}&=F_i,& \overline{K_h}&=K_{-h}\label{eq:defbarUqg}
\end{align}
for all $i\in I$, $h\in Q^\vee_{ext}$.
The bar involution on $\uqg$ also intertwines the skew derivations $\ri$ and $\ir$ in the sense that
\begin{equation}
  {}_ir(\overline{x})=q^{(\alpha_i,\mu-\alpha_i)}\overline{r_i(x)}\qquad \mbox{for all $x\in U^+_\mu$, $\mu\in Q^+$,}  \label{eq:ribar}
\end{equation}
see \cite[Lemma 1.2.14]{b-Lusztig94}.
%%%%%%%%%%%%%%%%%%%%%%%%%%%%%%%%%%%%%%%%
\section{The completion $\mathscr{U}$ of $\uqg$}\label{sec:completion}
%%%%%%%%%%%%%%%%%%%%%%%%%%%%%%%%%%%%%%%%
It is natural to consider completions of the infinite dimensional algebra $\uqg$ and related algebras. The quasi R-matrix for $\uqg$, for example, lies in a completion of $U^-\otimes U^+$, and the universal R-matrix lies in a completion of $U^\le\ot U^\ge$, see Section \ref{sec:Rmatrix}. Similarly, the universal K-matrix we construct in this paper lies in a completion $\scrU$ of $\uqg$. This completion is commonly used in the literature, see for example \cite[1.3]{a-Saito94}. Here, for the convenience of the reader, we recall the construction and properties of the completion $\scrU$ in quite some detail. This allows us to introduce further concepts, such as the Lusztig automorphisms, as elements of $\scrU$. It also provides a more conceptual view on the quasi R-matrix and the commutativity isomorphisms.  
%%%%%%%%%%%%%%%%%%%%%%%%%%%%%%%%%%%%%%%%
\subsection{The algebra $\mathscr{U}$}
%%%%%%%%%%%%%%%%%%%%%%%%%%%%%%%%%%%%%%%%
Let $\cO_{int}$ denote the category of integrable $\uqg$-modules in category $\cO$.
Recall that category $\cO$ consists of all finitely generated $\uqg$-modules $M$ which decompose into weight spaces $M=\oplus_{\lambda\in P} M_\lambda$ and on which the action of $U^+$ is locally finite. Objects in $\cO_{int}$ are additionally locally finite with respect to the action of $\uqislzi$ for all $i\in I$. Simple objects in $\cO_{int}$ are irreducible highest weight modules with dominant integral highest weight \cite[Corollary 6.2.3]{b-Lusztig94}. If $\gfrak$ is of finite type then $\cO_{int}$ is the category of finite dimensional type 1 representations. 

Let $\cV$ be the category of vector spaces over $\mathbb{K}(q^{1/d})$. Both $\cV$ and $\mathcal{O}_{int}$  are tensor categories, and the forgetful functor
\begin{align*}
 \For:\mathcal{O}_{int} \rightarrow \cV
\end{align*}
is a tensor functor. Let $\mathscr{U}=\End(\For)$ be the set of natural transformations from $\For$ to itself. The category $\mathcal{O}_{int}$ is equivalent to a small category and hence $\mathscr{U}$ is indeed a set. More explicitly, elements of $\sU$ are families $(\varphi_M)_{M\in Ob(\mathcal{O}_{int})}$ of vector space endomorphisms $\varphi_M:\For(M)\rightarrow \For(M)$ such that the diagram
\begin{align*}
\xymatrix{
  \For(M) \ar[r]^{\varphi_M} \ar[d]^{\For(\psi)}& \For(M)\ar[d]^{\For(\psi)}\\
  \For (N) \ar[r]^{\varphi_N}& \For (N)
}
\end{align*}
commutes for any $\uqg$-module homomorphism $\psi:M\rightarrow N$. Natural transformations of $\For$ can be added and multiplied by a scalar, both operations coming from the linear structure on vector spaces. Composition of natural transformations gives a multiplication on $\sU$ which turns $\sU$ into a $\mathbb{K}(q^{1/d})$-algebra. 
%%%%%%%%%%%%%%%%%%%%%%%%%%%%%%%%%%%%%%%%%%%%%%%
\begin{eg}
The action of $\uqg$ on objects of $\mathcal{O}_{int}$ gives an algebra homomorphism $\uqg \rightarrow \sU$ which is injective, see \cite[Proposition 3.5.4]{b-Lusztig94} and \cite[5.11]{b-Jantzen96}. We always consider $\uqg$ a subalgebra of $\sU$.
\end{eg}
%%%%%%%%%%%%%%%%%%%%%%%%%%%%%%%%%%%%%%%%%%%%%%%%
\begin{eg}
Let $\widehat{U^+}=\prod_{\mu \in Q^+}U^+_\mu$ and let $(X_\mu)_{\mu \in Q^+}\in \widehat{U^+}$. 
Let $M\in Ob(\mathcal{O}_{int})$ and $m\in M$. As the action of $U^+$ on $M$ is locally finite there exist only finitely many $\mu\in Q^+$ such that $X_\mu m\neq 0$. Hence the expression 
\begin{align}\label{eq:X-mu-m}
  \sum_{\mu\in Q^+}X_\mu m
\end{align}
is well defined. In this way the element $(X_\mu)_{\mu \in Q^+}\in \widehat{U^+}$ defines an endomorphism of $\For$, and we may thus consider $\widehat{U^+}$ as a subalgebra of $\sU$. We sometimes write elements of $\widehat{U^+}$ additively as $X=\sum_{\mu\in Q^+}X_\mu$. In view of \eqref{eq:X-mu-m}  this is compatible with the inclusions $U^+\subseteq \widehat{U^+} \subseteq \sU$.
\end{eg}
%%%%%%%%%%%%%%%%%%%%%%%%%%%%%%%%%%%%%%%%%%%%%%%%
\begin{eg}\label{eg:xiinU}
  Let $\xi:P\rightarrow \mathbb{K}(q^{1/d})$ be any map. For $M\in Ob(\mathcal{O}_{int})$ define a linear map $\xi_M:M\rightarrow M$ by $\xi_M(m)=\xi(\lambda) m$ for all $m\in M_\lambda$. Then the family $(\xi_M)_{M\in Ob(\mathcal{O}_{int})}$ is an element in $\sU$. By slight abuse of notation we denote this element by $\xi$ as well.
\end{eg}
%%%%%%%%%%%%%%%%%%%%%%%%%%%%%%%%%%%%%%%%%%%%%%%%5
Lusztig showed that $\sU$ contains a homomorphic image of the braid group corresponding to $W$. 
For any $M\in Ob(\mathcal{O}_{int})$ and any $i\in I$ the Lusztig automorphism ${T_i}_M:M\rightarrow M$ is defined on $m\in M_\lambda$ with $\lambda\in P$ by
\begin{align}\label{eq:TiV-def}
  {T_i}_M(m) = \sum_{\substack{a,b,c\ge0 \\a-b+c=\lambda(h_i)}} (-1)^b q_i^{b-ac} E_i^{(a)}F_i^{(b)}E_i^{(c)}m.
\end{align} 
The family $T_i=({T_i}_M)$ defines an element in $\sU$. By \cite[Proposition 5.2.3]{b-Lusztig94} the elements $T_i$ of $\sU$ are invertible with inverse $T_i^{-1}=({T_i}_M^{-1})$ given by
\begin{align*}
  {T_i}_M^{-1}(m) = \sum_{\substack{a,b,c\ge0 \\ a-b+c=\lambda(h_i)}}(-1)^b q_i^{ac-b} F_i^{(a)}E_i^{(b)}F_i^{(c)}m
\end{align*} 
By \cite[39.43]{b-Lusztig94} the elements $T_i$ for $i\in I$ satisfy the braid relations
\begin{align*}
  \underbrace{T_i T_j T_i \dots}_{\mbox{\small{$m_{ij}$ factors}}}= \underbrace{T_j T_i T_j \dots}_{\mbox{\small{$m_{ij}$ factors}}}
\end{align*}
where $m_{ij}$ denotes the order of $\sigma_i\sigma_j\in W$. Hence, for any $w\in W$ there is a well defined element $T_w\in \sU$ given by 
\begin{align*}
  T_w=T_{i_1} T_{i_2}\dots T_{i_k}
\end{align*}
if $w=\sigma_{i_1}\sigma_{i_2}\dots \sigma_{i_k}$ is a reduced expression.

We also use the symbol $T_i$ for $i\in I$ to denote the corresponding algebra automorphism of $\uqg$ denoted by $T_{i,1}''$ in \cite[37.1]{b-Lusztig94}. This is consistent with the above notation, in the sense that for any $u\in \uqg$, any $M\in Ob(\mathcal{O}_{int})$, and any $m\in M$ we have $T_{iM}(um)=T_i(u) T_{iM}(m)$. Hence $T_i$, as an automorphism of $\uqg$, is nothing but conjugation by the invertible element $T_i\in \scrU$. In this way we obtain automorphisms $T_{w}$ of $\uqg$ for all $w\in W$.

Furthermore, the bar involution for $\uqg$ intertwines $T_i$ and $T_i^{-1}$. More explicitly, for $u\in U^+_\mu$ one has
\begin{equation}
T_{i}(u)=(-1)^{\mu(h_i)}q^{(\mu,\alpha_i)}\overline{T_i^{-1}(\overline{u})},\label{eq:Tibarinverse}
\end{equation}
see \cite[37.2.4]{b-Lusztig94}
%%%%%%%%%%%%%%%%%%%%%%%%%%%%%%%%%%%%%%%%%%
\subsection{The coproduct on $\sU$}\label{sec:U-coproduct}
%%%%%%%%%%%%%%%%%%%%%%%%%%%%%%%%%%%%%%%%%%
To define a coproduct on $\sU$ consider the functor
\begin{align*}
  \Ftwo:\mathcal{O}_{int}\times \mathcal{O}_{int} \rightarrow \cV, \qquad (M,N)&\mapsto \For(M\otimes N)=\For(M)\otimes \For (N),\\ (f,g)&\mapsto \For(f\otimes g). 
\end{align*}
Let $\sU^{(2)}_0=\End(\Ftwo)$ denote the set of natural transformations from $\Ftwo$ to itself. Again, $\sU^{(2)}_0$ is an algebra for which the multiplication $\cdot$ is given by composition of natural transformations. The map
  \begin{align*}
    i^{(2)}:\sU\otimes \sU \rightarrow \sU^{(2)}_0, \qquad (\varphi_M)\otimes (\psi_N)\mapsto (\varphi_M\otimes \psi_N)
  \end{align*}
is an injective algebra homomorphism. However, it is not surjective, as the following example shows. 
\begin{eg}\label{eg:kappa-def}
For $M,N\in Ob(\mathcal{O}_{int})$ define a linear map
  \begin{align*}
    \kappa_{M,N}: M\ot N \rightarrow M\ot N, \quad m\ot n\mapsto q^{(\mu,\nu)}m\ot n \quad \mbox{if $m\in M_\mu$ and $n\in N_\nu$.}
  \end{align*}
The collection $\kappa=(\kappa_{M,N})_{M,N\in Ob(\mathcal{O}_{int})}$ lies in $\sU^{(2)}_0$. However, one can show that $\kappa$ is not of the form $\sum_{k=1}^n f_i\otimes g_i$ for any $n\in \mathbb{N}$ and any collection $f_i,g_i \in \sU$. Hence $\kappa$ does not lie in the image of the map $i^{(2)}$ described above.

The element $\kappa$ is an important building block of the universal R-matrix for $\uqg$, see Section \ref{sec:Rmatrix}. For $\kappa$ to be well defined the ground field needs to contain $q^{(\mu,\nu)}$ for all $\mu,\nu \in P$. This gives one of the reasons why we work over the field $\field(q^{1/d})$.
\end{eg}
%%%%%%%%%%%%%%%%%%%%%%%%%%%%%%%%%%%%%%%%%%%%%%%%%%%%
Any natural transformation $\varphi \in\sU$ can be restricted to all $\For(M\otimes N)$ for $M,N\in Ob(\mathcal{O}_{int})$. Moreover, restriction is compatible with composition and linear combinations of natural transformations. Hence we obtain an algebra homomorphism
\begin{align*}
  \kow_{\sU}: \sU \rightarrow \sU^{(2)}_0, \qquad (\varphi_M)_{M\in Ob(\mathcal{O}_{int})}\mapsto (\varphi_{M\otimes N})_{M,N\in Ob(\mathcal{O}_{int})}. 
\end{align*}  
We call $\kow_{\sU}$ the coproduct of $\sU$. The restriction of $\kow_\sU$ to $\uqg$ coincides with the coproduct of $\uqg$ from Section \ref{sec:QGps}. For this reason we will drop the subscript $\sU$ and just denote the coproduct on $\sU$ by $\kow$.

We would also like to consider families of linear maps flipping the two tensor factors by a similar formalism. To that end consider the functor
\begin{align*}
 \Ftwoop:\mathcal{O}_{int}\times \mathcal{O}_{int} \rightarrow \cV, \qquad (M,N)&\mapsto \For(N\otimes M)=\For (N)\otimes \For(M),\\ (f,g)&\mapsto \For(g\otimes f). 
\end{align*}
Define $\sU^{(2)}_1= \Hom(\Ftwo,\Ftwoop)$. For $M,N\in Ob(\mathcal{O}_{int})$ let 
\begin{align*}
  \flip_{M,N}:M\otimes N\rightarrow N\otimes M, \qquad m\otimes n \mapsto n\otimes m
\end{align*}
denote the flip of tensor factors. Then $\flip=(\flip_{M,N})_{M,N\in Ob(\mathcal{O}_{int})}$ is an element of $\sU^{(2)}_1$. The direct sum
\begin{align*}
  \sU^{(2)}=\sU^{(2)}_0\oplus \sU^{(2)}_1 
\end{align*}
is a $\Z_2$-graded algebra where multiplication $\cdot$ is given by composition of natural transformations. This is the natural algebra for the definition of the commutativity isomorphisms in the next subsection.
%%%%%%%%%%%%%%%%%%%%%%%%%%%%%%%%%%%%%%%%%%%%%%%%%%%%%%%%
\subsection{Quasi R-matrix and commutativity isomorphisms}\label{sec:Rmatrix}
%%%%%%%%%%%%%%%%%%%%%%%%%%%%%%%%%%%%%%%%%%%%%%%%%%%%%%%%
Let $\mu\in Q^+$ and let $\{b_{\mu,i}\}$ be a basis of $U^-_{-\mu}$. Let $\{{b_{\mu}}^i\}$ be the dual basis of $U^+_\mu$ with respect to the pairing \eqref{eq:form}. Define 
\begin{align*}
  R_\mu=\sum_i b_{\mu,i}\otimes {b_\mu}^i \in U^-\ot U^+.
\end{align*}   
The element $R_\mu$ is independent of the chosen basis $\{b_{\mu,i}\}$. The quasi R-matrix
\begin{align}
R=\sum_{\mu\in Q^+} R_\mu \label{eq:R-def}
\end{align}
gives a well-defined element $\sU^{(2)}_0$. Indeed, for $M,N\in Ob(\mathcal{O}_{int})$ only finitely many summands $R_\mu$ act nontrivially on any element of $M\otimes N$. 
%%%%%%%%%%%%%%%%%%%%%%%%%%%%%%%%%%%%%%%%%5
\begin{rema}
The element $R\in \sU^{(2)}_0$ coincides with the quasi-R-matrix defined in \cite[4.1.4]{b-Lusztig94} and in \cite[7.2]{b-Jantzen96} in the finite case. Those references use the symbol $\Theta$ for the quasi-R-matrix, but we change notation to avoid confusion with the involutive automorphism $\Theta:\hfrak^\ast\rightarrow \hfrak^\ast$ defined in Section \ref{sec:SecondKind}.
\end{rema}
%%%%%%%%%%%%%%%%%%%%%%%%%%%%%%%%%%%%%%%%%%
The quasi R-matrix has a second characterization in terms of the bar involution \eqref{eq:defbarUqg} of $\uqg$. Define a bar involution $\overline{\phantom{b}}$ on $\uqg\otimes \uqg$ by $\overline{u\otimes v}=\overline{u}\ot\overline{v}$. By \cite[Theorem 4.1.2]{b-Lusztig94} the quasi-R-matrix is the uniquely determined element $ R=\sum_{\mu\in Q^+}R_\mu\in\prod_{\mu\in Q^+} U^-_{-\mu}\ot U^+_\mu$ with $R_\mu\in U^-_{-\mu}\ot U^+_\mu$ and $R_0=1\ot 1$ for which
\begin{align}\label{eq:quasiR-property1}
\Delta(\overline{u})R=R\overline{\Delta(u)} \qquad \textrm{ for all } u \in \uqg.
\end{align}
Moreover, $R$ is invertible, with 
\begin{align}\label{eq:Rinv-Rbar}
  R^{-1}=\overline{R}.
\end{align}
%%%%%%%%%%%%%%%%%%%%%%%%%%%%%%%%%%%%%%%%%%%%%%%%%
\begin{rema}\label{rem:R-finite}
If $\gfrak$ is of finite type then the quasi-R-matrix $R$ can be factorized into a product of R-matrices for $\slfrak_2$. Choose a reduced expression $w_0=\sigma_{i_1}\dots \sigma_{i_t}$ for the longest element $w_0$ of $W$. For $j=1,\ldots, t$ set $\gamma_j=\sigma_{i_1}\sigma_{i_2}\dots \sigma_{i_{j-1}}(\alpha_{i_j})$ and define
\begin{equation}
E_{\gamma_j}=T_{i_1}T_{i_2}\dots T_{i_{j-1}}(E_{i_j}), \qquad F_{\gamma_j}=T_{i_1}T_{i_2}\dots T_{i_{j-1}}(F_{i_j}).\label{eq:rootvectors}
\end{equation}
Then $\{\gamma_1,\ldots, \gamma_t\}$ is the set of positive roots of $\gfrak$, and \eqref{eq:rootvectors} are the root vectors used in the construction of the PBW basis corresponding to the chosen reduced expression for $w_0$. For $j=1,\dots,t$ define 
\begin{align}\label{eq:Rgammaj}
  R^{[j]}=\sum_{r\ge 0}(-1)^r q_{i_j}^{-r(r-1)/2} \frac{(q_{i_j}-q_{i_j}^{-1})^r}{[r]_{q_{i_j}}!} F_{\gamma_j}^r\ot E_{\gamma_j}^r 
\end{align}
and for $i\in I$ set $R_i=R^{[j]}$ if $\gamma_{j}=\alpha_i$.
By \cite[Remark 8.29]{b-Jantzen96} one has
\begin{align}\label{eq:R-factorization}
  R=R^{[t]}\cdot R^{[t-1]}\cdot \dots \cdot R^{[2]}\cdot R^{[1]}.
\end{align}
\end{rema}
%%%%%%%%%%%%%%%%%%%%%%%%%%%%%%%%%%%%%%%%%%%%%%%%%%%%%%  
The quasi-R-matrix $R$ and the transformation $\kappa$ defined in Example \ref{eg:kappa-def} give rise to a family of commutativity isomorphisms. Define
\begin{align}\label{eq:Rh-def}
  \rh= R \cdot \kappa^{-1} \cdot \flip
\end{align}
in $\sU^{(2)}$. By \cite[Theorem 32.1.5]{b-Lusztig94} the maps 
\begin{align}\label{eq:RMN}
  \rh_{M,N}:M\otimes N\rightarrow N\otimes M
\end{align}
are isomorphisms of $\uqg$-modules for all $M,N\in Ob(\mathcal{O}_{int})$. Moreover, the isomorphisms $\rh_{M,N}$ satisfy the hexagon property
\begin{align*}
  \rh_{M,N\ot N'} = (\id_N \ot \rh_{M,N'})(\rh_{M,N} \ot \id_{N'}),\\
  \rh_{M\ot M',N} = (\rh_{M,N} \ot \id_{M'})(\id_M \ot \rh_{M',N'})
\end{align*}
for all $M,M',N,N'\in Ob(\Oint)$, see \cite[32.2]{b-Lusztig94}. This implies that $\Oint$ is a braided tensor category as defined for example in \cite[XIII.1.1]{b-Kassel1}.
%%%%%%%%%%%%%%%%%%%%%%%%%%%%%%%%%%%%%%%%%5
\begin{rema}
  In the construction of the commutativity isomorphisms $\rh_{M,N}$ in \cite[Chapter 32]{b-Lusztig94} it is assumed that $\gfrak$ is of finite type. Moreover, Lusztig defines the commutativity isomorphisms on tensor products of integrable weight modules. Lusztig's arguments extend to the Kac-Moody case if one restricts to category $\cO$. We retain the assumption of integrability so that the Lusztig automorphisms ${T_i}_M$ given by \eqref{eq:TiV-def} are well defined. The restrictions imposed by $\rh_{M,N}$ and ${T_i}_M$ force us to work with the category $\Oint$.
\end{rema}
%%%%%%%%%%%%%%%%%%%%%%%%%%%%%%%%%%%%%%%%%%
It follows from the definitions of the completion $\sU$ and the coproduct $\Delta:\sU\to \sU^{(2)}$ that in $\sU^{(2)}$ one has
\begin{align}
 \rh \cdot \Delta(u)=\Delta(u)\cdot  \rh \qquad \textrm{ for all } u \in \sU. \label{eq:RcommuteswithDeltau}
\end{align}
In the proof of the next Lemma we will use this property for $u=T_i$.
Moreover, by \cite[Proposition 5.3.4]{b-Lusztig94} the Lusztig automorphisms $T_i\in \sU$ satisfy
\begin{align}\label{eq:kowTi}
  \kow(T_i)=T_i\otimes T_i \cdot R_i^{-1}
\end{align}
where $R_i$ was defined just below \eqref{eq:Rgammaj}. To generalize the above formula we recall the following well known lemma, see for example \cite[Proposition 8.3.11]{b-CP94}. We include a proof to assure that we have the correct formula in our conventions. Recall that for $\gfrak$ of finite type $w_0\in W$ denotes the longest element. Define, moreover, $R_{21}=\flip\cdot R\cdot\flip\in\sU^{(2)}_0$.
%%%%%%%%%%%%%%%%%%%%%%%%%%%%%%%%%%%%%%%%%%%%%%%%%%%%
\begin{lem}\label{lem:kowTw}
Assume that $\gfrak$ is of finite type. Then the relations
\begin{align}
  \kow(T_{w_0})&=T_{w_0}\otimes T_{w_0}\cdot R^{-1},\label{eq:kowTw01}\\
  \kow(T_{w_0}^{-1})&=T_{w_0}^{-1}\otimes T_{w_0}^{-1}\cdot\kappa\cdot R_{21} \cdot \kappa^{-1}\label{eq:kowTw02}
\end{align}
hold in $\sU^{(2)}_0$.
\end{lem}
%%%%%%%%%%%%%%%%%%%%%%%%%%%%%%%%%%%%%%%%%%%%%%%%%%%%
\begin{proof}
  First observe that Equations \eqref{eq:kowTw01} and \eqref{eq:kowTw02} are equivalent. Indeed, Equation \eqref{eq:RcommuteswithDeltau} for $u=T_{w_0}$ implies that Equation \eqref{eq:kowTw01} is equivalent to 
  \begin{align*}
    \rh\cdot \kow(T_{w_0}) = \kappa^{-1} \cdot \flip \cdot T_{w_0} \ot T_{w_0}.     
  \end{align*}
  The inverse of the above equation is \eqref{eq:kowTw02}. It remains to verify that \eqref{eq:kowTw02} holds. Applying the equivalence between \eqref{eq:kowTw01} and \eqref{eq:kowTw02} to \eqref{eq:kowTi} one obtains 
  \begin{align*}
    \kow(T_i^{-1}) = T_i^{-1} \ot T_i^{-1} \cdot \kappa \cdot{R_i}_{21}\cdot \kappa^{-1}
  \end{align*}
  where ${R_i}_{21}=\flip \cdot R_i\cdot \flip$.  Hence for $T_{w_0}= T_{i_1} T_{i_2} \dots T_{i_t}$ one has
  \begin{align*}
     \kow(T_{w_0}^{-1}) &= \kow(T_{i_t}^{-1})\cdot\dots\cdot \kow(T_{i_2}^{-1})\cdot \kow(T_{i_1}^{-1})\\
     &=T_{i_t}^{-1} \ot T_{i_t}^{-1} \cdot \kappa \cdot{R_{i_t}}_{21}\cdot \kappa^{-1} \cdot \cdots 
       \cdot T_{i_1}^{-1} \ot T_{i_1}^{-1} \cdot \kappa \cdot{R_{i_1}}_{21}\cdot \kappa^{-1}\\
     &=T_{w_0}^{-1} \ot T_{w_0}^{-1} \cdot \kappa \cdot R^{[t]}_{21}\cdot R^{[t-1]}_{21} \cdot \cdots \cdot R^{[2]}_{21}\cdot R^{[1]}_{21} \cdot\kappa^{-1}
  \end{align*}
  where $R^{[j]}_{21}=\flip\cdot R^{[j]} \cdot \flip$. By \eqref{eq:R-factorization} one obtains relation \eqref{eq:kowTw02}.
\end{proof}

%%%%%%%%%%%%%%%%%%%%%%%%%%%%%%%%%%%%%%%%%%
%%%%%%%%%%%%%%%%%%%%%%%%%%%%%%%%%%%%%%%%%%%%%%%%%%%% 
\section{Braided tensor categories with a cylinder twist}\label{sec:cylinder twist}
%%%%%%%%%%%%%%%%%%%%%%%%%%%%%%%%%%%%%%%%%%%%%%%%%%%%
As explained in Section \ref{sec:Rmatrix} the commutativity isomorphisms \eqref{eq:RMN} turn $\cO_{int}$ into a braided tensor category. For any $V\in Ob(\Oint)$ there exists a graphical calculus for the action of $\rh$ on $V^{\ot n}$ in terms of braids in $\C\times[0,1]$, see \cite[Corollary XIII.3.8]{b-Kassel1}. If $\gfrak$ is finite dimensional then $\Oint$ has a duality in the sense of \cite[XIV.2]{b-Kassel1} and there exists a ribbon element which turns $\Oint$ into a ribbon category as defined in \cite[XIV.3.2]{b-Kassel1}. The graphical calculus extends to ribbon categories, see \cite[Theorem XIV 5.1]{b-Kassel1} also for original references.

In \cite{a-tD98} T.~tom Dieck outlined a program to extend the graphical calculus to braids or ribbons in the cylinder $\C^\ast \times [0,1]$. The underlying braid group corresponds to a Coxeter group of type $B$. In the papers \cite{a-tD98}, \cite{a-tDHO98}, \cite{a-HaOld01} tom Dieck and R.~H{\"a}ring-Oldenburg elaborated a categorical setting for such a graphical calculus, leading to the notion of tensor categories with a cylinder braiding. In the present section we recall this notion. In Section \ref{sec:TwCylTw} we will also give a slight generalization which captures all the examples which we obtain from quantum symmetric pairs in Section \ref{sec:deltaK}. These examples are determined by a coideal subalgebra of the braided Hopf algebra $\uqg$. Cylinder braiding in this setting naturally leads to the notion of a cylinder-braided coideal subalgebra of a braided bialgebra $H$ which we introduce in Section \ref{sec:cylinder-braided}. The key point is that a cylinder-braided coideal subalgebra of $H$ has a universal K-matrix which provides solutions of the reflection equation in all representations of $H$.
%%%%%%%%%%%%%%%%%%%%%%%%%%%%%%%%%%%%%%%%%%%%%%%%%%%%
\subsection{Cylinder twists and the reflection equation}\label{sec:CylTw+RE}
%%%%%%%%%%%%%%%%%%%%%%%%%%%%%%%%%%%%%%%%%%%%%%%%%%%%%%%%%%%%%%%
To define cylinder twists let $(\cA,\ot, I, a, l, r)$ be a tensor category as defined in \cite[Definition XI.2.1]{b-Kassel1}. Let $\cB$ be another category and assume that there exists a functor
$\ast:\cB\times\cA \rightarrow \cB$ which we write as
\begin{align*}
  (M,N)\mapsto M\ast N, \qquad (f,g)\mapsto f\ast g
\end{align*}
on objects $M\in Ob(\cB),N\in Ob(\cA)$ and morphisms $f,g$ in $\cB$ and $\cA$, respectively. The functor $\ast$ is called a right action of $\cA$ on $\cB$ if there exist natural isomorphisms $\alpha$ and $\rho$ with
\begin{align*}
  &\alpha_{M,N,N'}: (M\ast N) \ast N' \rightarrow M \ast (N\ot N')& &\mbox{for $M\in Ob(\cB)$, $N,N' \in Ob(\cA)$}\\
  &\rho_M: M\ast I \rightarrow M &&\mbox{for $M\in Ob(\cB)$}
\end{align*}
which satisfy the pentagon and the triangle axiom given in \cite[(2.1), (2.2)]{a-tD98}. A category $\cB$ together with a right action of $\cA$ on $\cB$ is called a right $\cA$-module category.
%%%%%%%%%%%%%%%%%%%%%%%%%%%%%
\begin{eg}\label{eg:Oint}
  As seen in Section \ref{sec:Rmatrix}, the category $\cA=\Oint$ is a braided tensor category. Let $B\subseteq \uqg$ be a right coideal subalgebra, that is a subalgebra satisfying
\begin{align*}
  \kow(B) \subseteq B\ot \uqg.  
\end{align*} 
Let $\cB$ be the category with $Ob(\cB)=Ob(\Oint)$ and $\Hom_\cB(M,N)=\Hom_B(M,N)$ for all $M,N\in Ob(\cB)$. Then $\cB$ is a right $\cA$-module category with $\ast$ given by $M\ast N=M\ot N$.
\end{eg}
%%%%%%%%%%%%%%%%%%%%%%%%%%%%%
From now on, following \cite{a-tD98}, we will consider the following data:
\begin{enumerate}
  \item $(\cA,\ot, I, a, l, r, c)$ is a braided tensor category with braiding $c_{M,N}:M\otimes N \rightarrow N\otimes M$ for all $M,N \in Ob(\cA)$. 
  \item $(\cB,\ast,\alpha,\rho)$ is a right $\cA$-module category.
  \item $\cA$ is a subcategory of $\cB$ with $Ob(\cA)=Ob(\cB)$. In other words $\Hom_\cA(M,N)$ is a subset of $\Hom_\cB(M,N)$ for all $M,N\in Ob(\cA)=Ob(\cB)$. 
  \item $\ast$, $\alpha$, $\rho$ restrict to $\ot$, $a$, $r$ on $\cA\times \cA$.
\end{enumerate}
We call $(\cB,\cA)$ a tensor pair if the above conditions (1)-(4) are satisfied. By condition (3) there exists a forgetful functor
\begin{align*}
  \For_{\cB}^{\cA}: \cA \rightarrow \cB.
\end{align*}
%%%%%%%%%%%%%%%%%%%%%%%%%%%%%%%%%%%%%%%%%%%%%%%%5
\begin{defi}
  Let $(\cB,\cA)$ be a tensor pair. A natural transformation 
  \begin{align*}
    t=(t_M)_{M\in Ob(\cA)}:\For_{\cB}^{\cA}\rightarrow \For_{\cB}^{\cA}
  \end{align*}
 is called a $\cB$-endomorphism of $\cA$. If $t_M:\For_{\cB}^{\cA}(M)\rightarrow \For_{\cB}^{\cA}(M)$ is an automorphisms for all $M\in Ob(\cA)$ then $t$ is called a $\cB$-automorphisms of $\cA$.
\end{defi}
%%%%%%%%%%%%%%%%%%%%%%%%%%%%%%%%%%%%%%%%%%%%%%%%%
In other words, a $\cB$-endomorphism of $\cA$ is a family $t=(t_M)_{M\in Ob(\cA)}$ of morphisms $t_M\in \Hom_\cB(M,M)$ such that
\begin{align}\label{eq:tXnatural}
  t_N\circ f = f\circ t_M 
\end{align}  
for all $f\in \Hom_\cA(M,N)$.
%%%%%%%%%%%%%%%%%%%%%%%%%%%%%%%%
\begin{eg}\label{eg:Oint2}
  The pair $(\cB,\cA)$ from Example \ref{eg:Oint} is a tensor pair. In this setting a $\cB$-endomorphism of $\Oint$ is an element $t\in \sU$ which commutes with all elements of the coideal subalgebra $B\subset \sU$. In other words, the maps $t_M:M\rightarrow M$ are $B$-module homomorphisms for all $M\in Ob(\cO_{int})$.
\end{eg}
%%%%%%%%%%%%%%%%%%%%%%%%%%%%%%%%
The following definition provides the main structure investigated by tom Dieck and H{\"a}ring-Oldenburg in \cite{a-tD98}, \cite{a-tDHO98}, \cite{a-HaOld01}.
%%%%%%%%%%%%%%%%%%%%%%%%%%%%%%%%
\begin{defi}[\cite{a-tD98}]\label{def:cyltw}
  Let $(\cB,\cA)$ be a tensor pair. A cylinder twist for $(\cB,\cA)$ consists of a $\cB$-automorphism $t=(t_M)_{M\in Ob(\cA)}$ of $\cA$ such that
  \begin{align}\label{eq:tXY}
     t_{M\ot N} = (t_M\ast 1_N) c_{N,M} (t_N\ast 1_M) c_{M,N}
  \end{align}
for all $M,N\in Ob(\cA)=Ob(\cB)$.
\end{defi}
%%%%%%%%%%%%%%%%%%%%%%%%%%%%%%%%
The definition of a cylinder twist in \cite{a-tD98} involves a second equation. This equation, however, is a consequence of \eqref{eq:tXY}. This was already observed in \cite[Proposition 2.10]{a-tD98}.
%%%%%%%%%%%%%%%%%%%%%%%%%%%%%%%%
\begin{prop}\label{prop:cylinderRE}
 Let $(\cB,\cA)$ be a tensor pair with a cylinder twist $(t_M)_{M\in Ob(\cA)}$. Then the relation
 \begin{align}\label{eq:tensorRE}
   (t_M\ast 1_N) c_{N,M} (t_N\ast 1_M) c_{M,N} = c_{N,M}(t_N\ast 1_M) c_{M,N} (t_M\ast 1_N)
 \end{align}
 holds for all $M,N\in Ob(\cA)$.
\end{prop} 
%%%%%%%%%%%%%%%%%%%%%%%%%%%%%%%%
\begin{proof}
  As $c_{N,M}$ is a morphism in $\cA$, relation \eqref{eq:tXnatural} implies that
  \begin{align*}
    t_{M\ot N}\circ c_{N,M} = c_{N,M}\circ t_{N\ot M}.
  \end{align*}
  If one inserts relation \eqref{eq:tXY} into both sides of the above equation one obtains Equation \eqref{eq:tensorRE}.
\end{proof}
%%%%%%%%%%%%%%%%%%%%%%%%%%%%%%%%%%%%%%
In \cite{a-tDHO98} Equation \eqref{eq:tensorRE} is called the \textit{four-braid relation}. Here we follow the mathematical physics literature \cite{a-KuSkl92} and call \eqref{eq:tensorRE} the \textit{reflection equation}. Equation \eqref{eq:tXY} is know as the \textit{fusion procedure}, see \cite[6.1]{a-KuSkl92}, as it allows us to fuse the two solutions $t_M$ and $t_N$ of the reflection equation for $M$ and $N$, respectively, to a new solution $t_{M\ot N}$ for the tensor product $M\otimes N$.
%%%%%%%%%%%%%%%%%%%%%%%%%%%%%%%%%%%%%%
\subsection{Twisted cylinder twists}\label{sec:TwCylTw}
%%%%%%%%%%%%%%%%%%%%%%%%%%%%%%%%%%%%%%
Let $(\cB,\cA)$ be a tensor pair. To cover the examples considered in the present paper in full generality, we introduce a slight generalization of tom Dieck's notion of a cylinder twist for $(\cB,\cA)$. This generalization involves a second twist which suggests the slightly repetitive terminology.

Let $\twn:\cA \rightarrow \cA$ be braided tensor equivalence given by $M\mapsto M^\twn$ and $f\mapsto f^\twn\in \Hom(M^\twn,N^\twn)$ for all $M,N\in Ob(\cA)$ and $f\in \Hom(M,N)$. This means that $\twn$ is a braided tensor functor as defined in \cite[Definition XIII.3.6]{b-Kassel1} and an equivalence of categories. A family $t=(t_M)_{M\in Ob(\cA)}$ of morphisms $t_M\in \Hom_\cB(M,M^\twn)$ is called a $\cB$-$\twn$-endomorphism of $\cA$ if 
\begin{align}\label{eq:tf=ftwt}
  t_N \circ f = f^\twn \circ t_M
\end{align}
for all $f\in \Hom_\cA(M,N)$. In other words, a $\cB$-$\twn$-endomorphism of $\cA$ is a natural transformation $t:\For_\cA^\cB\rightarrow \For_\cA^\cB\circ \twn$. 
%%%%%%%%%%%%%%%%%%%%%%%%%%%%%%%%%%555
\begin{defi}\label{def:tw-cylinder-twist}
  Let $(\cB,\cA)$ be a tensor pair and $\twn:\cA\rightarrow \cA$ a braided tensor equivalence. A $\twn$-cylinder twist for $(\cB,\cA)$ consists of a $\cB$-$\twn$-automorphism $t=(t_M)_{M\in Ob(\cA)}$ of $\cA$ such that
  \begin{align}\label{eq:tw-tXY}
     t_{M\ot N} = (t_M\ast 1_{N^\twn}) c_{N^\twn,M} (t_N\ast 1_M) c_{M,N}
  \end{align}
for all $M,N\in Ob(\cA)=Ob(\cB)$.
\end{defi}
%%%%%%%%%%%%%%%%%%%%%%%%%%%%%%%%%%%%%
Let $(\cB,\cA)$ be a tensor pair with a $\twn$-cylinder twist. The relation $c_{N,M}^\twn=c_{N^\twn,M^\twn}$ and \eqref{eq:tf=ftwt} imply that
\begin{align*}
  t_{M\otimes N}\circ c_{N,M}= c_{N^\twn,M^\twn}\circ t_{N\ot M}.
\end{align*}
As in the proof of Proposition \ref{prop:cylinderRE} one now obtains
\begin{align}\label{eq:tw-RE}
 (t_M\ast 1_{N^\twn}) c_{{N^\twn},M} (t_N\ast 1_M) c_{M,N} = c_{N^\twn,M^\twn}(t_N\ast 1_{M^\twn}) c_{M^\twn,N} (t_M\ast 1_N).
\end{align}
%%%%%%%%%%%%%%%%%%%%%%%%%%%%%%%%%%%%%%%%%%%%%%%%%%%%%%%%%%%%%%
\begin{eg}\label{eg:Oint3}
  Consider the setting of Example \ref{eg:Oint}. Let $\varphi:\uqg\rightarrow \uqg$ be a Hopf algebra automorphism. For any $M\in Ob(\Oint)$ let $M^\varphi$ be the integrable representation with left action $\bullet_\varphi$ given by $u\bullet_\varphi m=\varphi(u)m$ for all $u\in \uqg$, $m\in M$. By \cite[Theorem 2.1]{a-Twietmeyer92} one has $\varphi(U^+)=U^+$ and $\varphi(U^0)=U^0$ and hence $M^\varphi\in Ob(\Oint)$. Moreover, as $\varphi(U^0)=U^0$ the map $\varphi$ induces a group isomorphism $\varphi_P:P\rightarrow P$. We assume additionally that $\varphi_P$ is an isometry, that is $(\varphi_P(\lambda),\varphi_P(\mu))=(\lambda,\mu)$ for all $\lambda,\mu\in P$. Then one obtains an auto-equivalence of braided tensor categories 
  \begin{align*}
    \twn: \Oint \rightarrow \Oint
  \end{align*}
given by $\twn(M)=M^\varphi$ and $\twn(f)=f$. In this case relations \eqref{eq:tw-tXY} and \eqref{eq:tw-RE} become
\begin{align}
  t_{M\ot N} = (t_M\ot 1) &\rh_{N^\varphi,M} (t_N\ot 1) \rh_{M,N},  \label{eq:tXY-phi} \\
   (t_M\ot 1) \rh_{N^\varphi,M} (t_N\ot 1) \rh_{M,N} &= \rh_{N^\varphi,M^\varphi}(t_N \ot 1) \rh_{M^\varphi,N} (t_M\ot 1), \label{eq:RE-phi}
\end{align}
respectively, for any $M,N\in Ob(\Oint)$.
\end{eg}
%%%%%%%%%%%%%%%%%%%%%%%%%%%%%%%%%%%%%%%%%%%%%%%%%%%%%%%%%%%%%%%
\subsection{Cylinder-braided coideal subalgebras and the universal K-matrix} \label{sec:cylinder-braided}
%%%%%%%%%%%%%%%%%%%%%%%%%%%%%%%%%%%%%%%%%%%%%%%%%%%%%%%%%%%%%%%
Examples \ref{eg:Oint}, \ref{eg:Oint2}, and \ref{eg:Oint3} can be formalized in the setting of bialgebras and their coideal subalgebras. For the convenience of the reader we recall the relevant notions in the setting of the present paper. 
%%%%%%%%%%%%%%%%%%%%%%%%%%%%%%%%%%%%%%%%%%%%%%%%%%%%%%%%%%%%%%%
\begin{defi}{\upshape{(\cite[Definition VIII.2.2]{b-Kassel1})}}\label{def:braided}
  A bialgebra $H$ with coproduct $\kow_H:H\rightarrow H\otimes H$ is called braided (or quasitriangular) if there exists an invertible element $R^H\in H\ot H$ such that the following two properties hold:
  \begin{enumerate}
    \item For all $x\in H$ one has
      \begin{align}\label{eq:Rx=xR}
        \kow_H^\op(x)=(R^H)^{-1} \kow_H(x) R^H
      \end{align}
      where $\kow_H^\op=\flip\circ \kow_H:H\rightarrow H\otimes H$ denotes the opposite coproduct.
    \item The element $R$ satisfies the relations
      \begin{align}\label{eq:hexagon}
        (\kow_H \ot \mathrm{id}_H)(R^H)= R^H_{23}R^H_{13}, \qquad (\id_H\ot\kow_H)(R^H)=R^H_{12} R^H_{13}
      \end{align}  
      where we use the usual leg-notation.
  \end{enumerate} 
  In this case the element $R^H$ is called a universal R-matrix for $H$.
\end{defi}
%%%%%%%%%%%%%%%%%%%%%%%%%%%%%%%%%%%%%%%%%%%%%%%%%%%%%%%%%%%%%%%
Let $H$ be a braided bialgebra with universal R-matrix $R^H=\sum_{i} s_i\ot t_i\in H\ot H$. In this situation the category $\cA=H$-$\mathrm{mod}$ of $H$-modules is a braided tensor category with braiding
\begin{align}\label{eq:cHMN-def}
  c^H_{M,N}:M\ot N \rightarrow N\ot M, \quad m\ot n\mapsto \sum_i s_in\ot t_im
\end{align}
for all $M,N\in Ob(\cA)$, see \cite[VIII.3]{b-Kassel1}. 
%%%%%%%%%%%%%%%%%%%%%%%%%%%%%%%%%%%%%%%%%
\begin{rema}
  The conventions in Definition \ref{def:braided} slightly differ from the conventions in \cite{b-Kassel1}.
  The reason for this is that following \cite{b-Lusztig94} we use the braiding $R\cdot \kappa^{-1}\cdot \flip$ for $\Oint$ and hence the braiding $R^H\circ \flip$ for $H$-mod. To match conventions observe that $R^H$ in Definition \ref{def:braided} coincides with $R_{21}$ in \cite[Definition VIII.2.2]{b-Kassel1}.
\end{rema}
%%%%%%%%%%%%%%%%%%%%%%%%%%%%%%%%%%%%%%%%%
Let $B$ be a right coideal subalgebra of $H$. As in Example \ref{eg:Oint2} define $\cB$ to be the category with $Ob(\cB)=Ob(\cA)$ and $\Hom_\cB(M,N)=\Hom_B(M,N)$ for all $M,N\in Ob(\cA)$. Then $(\cB,\cA)$ is a tensor pair. For any bialgebra automorphism $\varphi:H\rightarrow H$ define
\begin{align*}
  R^{H,\varphi}=(\id\ot \varphi)(R^H)
\end{align*}
In analogy to the notion of a universal R-matrix the following definition is natural.
%%%%%%%%%%%%%%%%%%%%%%%%%%%%%%%%%5
\begin{defi}\label{def:B-cylinder-braided}
  Let $H$ be a braided bialgebra with universal R-matrix $R^H\in H\ot H$ and let $\varphi:H\rightarrow H$ be an automorphism of braided bialgebras. We say that a right coideal subalgebra $B$ of $H$ is $\varphi$-cylinder-braided if there exists an invertible element $\cK\in H$ such that
   \begin{align}
      \cK b&=\varphi(b)\cK \qquad \mbox{for all $b\in B$,}\label{eq:Kb=bK}\\
      \kow (\cK)&= (\cK\ot 1) R^{H,\varphi} (1\ot \cK) R_{21}^H.\label{eq:deltaK-proper}
   \end{align}
In this case we call $\cK$ a $\varphi$-universal K-matrix for the coideal subalgebra $B$. If $\varphi=\id_H$ then we simply say that $B$ is cylinder-braided and that $\cK$ is a universal K-matrix for $B$.  
\end{defi}
%%%%%%%%%%%%%%%%%%%%%%%%%%%%%%%%%%
The bialgebra automorphism $\varphi$ defines a braided tensor equivalence $\tw:\cA\rightarrow \cA$ given by $M\mapsto M^\varphi$, where as before $M^\varphi$ denotes the $H$-module which coincides with $M$ as a vector space and has the left action $h\otimes m\mapsto \varphi(h)m$.
In the above setting a $\varphi$-universal K-matrix for the coideal subalgebra $B$ defines a family of maps 
\begin{align}\label{eq:tM-def}
  t_M:M\rightarrow M, \quad m\mapsto \cK m,\qquad \mbox{for all $M\in Ob(\cA)$}.
\end{align}
By construction the natural transformation $t=(t_M)_{M\in Ob(\cA)}$ is a $tw$-cylinder twist for the tensor pair $(\cB,\cA)$.
%%%%%%%%%%%%%%%%%%%%%%%%%%%%%%%%%%%%%%%%%%
\begin{rema}
  Observe the parallel between Definition \ref{def:braided} and Definition \ref{def:B-cylinder-braided} in the case $\varphi=\id_{H}$. Indeed, condition \eqref{eq:Rx=xR} means that the maps $c^H_{M,N}$ defined by \eqref{eq:cHMN-def} are $H$-module homomorphisms while condition \eqref{eq:Kb=bK} means that the maps $t_M$ defined by \eqref{eq:tM-def} are $B$-module homomorphisms if $\varphi=\id_{H}$. Similarly, condition \eqref{eq:hexagon} and \eqref{eq:deltaK-proper} both express compatibility with the tensor product.  
\end{rema}
%%%%%%%%%%%%%%%%%%%%%%%%%%%%%%%%%%%%%%%%%%
Definition \ref{def:B-cylinder-braided} can be extended to include the quantized universal enveloping algebra $\uqg$ which is braided only in the completion. In this case we also need to allow for $\cK$ to lie in the completion $\sU$. We repeat Definition \ref{def:B-cylinder-braided} in this setting for later reference. Recall the notation from Section \ref{sec:completion} and from Example \ref{eg:Oint3}. For any Hopf algebra automorphisms $\varphi:\uqg\rightarrow \uqg$ define an element $\rh^\varphi\in \sU^{(2)}$ by 
\begin{align*}
  (\rh^\varphi)_{M,N}=\rh_{M^\varphi,N}\qquad \mbox{for all $M,N\in Ob(\Oint)$.}
\end{align*}
In the following definition we reformulate condition \eqref{eq:deltaK-proper} in terms of $\rh$ and $\rh^\varphi$.
%%%%%%%%%%%%%%%%%%%%%%%%%%%%%%%%%%%%%%%%%%%
\begin{defi}\label{def:U-cylinder-braided}
  Let $\varphi:\uqg\rightarrow \uqg$ be a Hopf algebra automorphism. A right coideal subalgebra $B\subseteq \uqg$ is called $\varphi$-cylinder-braided if there exists an invertible element $\cK\in \sU$ such that the relation
  \begin{align}
      \cK b&=\varphi(b) \cK \qquad \mbox{for all $b\in B$} \label{eq:phi-cylinder1}
  \end{align}
holds in $\sU$ and the relation
  \begin{align}\label{eq:phi-cylinder2}      
      \kow (\cK)&= (\cK\ot 1)\cdot \rh^\varphi \cdot (\cK\ot 1)\cdot \rh
   \end{align}
holds in $\sU^{(2)}$. In this case we call $\cK$ a $\varphi$-universal K-matrix for the coideal subalgebra $B$. If $\varphi=\id_\uqg$ then we simply say that $B$ is cylinder-braided and that $\cK$ is a universal K-matrix for $B$.  
\end{defi}
%%%%%%%%%%%%%%%%%%%%%%%%%%%%%%%%%%%%%%%%%%
Similarly to the discussion for the bialgebra $H$ above, a cylinder-braided coideal subalgebra of $\uqg$ naturally gives rise to a cylinder twist. For later reference we summarize the situation in the following remark.
%%%%%%%%%%%%%%%%%%%%%%%%%%%%%%%%%%%%%%%%%%
\begin{rema}\label{rem:uniK=tw-cyl-tw}
Let $B \subseteq \uqg$ be a right coideal subalgebra and let $(\cB,\Oint)$ be the tensor pair from Example \ref{eg:Oint}. Moreover, let $\varphi:\uqg\rightarrow\uqg$ be a Hopf-algebra automorphism and let $\twn:\Oint\rightarrow \Oint$ be the corresponding braided tensor equivalence as in Example \ref{eg:Oint3}. An element $\cK\in \sU$ is a $\varphi$-universal K-matrix for $B$ if and only if $\cK$ is a $\twn$-cylinder twist of $(\cB,\Oint)$. In this case, in particular, the element $t=\cK\in \sU$ satisfies the fusion procedure \eqref{eq:tXY-phi} and the reflection equation \eqref{eq:RE-phi} for all $M,N\in \Oint$.
\end{rema}
%%%%%%%%%%%%%%%%%%%%%%%%%%%%%%%%%%%%%%%%%%
\subsection{Cylinder braided coideal subalgebras via characters}\label{sec:characters}
%%%%%%%%%%%%%%%%%%%%%%%%%%%%%%%%%%%%%%%%%%
In their paper \cite{a-DKM03} J.~Donin, P.~Kulish, and A.~Mudrov introduced the notion of a universal solution of the reflection equation which they also called a universal K-matrix. In contrast to Definition \ref{def:B-cylinder-braided}, this notion does not refer to a coideal subalgebra of a Hopf algebra. Nevertheless, there is a close relationship between Definition \ref{def:B-cylinder-braided} and the notion of a universal K-matrix in \cite{a-DKM03}, and it is the purpose of the present section to explain this. This material will not be used in later parts of the present paper.

As in Section \ref{sec:cylinder-braided} let $(H,R^H)$ be a braided Hopf algebra over a field $\fieldtwo$. We retain the conventions from Definition \ref{def:braided} and hence the symbol $\calR$ in \cite{a-DKM03} corresponds to $R^H_{21}$ in our conventions. 
Let $H^\ast=\Hom_\fieldtwo(H,\fieldtwo)$ denote the linear dual space of $H$. Recall from \cite{a-ResSTS88} that the braided Hopf $H$ algebra is called factorizable if the linear map
\begin{align*}   
  H^\ast \rightarrow H, \qquad f\mapsto (f\ot \id)(R^H_{12} R^H_{21})
\end{align*}
is an isomorphism of vector spaces. This is only possible if $H$ is finite dimensional. If $H$ is factorizable then Donin, Kulish, and Mudrov call the element 
\begin{align*}
  \cK^{\dkm} = R^H_{12} R^H_{21}\in H\ot H
\end{align*}
the universal K-matrix of $H$. It follows from \eqref{eq:hexagon} that
\begin{align}
  (\id\ot \kow)(\cK^{\dkm}) &=  R^H_{01}\, R^H_{02}\, R^H_{20}\, R^H_{10}\, R^H_{12} (R^H_{12})^{-1}\,\nonumber\\
  &=  R^H_{01}\, R^H_{10}\, R^H_{12}\, R^H_{02}\, R^H_{20} R^H_{21}(R^H_{12}R^H_{21})^{-1}\,\nonumber\\
                             &= \cK^{\dkm}_{01} R^H_{12}\, \cK^{\dkm}_{02} R^H_{21}(R^H_{12}R^H_{21})^{-1} \label{eq:id-deltaK-DKM}
\end{align}
where we label the tensor legs of $H^{\ot 3}$ by $0$, $1$, $2$. The above formula is closely related to formula \eqref{eq:deltaK-proper} for $\kow(\cK)$. There are two differences, however, namely the occurrence of the additional factor $(R^H_{12} R^H_{21})^{-1}$ and the fact that \eqref{eq:id-deltaK-DKM} holds in $H^{\ot 3}$ while formula \eqref{eq:deltaK-proper} holds in $H^{\ot 2}$. Moreover, the element $\cK^{\dkm}$ makes no reference to a coideal subalgebra of $\uqg$.

To address the first difference, recall from \cite[Definition XIV.6.1]{b-Kassel1} that the braided Hopf algebra $(H,R^H)$ is called a ribbon algebra if there exists a central element $\theta_H\in H$ such that
\begin{align*}
  \kow(\theta_H)= (R^H_{12} R^H_{21})^{-1}(\theta_H\ot \theta_H), \qquad \vep(\theta_H)=1, \qquad S(\theta_H)=\theta_H.
\end{align*}
If such a ribbon element $\theta_H$ exists then the element
\begin{align*}
  \cK^{\dkm,\theta}=(1\ot \theta_H^{-1}) \cK^\dkm \in H\ot H
\end{align*}
satisfies the relation
\begin{align}
  (\id\ot \kow)(\cK^{\dkm,\theta}) &= \cK^{\dkm,\theta}_{01} R^H_{12}\, \cK^{\dkm,\theta}_{02} R^H_{21}  \label{eq:id-deltaKtheta-DKM}
\end{align}
in $H^{\ot 3}$.  

To eliminate the additional tensor factor in \eqref{eq:id-deltaK-DKM} and \eqref{eq:id-deltaKtheta-DKM} let
\begin{align*}
  f: H\rightarrow \fieldtwo
\end{align*}
be a character, that is a one-dimensional representation. Define
\begin{align}\label{eq:Bf-def}
  B_f = \{(f\ot \id)\kow(h)\,|\, h\in H\}
\end{align}
and observe that $B_f$ is a right coideal subalgebra of $H$. The element
\begin{align}
   \cK^{\dkm,\theta,f}=(f\ot \theta_H^{-1}) \cK^\dkm \in H
\end{align}
commutes with all elements of $B_f$ because $\kow(h)$ commutes with $\cK^{\dkm}=R^H_{12} R^H_{21}$ for all $h\in H$ by \eqref{eq:Rx=xR}. By \eqref{eq:id-deltaKtheta-DKM} one has
\begin{align*}
  \kow(\cK^{\dkm,\theta,f}) &=  (\cK^{\dkm,\theta,f}\ot 1) R^H\, (1\ot\cK^{\dkm,\theta,f}) R^H_{21}  
\end{align*}
which coincides with relation \eqref{eq:deltaK-proper} in Definition \ref{def:B-cylinder-braided}. We summarize the above discussion.
%%%%%%%%%%%%%%%%%%%%%%%%%%%%%%%%%%%%%%%%%%%%%%%%%%%%%%%%%%5
\begin{prop}\label{prop:Kdkm}
  Let $(H,R^H,\theta_H)$ be a factorizable ribbon Hopf algebra over a field $\fieldtwo$ and let $f:H\rightarrow \fieldtwo$ be a character. Then the right coideal subalgebra $B_f$ defined by \eqref{eq:Bf-def} is cylinder braided with universal K-matrix 
  \begin{align*}
    \cK^{\dkm,\theta,f}=(f\ot \theta_H^{-1})(R^H_{12} R^H_{21}) \in H.
  \end{align*}  
\end{prop}
%%%%%%%%%%%%%%%%%%%%%%%%%%%%%%%%%%%%%%%%%%%%%%%%%%%%%%%%%%%
By the above proposition the element $\cK^{\dkm,\theta,f}$ satisfies the reflection equation in every tensor product $M\ot N$ of representations of $H$. As the ribbon element $\theta_H$ is central, the element $(f\ot 1)(R^H_{12} R^H_{21})$ also satisfies the reflection equation. 
%%%%%%%%%%%%%%%%%%%%%%%%%%%%%%%%%%%%%%%%%%%%%%%%%%%%%%%%
\begin{rema}
  Assume that $\gfrak$ is of finite type. If one naively translates the construction of Proposition \ref{prop:Kdkm} to the setting of $\uqg$ then the resulting universal K-matrix is the identity element because $\uqg$ does not have any interesting characters. However, in \cite{a-DKM03} a universal K-matrix is also defined for non-factorizable $H$. In this case one chooses $\cK^\dkm$ to be the canonical element in $\tilde{H}^\ast\ot H$ where $\tilde{H}^\ast$ denotes a twisted version of the dual Hopf algebra $H^\ast$. One obtains a universal K-matrix by application of a character $f$ of $\tilde{H}^\ast$. This framework translates to the setting of $\uqg$ if one replaces $\tilde{H}^\ast$ by the braided restricted dual of $\uqg$.  The braided restricted dual of $\uqg$ is isomorphic as an algebra to the (right) locally finite part
  \begin{align*}
    F_r(\uqg)=\{x\in \uqg\,|\,\dim(\ad_r(\uqg)(x))<\infty\}
  \end{align*}
where $\ad_r(u)(x)=S(u_{(1)})x u_{(2)}$ for $u,x\in \uqg$ denotes the right adjoint action. The locally finite part has many non-trivial characters, and a cylinder braiding for $\Oint$ can be associated to each of them, see \cite[Propositions 2.8, 3.14]{a-KolbStok09}.

The constructions in \cite{a-DKM03} and in this subsection, however, do not answer the question how to find characters of $\tilde{H}^\ast$. For $H=\uqg$ this amounts to finding numerical solutions of the reflection equation which satisfy additional compatibility conditions. For $\gfrak=\slfrak_n(\C)$ this is a manageable problem, see \cite[Remark 5.11]{a-KolbStok09}. It would be interesting to find a conceptual classification of characters of $F_r(\uqg)$ for all $\gfrak$ of finite type. 
\end{rema}
%%%%%%%%%%%%%%%%%%%%%%%%%%%%%%%%%%%%%%%%%%
\section{Quantum symmetric pairs}\label{sec:QSP}
%%%%%%%%%%%%%%%%%%%%%%%%%%%%%%%%%%%%%%%%%
In the remainder of this paper we will show that quantum symmetric pair coideal subalgebras of $\uqg$ are $\varphi$-cylinder-braided as in Definition \ref{def:U-cylinder-braided} for a suitable automorphism $\varphi$ of $\uqg$. To set the scene we now recall the construction and properties of quantum symmetric pairs. We will in particular recall the existence of the intrinsic bar involution from \cite{a-BalaKolb14p} in Section \ref{sec:bar-involution}. Quantum symmetric pairs depend on a choice of parameters and the existence of the bar involution imposes further restrictions. In Section \ref{sec:Assumptions}, for later reference, we summarize our setting and assumptions including the restrictions on parameters.
%%%%%%%%%%%%%%%%%%%%%%%%%%%%%%%%%%%%%%%%%%%%%%%%%%%%%%%%%
\subsection{Involutive automorphisms of the second kind}\label{sec:SecondKind}
%%%%%%%%%%%%%%%%%%%%%%%%%%%%%%%%%%%%%%%%%%%%%%%%%%%%%%%%%
Let $\bfrak^+$ denote the positive Borel subalgebra of $\gfrak$. An automorphism $\theta:\gfrak\rightarrow \gfrak$ is said to be of the second kind if $\dim(\theta(\bfrak^+)\cap \bfrak^+)<\infty$. Involutive automorphisms of the second kind of $\gfrak$ were essentially classified in \cite{a-KW92}, see also \cite[Theorem 2.7]{a-Kolb14}. In this section we recall the combinatorial data underlying this classification.

For any subset $X$ of $I$ let $\g_X$ denote the corresponding Lie subalgebra of $\gfrak$. The sublattice $Q_X$ of $Q$ generated by $\{\alpha_i\,|\,i\in X\}$ is the root lattice of $\gfrak_X$. If $\gfrak_X$ is of finite type then let $\rho_X$ and $\rho_X^\vee$ denote the half sum of positive roots and positive coroots of $\gfrak_X$, respectively. The Weyl group $W_X$ of $\gfrak_X$ is the parabolic subgroup of $W$ generated by all $\sigma_i$ with $i\in X$. If $\g_X$ is of finite type then let $w_X\in W_X$ denote the longest element. Let $\Aut(A)$ denote the group of permutations $\tau:I\rightarrow I$ such that the entries of the Cartan matrix $A=(a_{ij})$ satisfy $a_{ij}=a_{\tau(i)\tau(j)}$ for all $i,j\in I$. Let $\Aut(A,X)$ denote the subgroup of all $\tau\in\Aut(A)$ which additionally satisfy $\tau(X)=X$.

Involutive automorphisms of $\gfrak$ of the second kind are parametrized by combinatorial data attached to the Dynkin diagram of $\gfrak$. This combinatorial data is a generalization of Satake diagrams from the finite dimensional setting to the Kac-Moody case, see \cite{a-Araki62}, \cite[Definition 2.3]{a-Kolb14}.
%%%%%%%%%%%%%%%%%%%%%%%%%%%%%%%%%%%%%%
\begin{defi}\label{def:admissible}
  A pair $(X,\tau)$ consisting of a subset $X\subseteq I$ of finite type and an element $\tau\in \Aut(A,X)$ is called admissible if the following
  conditions are satisfied:
  \begin{enumerate}
    \item $\tau^2=\id_I$.
    \item \label{adm2} The action of $\tau$ on $X$ coincides with the action of $-w_X$.
    \item \label{adm3} If $j\in I\setminus X$ and $\tau(j)=j$ then
         $\alpha_j(\rho_X^\vee)\in \Z$. 
   \end{enumerate}
\end{defi}
%%%%%%%%%%%%%%%%%%%%%%%%%%%%%%%%%%%%%%%
We briefly recall the construction of the involutive automorphisms $\theta=\theta(X,\tau)$ corresponding to the admissible pair $(X,\tau)$, see \cite[Section 2]{a-Kolb14} for details. Let $\omega:\gfrak\rightarrow \gfrak$ denote the Chevalley involution as in \cite[(1.3.4)]{b-Kac1}. Any $\tau\in \Aut(A,X)$ can be lifted to a Lie algebra automorphism $\tau:\gfrak\rightarrow \gfrak$. Moreover, for $X\subset I$ of finite type let $\Ad(w_X): \gfrak\rightarrow \gfrak$ denote the corresponding braid group action of the longest element in $W_X$. Finally, let $s:I\rightarrow \field^\times$ be a function such that
\begin{align}
s(i)&=1 &&\textrm{ if } i \in X \textrm{ or } \tau(i)=i, \label{eq:defs(i)1} \\
\frac{s(i)}{s(\tau(i))}&=(-1)^{\alpha_i(2\rho_X^\vee)}  &&\textrm{ if } i \notin X \textrm{ and } \tau(i)\ne i. \label{eq:defs(i)2}
\end{align}
Such a function always exists. The map $s$ gives rise to a group homomorphism 
$s_Q:Q\rightarrow \field^\times$ such that $s_Q(\alpha_i)=s(i)$. This in turn allows us to define a Lie algebra automorphism $\Ad(s):\gfrak\rightarrow \gfrak$ such that the restriction of $\Ad(s)$ to any root space $\gfrak_\alpha$ is given by multiplication by $s_Q(\alpha)$. 
%%%%%%%%%%%%%%%%%%%%%%%%%%%%%%%%%%%%%%%%%%%%%%%%%%
\begin{rema}
In \cite[(2.7)]{a-Kolb14} and in \cite[(3.2)]{a-BalaKolb14p} we chose the values $s(i)$ for $i\in I$ to be certain fourth roots of unity. This had the advantage that $\Ad(s)$ commutes with the involutive automorphism corresponding to the admissible pair $(X,\tau)$. However, the only properties of $s$ used in \cite{a-Kolb14}, \cite{a-BalaKolb14p}, and in the present paper are the relations \eqref{eq:defs(i)1} and \eqref{eq:defs(i)2}. It is hence possible to choose $s(i)\in \{-1,1\}$. This is more suitable for the categorification program in \cite{a-EhrigStroppel13p} and for the program of canonical bases for coideal subalgebras in \cite{a-BaoWang13p}.
\end{rema}
%%%%%%%%%%%%%%%%%%%%%%%%%%%%%%%%%%%%%%%%%%%%%%%%%%%%
With the above notations at hand we can now recall the classification of involutive automorphisms of the second kind in terms of admissible pairs.
%%%%%%%%%%%%%%%%%%%%%%%%%%%%%%%%%%%%%%%%%%%%%%%%%%%%
\begin{thm}{\upshape(\cite{a-KW92}, \cite[Theorem 2.7]{a-Kolb14})}\label{classThm}
   The map 
   \begin{align*}
     (X,\tau)\mapsto \theta(X,\tau)=\Ad(s) \circ \Ad(w_X) \circ \tau\circ \omega
   \end{align*}
  gives a bijection between the set of $\Aut(A)$-orbits of admissible pairs for
  $\gfrak$ and the set of $\Aut(\gfrak)$-conjugacy classes of involutive automorphisms
  of the second kind. 
\end{thm}
%%%%%%%%%%%%%%%%%%%%%%%%%%%%%%%%%%%%%%%%%%%%%%%%%%%
Let $\kfrak=\{x\in \gfrak\,|\,\theta(X,\tau)(x)=x\}$ denote the fixed Lie subalgebra of $\gfrak$. We refer to $(\gfrak,\kfrak)$ as a symmetric pair.
The involution $\theta=\theta(X,\tau)$ leaves $\hfrak$ invariant. The induced map $\Theta:\hfrak^\ast\rightarrow \hfrak^\ast$ is given by 
\begin{align*}
   \Theta=-w_X\circ \tau
\end{align*}
where $\tau(\alpha_i)=\alpha_{\tau(i)}$ for all $i\in I$, see \cite[2.2, (2.10)]{a-Kolb14}. Hence $\Theta$ restricts to an involution of the root lattice. Let $Q^\Theta$ be the sublattice of $Q$ consisting of all elements fixed by $\Theta$.
For later use we note that
\begin{equation}
  \Theta(\alpha_{\tau(i)})-\alpha_{\tau(i)}=\Theta(\alpha_{i})-\alpha_{i}\qquad\mbox{for all $i\in I$,} \label{eq:alphai-alphataui}
\end{equation}
see \cite[Lemma 3.2]{a-BalaKolb14p}.
%%%%%%%%%%%%%%%%%%%%%%%%%%%%%%%%%%%%%%%%%%%%%%%
%%%%%%%%%%%%%%%%%%%%%%%%%%%%%%%%%%%%%%%%%%%%%%%
\subsection{The construction of quantum symmetric pairs} 
%%%%%%%%%%%%%%%%%%%%%%%%%%%%%%%%%%%%%%%%%%%%%%%
We now recall the definition of quantum symmetric pair coideal subalgebras following \cite{a-Kolb14}. For the remainder of this paper let $(X,\tau)$ be an admissible pair and $s:I\to \mathbb{K}^\times$ a function satisfying \eqref{eq:defs(i)1} and \eqref{eq:defs(i)2}. Let $\cM_X=U_q(\gfrak_X)$ denote the subalgebra of $\uqg$ generated by the elements $E_i$, $F_i$, $K_i^{\pm 1}$ for all $i\in X$. Correspondingly, let $\cM_X^+$ and $\cM_X^-$ denote the subalgebras of $\cM_X$ generated by the elements in the sets $\{E_i\,|\,i \in X\}$ and  $\{F_i\,|\,i \in X\}$, respectively. 

The derived Lie subalgebra $\gfrak'$ is invariant under the involutive automorphism $\theta=\theta(X,\tau)$.  One can define a quantum group analog $\theta_q:\uqgp\rightarrow \uqgp$ of $\theta$, see \cite[Definition 4.3]{a-Kolb14} for details. The quantum involution $\theta_q$ is a $\field(q^{1/d})$-algebra automorphism but it is not a coalgebra automorphism and $\theta_q^2\neq \id_{\uqgp}$. However, the map $\theta_q$ has the following desirable properties 
\begin{align*}
\theta_q|_{\cM_X}&=\id_{\cM_X}, &&\\
\theta_q(K_\mu)&=K_{\Theta(\mu)} &&\mbox{for all $\mu\in Q$},\\
\theta_q(K_i^{-1}E_i)&=-s(\tau(i))^{-1}T_{w_X}(F_{\tau(i)})\in  U^-_{\Theta(\alpha_i)} && \mbox{for all $i\in I\setminus X$,}\\
\theta_q(F_iK_i)&= -s(\tau(i))T_{w_X}(E_{\tau(i)})\in U^+_{\Theta(-\alpha_i)}&& \mbox{for all $i\in I\setminus X$.}
\end{align*}
To shorten notation define
\begin{align}
X_i&=\theta_q(F_iK_i)=-s(\tau(i))T_{w_X}(E_{\tau(i)})\qquad \qquad\mbox{for all $i\in I\setminus X$.}\label{def:Xi} 
\end{align}
Quantum symmetric pair coideal subalgebras depend on a choice of parameters $\bc=(c_i)_{i\in I\setminus X}\in  (\mathbb{K}(q^{1/d})^\times)^{I\setminus X}$  and $\bs=(s_i)_{i\in I\setminus X}\in  \mathbb{K}(q^{1/d})^{I\setminus X}$. Define 
\begin{align}
  I_{ns} &= \{ i \in I\setminus X | \tau(i)= i \textrm{ and } a_{ij}=0 \textrm{ for all } j\in X \}. \label{def:setIns} 
\end{align}
In \cite[(5.9), (5.11)]{a-Kolb14}) the following parameter sets appeared 
\begin{align}
\mathcal{C} &= \{ \mathbf{c}\in (\mathbb{K}(q^{1/d})^\times)^{I\setminus X} | c_i=c_{\tau(i)} \textrm{ if } \tau(i)\ne i \textrm{ and } (\alpha_i,\Theta(\alpha_i))=0 \}, \label{def:setC} \\
\mathcal{S} &= \{ \mathbf{s}\in \mathbb{K}(q^{1/d})^{I\setminus X} | s_j\ne 0 \Rightarrow ( j \in I_{ns} \,\, \mathrm{ and }\,\, a_{ij} \in -2\mathbb{N}_0 \,\, \forall i \in I_{ns}\setminus \{j\}) \}, \label{def:setS}
\end{align}
see also \cite[Remark 3.3]{a-BalaKolb14p}. 

%%%%%%%%%%%%%%%%%%%%%%%%%%%%%%%%%%%%%%%%%%%%%%%%%%
Let $U^0_\Theta{}'$ be the subalgebra of $U^0$ generated by all $K_\mu$ with $\mu\in Q^\Theta$. 
%%%%%%%%%%%%%%%%%%%%%%%%%%%%%%%%%%%%%%%%%%%%%%%%%%
\begin{defi}\label{def:qsp}
Let $(X,\tau)$ be an admissible pair, $\bc=(c_i)_{i\in I\setminus X}\in \cC$, and $\bs=(s_i)_{i\in I\setminus X}\in \cS$. The quantum symmetric pair coideal subalgebra $\coid=\coid(X,\tau)$ is the subalgebra of $\uqgp$ generated by $\mathcal{M}_X$, $U^0_\Theta{}'$, and the elements 
\begin{align}
  B_i &= F_i + c_i X_i K_i^{-1} + s_i K_i ^{-1} \label{eq:Bi-def} 
\end{align}
for all $i \in I\setminus X$.
\end{defi}
%%%%%%%%%%%%%%%%%%%%%%%%%%%%%%%%%%%%%%%%%%%%%%%%%%%
\begin{rema}
The conditions $\mathbf{c}\in \mathcal{C}$, $\mathbf{s}\in \mathcal{S}$ can be found in \cite[(5.9),(5.11)]{a-Kolb14}. They are necessary to ensure that the intersection of the coideal subalgebra with $U^0$ is precisely $U^0_\Theta{}'$. This in turn implies that the coideal subalgebra $\coid$ specializes to $U(\kfrak')$ at $q=1$ with $\kfrak'=\{x\in \gfrak'\,|\,\theta(x)=x\}$, see \cite[Remark 5.12, Theorem 10.8]{a-Kolb14}. 
\end{rema}
%%%%%%%%%%%%%%%%%%%%%%%%%%%%%%%%%%%%%%%%%%%%%%%%%%%

For $i \in X$ we set $c_i=s_i=0$ and $B_i=F_i$. This convention will occasionally allow us to treat the cases $i\in X$ and $i\notin X$ simultaneously. 

The algebra $\coid$ is a right coideal subalgebra of $\uqgp$, that is
\begin{align*}
  \kow(\coid)\subseteq \coid\ot \uqgp,
\end{align*}
see \cite[Proposition 5.2]{a-Kolb14}. One can calculate the coproduct of the generators $B_i$ for $i\in I\setminus X$ more explicitly and obtains
\begin{align}\label{eq:kowBi}
  \kow(B_i) = B_i\ot K_i^{-1} + 1\ot F_i + c_i \,r_{\tau(i)}(X_i) K_i^{-1} K_{\tau(i)}\ot E_{\tau(i)}K_i^{-1} + \Upsilon
\end{align}
for some $\Upsilon\in \cM_X U^0_\Theta{'}\ot \sum_{\gamma>\alpha_{\tau(i)}} U^+_\gamma K_i^{-1}$, see \cite[Lemma 7.2]{a-Kolb14}. By \eqref{eq:riDelta} this implies that
\begin{align}\label{eq:rj(Xi)} 
 r_j(X_i)=0 \qquad\qquad \mbox{whenever $j\neq \tau(i)$.} 
\end{align}
In view of \eqref{eq:kowBi} it makes sense to define
\begin{align}
\mathcal{Z}_i&=r_{\tau(i)}(X_i) K_i^{-1}K_{\tau(i)}. \label{def:Zi}
\end{align}
The elements $\cZ_i$ play a crucial role in the description of $\coid$ in terms of generators and relations, see \cite[Section 7]{a-Kolb14}, \cite[Section 3.2]{a-BalaKolb14p}.
%%%%%%%%%%%%%%%%%%%%%%%%%%%%%%%%%%%%%%%%%%%%%%%%%%%%%%%%%
\subsection{The bar involution for quantum symmetric pairs}\label{sec:bar-involution}
%%%%%%%%%%%%%%%%%%%%%%%%%%%%%%%%%%%%%%%%%%%%%%%%%%%%%%%%%
The bar involution for $\uqg$ defined in \eqref{eq:defbarUqg} does not map $\coid$ to itself. Inspired by the papers \cite{a-EhrigStroppel13p}, \cite{a-BaoWang13p}, it was shown in \cite{a-BalaKolb14p} under mild additional assumptions that $\coid$ allows an intrinsic bar involution $\barB:\coid\rightarrow \coid$. We now recall these assumptions and the construction of the intrinsic bar involution for $\coid$.

In \cite[Section 3.2]{a-BalaKolb14p} the algebras $\coid$ are given explicitly in terms of generators and relations for all Cartan matrices $A=(a_{ij})$ and admissible pairs $(X,\tau)$ which satisfy the following properties:
\begin{enumerate}
     \item[(i)]  If $i\in I\setminus X$ with $\tau(i)=i$ and $j\in X$ then $a_{ij}\in \{0,-1,-2\}$.
     \item[(ii)]  If $i\in I\setminus X$ with $\tau(i)=i$ and $i\neq j\in I\setminus X$ then $a_{ij}\in \{0,-1,-2,-3\}$.
  \end{enumerate}
The existence of the bar involution $\barB$ on $\coid$ was then proved  by direct computation based on the defining relations. 
%%%%%%%%%%%%%%%%%%%%%%%%%%%%%
\begin{thm}{\upshape(\cite[Theorem 3.11]{a-BalaKolb14p})}\label{thm:bar-involution}
  Assume that conditions (i) and (ii) hold. The following statements are equivalent.
  \begin{enumerate}
    \item There exists a $\field$-algebra automorphism $\barB:B_{\bc,\bs}\rightarrow 
          B_{\bc,\bs}$, $x\mapsto \overline{x}^B$ such that
          \begin{align}
\overline{x}^B=\overline{x} \textrm{ for all }  x \in \mathcal{M}_X U^0_\Theta{}', \qquad  \overline{B_i}^B&=B_i \textrm{ for all } i \in I\setminus X. \label{eq:defbarBcs}
         \end{align} 
         In particular $\overline{q^{1/d}}=q^{-1/d}$.
    \item The relation
       \begin{align}\label{eq:ocZi}
         \overline{c_i \cZ_i} = q^{(\alpha_i,\alpha_{\tau(i)})} c_{\tau(i)} \cZ_{\tau(i)}
       \end{align}
         holds for all $i\in I\setminus X$ for which $\tau(i)\neq i$ or for which there exists $j\in I\setminus \{i\}$ such that $a_{ij}\neq 0$. 
  \end{enumerate}
\end{thm}
%%%%%%%%%%%%%%%%%%%%%%%%%%%%%
It is conjectured that Theorem \ref{thm:bar-involution} holds without the assumptions (i) and (ii).
In \cite[Proposition 3.5]{a-BalaKolb14p} it was proved that for all $i\in I\setminus X$ one has 
\begin{align}\label{eq:oZ}
  \overline{\cZ_i} = \nu_i q^{(\alpha_i,\alpha_i- w_X(\alpha_i)-2\rho_X)} \cZ_{\tau(i)}
\end{align}
for some $\nu_i\in \{-1,1\}$. For $\gfrak$ of finite type it was moreover proved that $\nu_i=1$ for all $i\in I\setminus X$, and this was conjectured to hold also in the Kac-Moody case \cite[Proposition 2.3, Conjecture 2.7]{a-BalaKolb14p}. 
%%%%%%%%%%%%%%%%%%%%%%%%%%%%%
\subsection{Assumptions}\label{sec:Assumptions}
%%%%%%%%%%%%%%%%%%%%%%%%%%%%%
For later reference we summarize our setting. As before $\gfrak$ denotes the Kac-Moody algebra corresponding to the symmetrizable Cartan matrix $A=(a_{ij})$ and $(X,\tau)$ is an admissible pair. We fix parameters $\bc\in \cC$ and $\bs\in \cS$ and let $\coid$ denote the corresponding quantum symmetric pair coideal subalgebra of $\uqgp$ as given in Definition \ref{def:qsp}. Additionally the following assumptions are made for the remainder of this paper.
\begin{enumerate}
  \item The Cartan matrix $A=(a_{ij})$ satisfies conditions (i) and (ii) in Section \ref{sec:bar-involution}. 
  \item The parameters $\bc\in \cC$ satisfy the condition
  \begin{align}
\overline{c_i\mathcal{Z}_i}=q^{(\alpha_i,\alpha_{\tau(i)})}c_{\tau(i)}\mathcal{Z}_{\tau(i)}
  \qquad\mbox{ for all $i \in I\setminus X$.} \label{Mparameters1}
  \end{align}
  \item The parameters $\bs\in \cS$ satisfy the condition
  \begin{align}
    \overline{s}_i=s_i \qquad \mbox{for all $i \in I\setminus X$.} \label{Mparameters2}
  \end{align}
\item One has $\nu_i=1$ for all $i\in I \setminus X$, that is \cite[Conjecture 2.7]{a-BalaKolb14p} holds true.
\end{enumerate}
If (4) holds then using \eqref{eq:oZ} and \eqref{eq:alphai-alphataui} one sees that Equation \eqref{Mparameters1} is equivalent to 
\begin{align}\label{eq:octau}
  c_{\tau(i)} = q^{(\alpha_i,\Theta(\alpha_i)-2\rho_X)}\overline{c_i}. 
\end{align}
%%%%%%%%%%%%%%%%%%%%%%%%%%%%%%%%%%%%%
\begin{rema}
  Assumption (1) is only used in the proof of Theorem \ref{thm:bar-involution}. Assumption (4) is only used to obtain Equation \eqref{eq:octau}. Once Theorem \ref{thm:bar-involution} is established without assuming conditions (i) and (ii), and once it is proved that $\nu_i=1$ for all $i\in I\setminus X$, all results of this paper hold for $\coid$ with $\bc\in \cC$ and $\bs\in \cS$ satisfying relations \eqref{Mparameters1} and \eqref{Mparameters2}.
\end{rema}
%%%%%%%%%%%%%%%%%%%%%%%%%%%%%%%%%%%%%%%%%%%%%%%%%%
\begin{rema}
  Observe that assumption (2) is a stronger statement then what is needed for the existence of the bar-involution $\barB$ in Theorem \ref{thm:bar-involution}. This stronger statement will be used in the construction of the quasi-K-matrix in Section \ref{sec:constructXmu}, see the end of the proof of Lemma \ref{lem:Xcond1}. It is moreover used in the calculation of the coproduct of the universal K-matrix in Section \ref{sec:coproductK}, see proof of Lemma \ref{lem:1riXK2}. Assumption (3) is new and will be used in the proofs of Lemma \ref{lem:Xcond2} and Theorem \ref{thm:Xfrak}. 
\end{rema}
%%%%%%%%%%%%%%%%%%%%%%%%%%%%%%%%%%%%%%%%%%%%%%%%%%
\begin{rema}\label{rem:qhalf}
For every admissible pair there exist parameters $c_i \in \mathbb{K}(q)$ satisfying Equation \eqref{eq:octau}, see \cite[Remark 3.14]{a-BalaKolb14p}.
\end{rema}

%%%%%%%%%%%%%%%%%%%%%%%%%%%%%%%%%%%%%%%%%%%%%%%%%%
\section{The Quasi K-matrix $\Xfrak$}\label{sec:quasiK}
%%%%%%%%%%%%%%%%%%%%%%%%%%%%%%%%%%%%%%%%%%%%%%%%%%
The bar involution $x\mapsto \overline{x}$ on $\uqg$ defined by \eqref{eq:defbarUqg} and the internal bar involution $x\mapsto \overline{x}^B$ on $\coid$ defined by \eqref{eq:defbarBcs} satisfy $\overline{B_i}\ne \overline{B_i}^B$ if $i\in I\setminus X$. Hence the two bar involutions do not coincide when restricted to $B_{\bc,\bs}$. The aim of this section is to construct an element $\Xfrak\in \widehat{U^+}$ which intertwines between the two bar involutions. More precisely, we will find $(\Xfrak_\mu)_{\mu \in Q^+}$ with $\Xfrak_\mu \in U^+_\mu$ and $\Xfrak_0=1$ such that $\Xfrak=\sum_\mu \Xfrak_\mu$ satisfies  
\begin{equation}
\overline{x}^B  \Xfrak=\Xfrak\, \overline{x} \qquad \mbox{for all $ x \in \coid$.}\label{eq:Xfrak}
\end{equation}
In view of relation \eqref{eq:quasiR-property1}, the element $\Xfrak \in \widehat{U^+}\subseteq \mathscr{U}$ is an analog of the quasi-R-matrix $R$ for quantum symmetric pairs. For this reason we will call $\Xfrak$ the quasi K-matrix for $\coid$. Examples of quasi K-matrices $\Xfrak$ were first constructed in \cite[Theorems 2.10, 6.4]{a-BaoWang13p} for the coideal subalgebras corresponding to the symmetric pairs $(\slfrak_{2n}, \mathfrak{s}(\mathfrak{gl}_n\times \mathfrak{gl}_n))$ and $(\slfrak_{2n+1}, \mathfrak{s}(\mathfrak{gl}_{n+1}\times \mathfrak{gl}_n))$.

%%%%%%%%%%%%%%%%%%%%%%%%%%%%%%%%%%%%%%%%%%%%%%%%%%5
\subsection{A recursive formula for $\Xfrak$}
%%%%%%%%%%%%%%%%%%%%%%%%%%%%%%%%%%%%%%%%%%%%%%%%
As a first step towards the construction of $\Xfrak$ we translate relation \eqref{eq:Xfrak} into a recursive formula for the components $\Xfrak_\mu$.
%%%%%%%%%%%%%%%%%%%%%%%%%%%%%%%%%%%%%%%%%%%%%%%%%%
\begin{prop}\label{prop:TFAE}
Let $$\Xfrak=\sum_{\mu \in Q^+} \Xfrak_{\mu} \in \widehat{U^+},  \qquad \textrm{with } \Xfrak_\mu \in U^+_\mu.$$ The following are equivalent: 
\begin{enumerate}
\item For all $x \in \coid$ one has $\overline{x}^B  \Xfrak=\Xfrak\, \overline{x}$.
\item For all $i \in I$ one has $\overline{B_i}^B \Xfrak=\Xfrak\, \overline{B_i}$.
\item For all $\mu \in Q^+$ and all $i \in I$ one has
        \begin{align}
           r_i(\Xfrak_\mu) &= -(q_i-q_i^{-1}) \left( \Xfrak_{\mu+\Theta(\alpha_i)-\alpha_i} \overline{c_iX_i} + \overline{s_i} \Xfrak_{\mu-\alpha_i}\right), \label{eq:riX}\\
           {}_ir(\Xfrak_\mu)&=-(q_i-q_i^{-1}) \left(q^{-(\Theta(\alpha_i),\alpha_i)} c_i X_i \Xfrak_{\mu+\Theta(\alpha_i)-\alpha_i}  + s_i \Xfrak_{\mu-\alpha_i}\right). \label{eq:irX}
        \end{align}
\end{enumerate}
If these equivalent conditions hold then additionally
\begin{enumerate}
\item[(4)]  For all $\mu \in Q^+$ such that $\Xfrak_\mu\ne 0$, one has $\Theta(\mu)=-\mu$.
\end{enumerate}
\end{prop}
%%%%%%%%%%%%%%%%%%%%%%%%%%%%%%%%%%%%%%%%%%%%%%%%%%%
\begin{proof}
$\boldsymbol{(1)\Rightarrow (2):}$ Property (2) is the special case $x=B_i$ of property (1).

\noindent $\boldsymbol{(2)\Leftrightarrow (3):}$ Fix $i\in I$. Using the definition \eqref{eq:Bi-def} of $B_i$, definition \eqref{eq:defbarUqg} of $\overline{B_i}$ and definition \eqref{eq:defbarBcs} of $\overline{B_i}^B$, we see that (2) is equivalent to 
\begin{align}\label{eq:Xx=xX}
  \left(F_i+c_i X_i K_i^{-1}+s_iK_i^{-1}  \right) \Xfrak = \Xfrak \left(F_i+  \overline{c_iX_i} K_i+\overline{s_i}K_i  \right).
\end{align}  
Now compare the $(\mu-\alpha_i)$-homogeneous components for all $\mu\in Q^+$. One obtains that Equation \eqref{eq:Xx=xX} holds if and only if for all $\mu \in Q^+$ one has 
\begin{multline*}
[\Xfrak_\mu,F_i]=-\left(\Xfrak_{\mu-\alpha_i+\Theta(\alpha_i)} \overline{c_iX_i} + \overline{s_i} \Xfrak_{\mu-\alpha_i}  \right)K_i  \\
+K_i^{-1}\left(q^{-\left( \alpha_i, \Theta(\alpha_i) \right)} c_iX_i  \Xfrak_{\mu-\alpha_i+\Theta(\alpha_i)} + s_i \Xfrak_{\mu-\alpha_i}  \right).
\end{multline*}
By \eqref{eq:riCommute}, this is equivalent to 
relations (\ref{eq:riX}) and (\ref{eq:irX}) for all $\mu\in Q^+$.

\noindent $\boldsymbol{(3)\Rightarrow (4)}:$
We prove this implication by induction on $\mathrm{ht}(\mu)$. For $\mu=0$ there is nothing to show. Assume that $\mu>0$. If $\Xfrak_\mu\ne 0$, then by \eqref{eq:allrizero}, there exists $i\in I$ such that $r_i(\Xfrak_{\mu})\ne 0$. By (\ref{eq:riX}) we have either $\Xfrak_{\mu+\Theta(\alpha_i)-\alpha_i}\ne 0$ or $s_i \Xfrak_{\mu-\alpha_i}\ne 0$. In the case $\Xfrak_{\mu+\Theta(\alpha_i)-\alpha_i}\ne 0$, by induction hypothesis $\Theta(\mu+\Theta(\alpha_i)-\alpha_i)=-(\mu+\Theta(\alpha_i)-\alpha_i)$, which implies $\Theta(\mu)=-\mu$. In the case $s_i \Xfrak_{\mu-\alpha_i}\ne 0$, the condition $\mathbf{s}\in \mathcal{S}$ implies that $\Theta(\alpha_i)=-\alpha_i$, while the induction hypothesis implies that $\Theta(\mu-\alpha_i)=-(\mu-\alpha_i)$. Together, this gives $\Theta(\mu)=-\mu$.

\noindent $\boldsymbol{(3)\Rightarrow (1):}$ We have already seen that $(3)\Rightarrow (2)$ and hence $\overline{x}^B \Xfrak=\Xfrak\, \overline{x}$ for $x=B_i$. 

Let $\beta \in Q^\Theta$ and assume that $\Xfrak_\mu\ne 0$. The implication $(3)\Rightarrow (4)$ gives $\Theta(\mu)=-\mu$. On the other hand $\Theta(\beta)=\beta$ and therefore $\left(\beta, \mu\right)=0$. This implies that $K_\beta \Xfrak_\mu K_\beta^{-1}=q^{\left(\beta, \mu\right)}\Xfrak_\mu=\Xfrak_\mu$ and consequently $\overline{x}^B \Xfrak=\Xfrak\, \overline{x}$ for all $x \in U^0_\Theta$. 

Finally, let $i \in X$ and again assume that $\Xfrak_\mu\ne 0$. As $K_i \in U^0_\Theta$ and $F_i=B_i$, we already know that $\ad(K_i)(\Xfrak_\mu)= \Xfrak_\mu$ and $\ad(F_i)(\Xfrak_\mu) = 0$. Hence $\Xfrak_\mu$ is the lowest weight vector for the left adjoint action of $U_{q_i}(\mathfrak{sl}_2)_i$ on $\uqg$. As $U^+$ is locally finite for the left adjoint action of $U^+$, we conclude that $\Xfrak_\mu$ is also a highest weight vector, and hence
$$0=\ad(E_i)(\Xfrak_\mu)=E_i \Xfrak_\mu - K_i \Xfrak_\mu K_{i}^{-1}E_i=E_i \Xfrak_\mu - \Xfrak_\mu E_i.$$
Thus $E_i \Xfrak_\mu= \Xfrak_\mu E_i$ and consequently $\overline{x}^B \Xfrak=\Xfrak\, \overline{x}$ for all $x \in \cM_X$. 

This proves that the relation $\overline{x}^B \Xfrak=\Xfrak\, \overline{x}$ holds for the generators of the algebra $\coid$ and hence it holds for all $x\in \coid$.
\end{proof}
%%%%%%%%%%%%%%%%%%%%%%%%%%%%%%%%%%%%%%%%%%%%%%%%
The proof of the implication $(3)\Rightarrow (4)$ only refers to $\Xfrak_{\mu'}$ with $\mu'\le \mu$. Hence we get the following corollary. 
%%%%%%%%%%%%%%%%%%%%%%%%%%%%%%%%%%%%%%%%%%%%%%%%
\begin{cor}\label{cor:TFAE}
Let $\mu \in Q^+$ and let $(\Xfrak_{\mu'})_{\mu'\le\mu \in Q^+}$, with  $\Xfrak_{\mu'} \in U^+_{\mu'}$, be a collection of elements satisfying \eqref{eq:riX} and \eqref{eq:irX} for all $\mu'\le \mu$ and all $i\in I$. If $\Xfrak_\mu\ne 0$ then $\Theta(\mu)=-\mu$.
\end{cor}
%%%%%%%%%%%%%%%%%%%%%%%%%%%%%%%%%%%%%%%%%%%%%%%%

%%%%%%%%%%%%%%%%%%%%%%%%%%%%%%%%%%%%%%%%%%%%%%%%%
\subsection{Systems of equations given by skew derivations}\label{sec:SysEq}
%%%%%%%%%%%%%%%%%%%%%%%%%%%%%%%%%%%%%%%%%%%%%%%%%%
By Proposition \ref{prop:TFAE} the quasi K-matrix  $\Xfrak$ can be constructed inductively if in each step it is possible to solve the system of equations given by (\ref{eq:riX}) and (\ref{eq:irX}) for all $i$. In this subsection we derive necessary and sufficient conditions for such a system to have a solution. 
%%%%%%%%%%%%%%%%%%%%%%%%%%%%%%%%%%%%%%%%5
\begin{prop} \label{prop:Xi}
Let  $\mu \in Q^+$ with $\hght(\mu)\ge 2$ and fix elements $A_i, {}_iA \in U^+_{\mu-\alpha_i}$ for all $i \in I$. The following are equivalent:
\begin{enumerate}
\item There exists an element $\uX\in U_\mu^+$  such that 
\begin{align}\label{eq:riUnknown}
  r_i(\uX)&= A_i  \quad \mbox{and} \quad {}_ir(\uX)= {}_iA \qquad \mbox{for all $i\in I$.}
\end{align}
\item  The elements $A_i, {}_iA$ have the following two properties:
\begin{enumerate}
\item For all $i,j\in I$ one has
\begin{equation}
r_i(_jA)=\, \, {}_jr(A_i);\label{eq:xi2a}
\end{equation}
\item For all $i\ne j\in I$ one has
\begin{multline}
  \frac{-1}{q_i-q_i^{-1}} \sum_{s=1}^{1-a_{ij}} {1-a_{ij} \brack s}_{q_i}(-1)^s \left<F_i^{1-a_{ij}-s}F_jF_i^{s-1},A_i \right>\\
  -\frac{1}{q_j-q_j^{-1}}\left<F_i^{1-a_{ij}},A_j \right>=0. \label{eq:xi2b} \end{multline}
\end{enumerate}
\end{enumerate}
Moreover, if the system of equations \eqref{eq:riUnknown} has a solution $\uX$, then this solution is uniquely determined. 
\end{prop}
%%%%%%%%%%%%%%%%%%%%%%%%%%%%%%%%%%%%%%%%%%%%55
\begin{proof}
$\boldsymbol{(1) \Rightarrow (2):}$
Assume that there exists and element $\uX\in U^+_\mu$ which satisfies the equations \eqref{eq:riUnknown}. Then
\begin{align*}
  r_i({}_jA)=r_i(\, {}_jr(\uX))\stackrel{\eqref{eq:rijr}}{=} \, {}_jr(r_i(\uX))=\, {}_jr(A_i)
\end{align*}  
and hence \eqref{eq:xi2a} holds for all $i,j\in I$.

Moreover, using the quantum Serre relation $S_{ij}(F_i,F_j)=0$ and the properties \eqref{eq:form-ri} of the bilinear form $\left< \cdot,\cdot\right>$, we get
\begin{align*}
0=&\left<S_{ij}(F_i,F_j),\uX\right>\\
=&\sum_{s=0}^{1-a_{ij}} {1-a_{ij} \brack s}_{q_i}(-1)^s\left<F_i^{1-a_{ij}-s}F_jF_i^{s},\uX\right>\\
=&\frac{-1}{q_i-q_i^{-1}} \sum_{s=1}^{1-a_{ij}} {1-a_{ij} \brack s}_{q_i}(-1)^s\left<F_i^{1-a_{ij}-s}F_jF_i^{s-1},A_i \right> - \\
&\qquad -\frac{1}{q_j-q_j^{-1}}\left<F_i^{1-a_{ij}},A_j \right>,
\end{align*}
which proves relation \eqref{eq:xi2b}. Hence property (2) holds.

\noindent $\boldsymbol{(2)\Rightarrow (1):}$ Assume that the elements $A_i, {}_iA$ satisfy the relations \eqref{eq:xi2a} and \eqref{eq:xi2b}. We first solve the system dual to (\ref{eq:riUnknown}) with respect to the bilinear form $\left< \cdot,\cdot\right>$. By slight abuse of notation we consider $\langle\cdot,\cdot\rangle$ as a pairing on $\free\times U^+$ via the canonical projection $\free\rightarrow U^-$ on the first factor. Fix $\mu\in Q^+$ with $\hght(\mu)\ge 2$. As $\mu>0$ there exist uniquely determined linear functionals $\uX^*_{L}, \uX^*_{R}:\free_{\mu} \rightarrow \field(q^{1/d})$ such that
\begin{align}
\uX^*_{L}(\ccf_i z)&=\frac{-1}{q_i-q_{i}^{-1}}\cdot \left<z,{}_iA \right> \label{eq:dualxi1}, \\
\uX^*_{R}(z\ccf_i)&=\frac{-1}{q_i-q_{i}^{-1}}\cdot \left<z,A_i \right> \label{eq:dualxi2}
\end{align}
for all $z\in \free_{\mu-\alpha_i}$. For any $i,j \in I$ and any $x\in \free_{\mu-\alpha_i-\alpha_j}$ we have
\begin{multline*}
\uX^*_{L}(\ccf_jx\ccf_i)\stackrel{\eqref{eq:dualxi1}}{=}\frac{-1}{q_j-q_{j}^{-1}} \cdot  \left<x\ccf_i,{}_jA\right> \stackrel{\eqref{eq:form-ri}}{=} \frac{-1}{q_i-q_{i}^{-1}} \cdot \frac{-1}{q_j-q_{j}^{-1}} \left<x, r_i({}_jA) \right>\\
\stackrel{\eqref{eq:xi2a}}{=} \frac{-1}{q_i-q_{i}^{-1}}\cdot\frac{-1}{q_j-q_{j}^{-1}} \left<x,\, {}_jr(A_i) \right> \stackrel{\eqref{eq:form-ri}}{=}\frac{-1}{q_i-q_{i}^{-1}}\cdot \left<\ccf_j x,A_i \right>\stackrel{\eqref{eq:dualxi2}}{=} \uX^*_{R}(\ccf_jx\ccf_i).
\end{multline*}
As $\hght(\mu)\ge 2$, any element in $\free_{\mu}$ can be written as a linear combination of elements of the form $\ccf_jx\ccf_i$ with $x\in \free_{\mu-\alpha_i-\alpha_j}$ for $i,j\in I$. Consequently, the above relation implies that the functionals $\uX^*_{L}$ and $\uX^*_{R}$ coincide on $\free_{\mu}$. To simplify notation we write $\uX^*=\uX^*_{L}=\uX^*_{R}$.

We claim that relation (\ref{eq:xi2b}) implies that $\uX^*$ descends from $\free_{\mu}$ to a linear functional on $U^-_{-\mu}$. Recall that the kernel of the projection $\free\rightarrow U^-$ is the ideal generated by the elements $S_{ij}(\ccf_i,\ccf_j)$ for all $i,j\in I$. Hence it is enough to show that all elements of the form $x=\ccf_{a_1}\ldots \ccf_{a_l}\cdot S_{ij}(\ccf_i,\ccf_j) \cdot \ccf_{b_1}\ldots \ccf_{b_k}$ lie in the kernel of the linear functional $\uX^*$. If $l>0$ then the fact that $S_{ij}(\ccf_i,\ccf_j)$ lies in the radical of the bilinear form $\left<\cdot,\cdot\right>$ implies that
\begin{align*} 
\uX^*(x)&= \uX_L^*(\ccf_{a_1}\ldots \ccf_{a_l}\cdot S_{ij}(\ccf_i,\ccf_j)\cdot \ccf_{b_1}\ldots \ccf_{b_k})\\
&= \frac{-1}{q_{a_1}-q_{a_1}^{-1}}\cdot \left<\ccf_{a_2}\ldots \ccf_{a_l}\cdot S_{ij}(\ccf_i,\ccf_j)\cdot \ccf_{b_1}\ldots \ccf_{b_k} ,{}_{a_1}A \right>\\
&= 0.
\end{align*}
Similarly, if $k>1$ then we get $\uX^*(x)=\uX_R^*(x)=0$. Assume now that $l=k=0$. Then
\begin{align*} 
\uX^*(S_{ij}(\ccf_i,\ccf_j))&= \sum_{s=0}^{1-a_{ij}} {1-a_{ij} \brack s}_{q_i}(-1)^s \cdot \uX_R^*(\ccf_i^{1-a_{ij}-s}\ccf_j\ccf_i^{s})\\
&= \frac{-1}{q_i-q_i^{-1}} \sum_{s=1}^{1-a_{ij}} {1-a_{ij} \brack s}_{q_i}(-1)^s\left<\ccf_i^{1-a_{ij}-s}\ccf_j\ccf_i^{s-1},A_i \right> \\
&\quad -\frac{1}{q_j-q_j^{-1}}\left<\ccf_i^{1-a_{ij}},A_j \right>\\
& \stackrel{(\ref{eq:xi2b})}{=} 0.
\end{align*}
Hence $\uX^*$ does indeed descend to a linear functional $\uX^*:U_{-\mu}^-\to \mathbb{K}(q^{1/d})$.

Let $\uX\in U_{\mu}^+$ be the element dual to  $\uX^*$ with respect to the nondegenerate pairing $\left< \cdot ,\cdot \right>$ on $U^-_\mu\times U^+_\mu$. In other words, for all $z\in U^-$ we have
$\uX^*(z)=\left<z,\uX\right>$. Then 
$$\left<z,r_i(\uX) \right>=-(q_i-q_i^{-1})\left<zF_i,\uX\right>=-(q_i-q_i^{-1})\uX^*(zF_i)\stackrel{(\ref{eq:dualxi2})}{=}\left<z,A_i \right>$$
for any $z\in U^-_{\mu-\alpha_i}$ and hence $r_i(\uX)=A_i$ for all $i\in I$. Similarly, (\ref{eq:dualxi1}) implies that ${}_ir(\uX)={}_iA$ for all $i\in I$. This completes the proof of relation \eqref{eq:riUnknown} and hence (1) holds.

Finally, to see uniqueness, assume that $\uX$ and $\uX'$ both satisfy the system of equations \eqref{eq:riUnknown}. Then $r_i(\uX-\uX')=0$ for all $i\in I$, so by \eqref{eq:allrizero}, we have that $\uX=\uX'$.
\end{proof}
%%%%%%%%%%%%%%%%%%%%%%%%%%%%%%%%

%%%%%%%%%%%%%%%%%%%%%%%%%%%%%%%%
\subsection{Three technical lemmas}
%%%%%%%%%%%%%%%%%%%%%%%%%%%%%%%%
We will use Proposition \ref{prop:Xi} in Section \ref{sec:constructXmu} to inductively construct $\Xfrak_\mu$ by solving the system of equations given by \eqref{eq:riX}, \eqref{eq:irX} for all $i\in I$. To simplify the proof that the right hand sides of equations \eqref{eq:riX}, \eqref{eq:irX} satisfy the conditions from Proposition \ref{prop:Xi}.(2), we provide several technical lemmas. These results are auxiliary and will only be used in the proof of Lemma \ref{lem:Xcond2}.
%%%%%%%%%%%%%%%%%%%%%%%%%%%%%%%%%
\begin{lem}\label{lem:technical1}
Let $i\ne j\in I$ and $\mu=(1-a_{ij})\alpha_i+\alpha_j$. If $\Theta(\mu)=-\mu$ then $i,j\in I\setminus X$ and one of the following two cases holds: 
\begin{enumerate}
\item $\Theta(\alpha_i)=-\alpha_j$ and $a_{ij}=0$.
\item $\Theta(\alpha_i)=-\alpha_i$ and $\Theta(\alpha_j)=-\alpha_j $.
\end{enumerate}
\end{lem}
%%%%%%%%%%%%%%%%%%%%%%%%%%%%%%%%%
\begin{proof}
Assume that $i \in X$. Then $\Theta(\alpha_i)=\alpha_i$ which together with $\Theta(\mu)=-\mu$ implies that 
\begin{equation*}
w_X(\alpha_{\tau(j)})=-\Theta(\alpha_j)=-\Theta(\mu-(1-a_{ij})\alpha_i)=\alpha_j+2(1-a_{ij})\alpha_i.
\end{equation*}
Hence $\tau(j)=j$ and $\sigma_i(\alpha_j)=w_X(\alpha_{j})$ and $-a_{ij}=2(1-a_{ij})$. This would mean that $a_{ij}=2$ which is impossible. 

Assume that $j\in X$. Then 
$$w_X(\alpha_{\tau(i)})=-\Theta(\alpha_i)=\frac{-1}{(1-a_{ij})}\Theta(\mu-\alpha_j)=\alpha_i +\frac{2}{(1-a_{ij})}\alpha_j.$$ 
Hence $\tau(i)=i$ and $\sigma_j(\alpha_i)=w_X(\alpha_{i})$ and $a_{ji}=-\frac{2}{(1-a_{ij})}$. This is only possible if $a_{ji}=a_{ij}=-1$. But then $$\alpha_i(\rho_X^\vee)=\frac{1}{2}\alpha_i(h_j)=\frac{-1}{2}\notin \mathbb{Z}$$ 
which contradicts condition (3) in Definition \ref{def:admissible} of an admissible pair. 

Hence $i,j \in I \setminus X$. As $(w_X-\mathrm{id})(\alpha_k)\in Q_X$ for any $k \in I$, it follows that
\begin{equation*}
(1-a_{ij})(\alpha_i-\alpha_{\tau(i)})+(\alpha_j-\alpha_{\tau(j)})=-\Theta(\mu)-\tau(\mu)=(w_X-\mathrm{id})(\tau(\mu))
\end{equation*}
lies in $Q_X$. 
Using $i,j\in I\setminus X$, it follows that $(1-a_{ij})(\alpha_i-\alpha_{\tau(i)})+(\alpha_j-\alpha_{\tau(j)})=0$.
So, there are two possibilities: either (1) $\tau(i)=j$ and $a_{ij}=0$, or (2) $\tau(i)=i$ and $\tau(j)=j$. \end{proof}
%%%%%%%%%%%%%%%%%%%%%%%%%%%%%%%%%%%

%%%%%%%%%%%%%%%%%%%%%%%%%%%
\begin{lem}\label{lem:technical2}
Let $\mu\in Q^+$ and let $j\in I\setminus X$ with $s_j=0$. Assume that a collection      
$(\Xfrak_{\mu'})_{\mu'\le \mu}$ with $\Xfrak _{\mu'} \in U^+_{\mu'}$ satisfies condition (\ref{eq:riX}) for all $\mu'\le \mu$ and for all $i \in I$. If $\Xfrak _{\mu}\ne 0$, then $\mu \in \mathrm{span}_{\mathbb{N}_0} \{ \alpha_j-\Theta(\alpha_j)\}\oplus \mathrm{span}_{\mathbb{N}_0} \{ \alpha_k | k\ne j\}$.
\end{lem}
%%%%%%%%%%%%%%%%%%%%%%%%%%%
\begin{proof}
We prove this by induction on $\mathrm{ht}(\mu)$. If $\mu>0$ and $\Xfrak _{\mu}\ne 0$ then by \eqref{eq:allrizero} there exists some $i$ such that $r_i(\Xfrak _{\mu})\ne 0$. Relation (\ref{eq:riX}) implies that $\Xfrak _{\mu+\Theta(\alpha_i)-\alpha_i}\ne 0$ or $s_i\Xfrak _{\mu-\alpha_i}\ne 0$. If $i\ne j$ then the induction hypothesis on $\mu+\Theta(\alpha_i)-\alpha_i$ and $\mu-\alpha_i$ implies the claim. If $i=j$ then the induction hypothesis on $\mu+\Theta(\alpha_i)-\alpha_i$ implies the claim.
\end{proof}
%%%%%%%%%%%%%%%%%%%%%%%%%%%%%%
Recall that $\sigma$ denotes the involutive antiautomorphism of $\uqg$ defined by \eqref{def:sigma}.
%%%%%%%%%%%%%%%%%%%%%%%%%%%%%%
\begin{lem}\label{lem:auxSerre}
Let $\nu \in Q^+$, and let $(\Xfrak_\mu)_{\mu < \nu\in Q^+}$ be a collection with $\Xfrak_\mu \in U^+_{\mu}$ and  $\Xfrak_0=1$. For all $\mu< \nu$ assume that $\Xfrak_\mu$ satisfies \eqref{eq:riX} and \eqref{eq:irX} for all $i\in I$. Let $j,k\in I\setminus X$ be such that $\Theta(\alpha_j)=-\alpha_j$ and $\Theta(\alpha_k)=-\alpha_k$. Assume that $n\ge 0$, and that $x\in U^-_{n\alpha_k+\alpha_j}$ satisfies $\sigma(x)=-x$. Then
\begin{align}
\left<x,\Xfrak_\mu\right>&=0 \label{eq:auxSerre3}
\end{align}
for all $\mu< \nu$.
\end{lem}
%%%%%%%%%%%%%%%%%%%%%%%%%%%%%%

\begin{proof}
The space $U^-_{n\alpha_k+\alpha_j}$ is spanned by elements of the form $F_k^{a}F_j F_k^{b}$ with $a+b=n$. As the antiautomorphism $\sigma$ is involutive it is enough to verify Equation \eqref{eq:auxSerre3} for elements of the form $x=F_k^{a}F_j F_k^{b}-\sigma(F_k^{a}F_j F_k^{b})=F_k^{a}F_j F_k^{b}-F_k^{b}F_j F_k^{a}.$
We will prove that
\begin{align}
  \left<F_k^{a}F_jF_k^{b}-F_k^{b}F_jF_k^{a},\Xfrak_\mu\right>&=0 \label{eq:auxSerre3.1}
\end{align}
for all $\mu< \nu$ and $a,b\ge 0$ by induction on $n=a+b$. It holds for $n=0$. Let $a+b=n>0$, and assume that \eqref{eq:auxSerre3.1} holds for all $a',b'$ with $a'+b'<n$. Without loss of generality assume that $b>0$. Using the assumption that  $\Xfrak_\mu$ satisfies \eqref{eq:riX} and \eqref{eq:irX}, we get that 
\begin{align}
  \left<F_k^{a}F_jF_k^{b}{-}F_k^{b}F_jF_k^{a},\Xfrak_\mu\right>&=\frac{-1}{q_k-q_{k}^{-1}}  
   \left(\left<F_k^{a}F_jF_k^{b-1},r_k(\Xfrak_\mu)\right>-   
   \left<F_k^{b-1}F_jF_k^{a},{}_kr(\Xfrak_\mu)\right> \right) \notag \\
  &=\left<F_k^{a}F_jF_k^{b-1},\Xfrak_{\mu-2\alpha_k} \overline{c_kX_k} + \overline{s_k}   
   \Xfrak_{\mu-\alpha_k}\right>\notag\\
  & \quad -\left<F_k^{b-1}F_jF_k^{a},q^{(\alpha_k,\alpha_k)} c_k X_k \Xfrak_{\mu-2\alpha_k} 
   +s_k \Xfrak_{\mu-\alpha_k}\right>.\notag
\end{align}
The assumption $\Theta(\alpha_k)=-\alpha_k$ implies that $X_k=-E_k$ and by \eqref{eq:octau} and \eqref{Mparameters2} one has $\overline{c_k}=q^{(\alpha_k,\alpha_k)}c_k$ and $\overline{s_k}=s_k$. Hence the above equation turns into
\begin{align}
  \left<F_k^{a}F_jF_k^{b}-F_k^{b}F_jF_k^{a},\Xfrak_\mu\right>&=s_k 
  \left<F_k^{a}F_jF_k^{b-1}-F_k^{b-1}F_jF_k^{a} , \Xfrak_{\mu-\alpha_k}\right>\notag\\
  -&\overline{c_k}\left(\left<F_k^{a}F_jF_k^{b-1},\Xfrak_{\mu-2\alpha_k}E_k\right>-
  \left<F_k^{b-1}F_jF_k^{a},E_k\Xfrak_{\mu-2\alpha_k}\right>\right) \notag
\end{align}
By the induction hypothesis one has $\left<F_k^{a}F_jF_k^{b-1}-F_k^{b-1}F_jF_k^{a} , \Xfrak_{\mu-\alpha_k}\right>=0$. Hence, 
\begin{align*}
  \left<F_k^{a}F_jF_k^{b}{-}F_k^{b}F_jF_k^{a},\Xfrak_\mu\right>&=
  \frac{\overline{c_k}}{q_k{-}q_k^{-1}}  
  \left<r_k(F_k^{a}F_jF_k^{b-1}){-}{}_kr(F_k^{b-1}F_jF_k^{a}),\Xfrak_{\mu-2\alpha_k}\right>. 
\end{align*}
As 
\begin{align}
\sigma(r_k(F_k^{a}F_jF_k^{b-1})-{}_kr(F_k^{b-1}F_jF_k^{a}))&\stackrel{\eqref{eq:sigmari}}{=}{}_kr(\sigma(F_k^{a}F_jF_k^{b-1}))-r_k (\sigma(F_k^{b-1}F_jF_k^{a})) \notag\\
&=-(r_k(F_k^{a}F_jF_k^{b-1})-{}_kr(F_k^{b-1}F_jF_k^{a})), \notag
\end{align}
Equation \eqref{eq:auxSerre3.1} follows from the induction hypothesis. 
\end{proof}
%%%%%%%%%%%%%%%%%%%%%%%%%%%%%%%%%%%%

%%%%%%%%%%%%%%%%%%%%%%%%%%%%%%%%%%%%%%%%%%%%%%%%%%%
\subsection{Constructing $\Xfrak_\mu$}\label{sec:constructXmu}
%%%%%%%%%%%%%%%%%%%%%%%%%%%%%%%%%%%%%%%%%%%%%%%%%%%
We are now ready to construct $\Xfrak_\mu$ inductively. Fix $\mu \in Q^+$ and assume that a collection $(\Xfrak _{\mu'})_{\mu'<\mu \in Q^+}$ with  $\Xfrak_{\mu'}\in U^+_{\mu'}$ and $\Xfrak_0=1$ has already been constructed and that this collection satisfies
conditions (\ref{eq:riX}) and (\ref{eq:irX}) for all $\mu'<\mu$ and for all $i \in I$. Define 
\begin{align}
         A_i&=-(q_i-q_i^{-1}) \left(  \Xfrak_{\mu+\Theta(\alpha_i)-\alpha_i} \overline{c_iX_i} + \overline{s_i} \Xfrak_{\mu-\alpha_i}\right) , \label{eq:Ai}\\
         {}_iA&=-(q_i-q_i^{-1}) \left(q^{-(\Theta(\alpha_i),\alpha_i)}c_iX_i \Xfrak_{\mu+\Theta(\alpha_i)-\alpha_i}  + s_i \Xfrak_{\mu-\alpha_i}\right) \label{eq:iA}
\end{align}
for all $i\in I$. We will keep the above assumptions and the definition of $A_i$ and ${}_iA$ all through this subsection. We will prove that the elements $A_i,{}_iA$, which are the right hand sides of equations \eqref{eq:riX}, \eqref{eq:irX}, satisfy the conditions \eqref{eq:xi2a} and \eqref{eq:xi2b}. By Proposition \ref{prop:Xi} this will prove the existence of an element $\Xfrak_\mu$ with the desired properties.
%%%%%%%%%%%%%%%%%%%%%%%%
\begin{lem}\label{lem:Xcond1}
The relation  $r_i({}_jA)=\, {}_jr(A_i)$ holds for all $i,j\in I$.
\end{lem}
%%%%%%%%%%%%%%%%%%%%%%%%%
\begin{proof}
  This is a direct calculation. Note that all computations include the case $i=j$. We expand both sides of the desired equation, using \eqref{def:ri} and \eqref{def:ir} and the assumption that the elements $\Xfrak _{\mu'}$ satisfy (\ref{eq:riX}) and (\ref{eq:irX}) for $\mu'<\mu$. We obtain
\begin{align*}
  r_i(_jA)&=-(q_j-q_j^{-1})q^{-(\Theta(\alpha_j),\alpha_j)}
           c_jX_jr_i( \Xfrak_{\mu+\Theta(\alpha_j)-\alpha_j})\\
          &\quad-(q_j-q_j^{-1}) 
          q^{-(\Theta(\alpha_j),\alpha_j)}q^{(\alpha_i,\mu+\Theta(\alpha_j)-\alpha_j)}r_i(c_jX_j) 
          \Xfrak_{\mu+\Theta(\alpha_j)-\alpha_j}\\
          &\quad-(q_j-q_j^{-1})s_j r_i\left(\Xfrak_{\mu-\alpha_j}\right)\\
          &=(q_j-q_j^{-1})q^{-(\Theta(\alpha_j),\alpha_j)}c_jX_j(q_i-q_i^{-1}) 
          \Xfrak_{\mu+\Theta(\alpha_j)-\alpha_j+\Theta(\alpha_i)-\alpha_i} \overline{c_iX_i}\\
          & \quad+(q_j-q_j^{-1})q^{-(\Theta(\alpha_j),\alpha_j)}c_jX_j(q_i-q_i^{-1})\overline{s_i} 
          \Xfrak_{\mu+\Theta(\alpha_j)-\alpha_j-\alpha_i}\\ 
          &\quad-(q_j-q_j^{-1}) 
          q^{-(\Theta(\alpha_j),\alpha_j)}q^{(\alpha_i,\mu+\Theta(\alpha_j)-\alpha_j)}r_i(c_jX_j) 
          \Xfrak_{\mu+\Theta(\alpha_j)-\alpha_j}\\
          &\quad+(q_j-q_j^{-1})s_j(q_i-q_i^{-1})\Xfrak_{\mu-\alpha_j+\Theta(\alpha_i)-\alpha_i} 
          \overline{c_iX_i}\\
          &\quad+(q_j-q_j^{-1})s_j(q_i-q_i^{-1})\overline{s_i} \Xfrak_{\mu-\alpha_i-\alpha_j},
\end{align*}
\begin{align*}
{}_jr(A_i)&=-(q_i-q_i^{-1})\, {}_jr(\Xfrak_{\mu+\Theta(\alpha_i)-\alpha_i}) \overline{c_iX_i}\\
          &\quad-(q_i-q_i^{-1}) q^{(\alpha_j, 
          \mu+\Theta(\alpha_i)-\alpha_i)}\Xfrak_{\mu+\Theta(\alpha_i)-\alpha_i}\, 
          {}_jr(\overline{c_iX_i})\\
          &\quad-(q_i-q_i^{-1})\overline{s_i} \,{}_jr( \Xfrak_{\mu-\alpha_i})\\
          &=(q_i-q_i^{-1})(q_j-q_j^{-1})q^{-(\Theta(\alpha_j),\alpha_j)}
          c_jX_j\Xfrak_{\mu+\Theta(\alpha_i)-\alpha_i+\Theta(\alpha_j)-\alpha_j} \overline{c_iX_i}\\
          &\quad +(q_i-q_i^{-1})(q_j-q_j^{-1})s_j 
          \Xfrak_{\mu+\Theta(\alpha_i)-\alpha_i-\alpha_j}\overline{c_iX_i}\\
          &\quad-(q_i-q_i^{-1})q^{(\alpha_j, 
          \mu+\Theta(\alpha_i)-\alpha_i)}\Xfrak_{\mu+\Theta(\alpha_i)-\alpha_i}\, 
          {}_jr(\overline{c_iX_i})\\
          &\quad+(q_i-q_i^{-1})\overline{s_i} (q_j-q_j^{-1}) q^{-(\Theta(\alpha_j),\alpha_j)}c_jX_j  
          \Xfrak_{\mu-\alpha_i+\Theta(\alpha_j)-\alpha_j}\\
          &\quad+(q_i-q_i^{-1})\overline{s_i} (q_j-q_j^{-1}) s_j \Xfrak_{\mu-\alpha_i-\alpha_j}.
\end{align*}
We see that the first and fifth summands in the above expansions of $r_i(_jA)$ and $\, \, {}_jr(A_i)$ coincide, the second summand of $r_i(_jA)$ is the same as the fourth summand of $\, \, {}_jr(A_i)$, and the fourth summand of $r_i(_jA)$ coincides with the second summand of $\, \, {}_jr(A_i)$. Therefore, the claim of the Lemma, $r_i(_jA)={}_jr(A_i)$,  is equivalent to the third summands being equal, 
\begin{multline}
-(q_j-q_j^{-1}) q^{-(\Theta(\alpha_j),\alpha_j)}q^{(\alpha_i,\mu+\Theta(\alpha_j)-\alpha_j)}r_i(c_jX_j) \Xfrak_{\mu+\Theta(\alpha_j)-\alpha_j}\\
=-(q_i-q_i^{-1})q^{(\alpha_j, \mu+\Theta(\alpha_i)-\alpha_i)}\Xfrak_{\mu+\Theta(\alpha_i)-\alpha_i}\, {}_jr(\overline{c_iX_i}). \label{eq:Xcond1.0}
\end{multline}
By \eqref{eq:rj(Xi)} and \eqref{eq:ribar} we may assume that $i=\tau(j)\in I\setminus X$ because otherwise both sides of the above equation vanish. By \eqref{eq:ribar} we have $\, {}_jr\left( \overline{c_iX_i} \right)=q^{(\alpha_j,-\Theta(\alpha_i)-\alpha_j)} \overline{r_j (c_iX_i)}$. Substituting this and using $q_i=q_j$, we see that \eqref{eq:Xcond1.0} is equivalent to 
\begin{multline}
q^{(\alpha_j,\Theta(\alpha_i)-\Theta(\alpha_j))+(\alpha_i,\mu-\alpha_j)}r_i(c_jX_j) \Xfrak_{\mu+\Theta(\alpha_j)-\alpha_j}\\
=q^{(\alpha_j, \mu-\alpha_i-\alpha_j)}\Xfrak_{\mu+\Theta(\alpha_i)-\alpha_i} \overline{r_j (c_iX_i)}. \label{eq:Xcond1.1}
\end{multline}
By Equation \ref{eq:alphai-alphataui} one has $\Theta(\alpha_i)-\alpha_i=\Theta(\alpha_j)-\alpha_j$ and hence $\Xfrak_{\mu+\Theta(\alpha_j)-\alpha_j}=\Xfrak_{\mu+\Theta(\alpha_i)-\alpha_i}$. Moreover, $r_i(T_{w_X}(E_i))$ lies in $\cM_X$ and hence it commutes with $\Xfrak_{\mu+\Theta(\alpha_i)-\alpha_i}$. Using this, we can rewrite \eqref{eq:Xcond1.1} as
\begin{multline}
q^{(\alpha_j,\Theta(\alpha_i-\alpha_j))+(\alpha_i,\mu)}\Xfrak_{\mu+\Theta(\alpha_i)-\alpha_i} r_i(c_jX_j) \\
=q^{(\alpha_j, \mu-\alpha_j)}\Xfrak_{\mu+\Theta(\alpha_i)-\alpha_i} \overline{r_j (c_iX_i)}. \label{eq:Xcond1.2}
\end{multline}
If $\Xfrak_{\mu+\Theta(\alpha_i)-\alpha_i}=0$ then both sides of the above equation vanish. Hence assume that $\Xfrak_{\mu+\Theta(\alpha_i)-\alpha_i}$ is nonzero. Corollary \ref{cor:TFAE} states that then $\Theta(\mu)=-\mu$. Along with $\Theta(\alpha_i-\alpha_j)=\alpha_i-\alpha_j$, this implies that $(\alpha_i-\alpha_j,\mu)=0$. Hence \eqref{eq:Xcond1.2} is equivalent to the relation
\begin{equation}
q^{(\alpha_j,\alpha_i)}r_i(c_jX_j) =\overline{r_j (c_iX_i)}. \label{eq:Xcond1.3}
\end{equation}
Using the definition \eqref{def:Zi} of $\mathcal{Z}_i$ the above formula follows from assumption \eqref{Mparameters1} about the parameters $\bc$.
\end{proof}
%%%%%%%%%%%%%%%%%%%%%%%%%%%%%%%%%%%%%%%%%%%%%%%%%
This proves that the elements $A_i,{}_iA$ satisfy the first condition from Proposition \ref{prop:Xi}.(2). Next we prove that they also satisfy the second condition. 
%%%%%%%%%%%%%%%%%%%%%%%%%%%%%%%%%%%%%%%%%%%%%%%%
\begin{lem}\label{lem:Xcond2}
 For all $i\ne j\in I$ the elements $A_i, A_j$ given by \eqref{eq:Ai} satisfy the relation
\begin{multline}\frac{-1}{q_i-q_i^{-1}} \sum_{s=1}^{1-a_{ij}} {1-a_{ij} \brack s}_{q_i}(-1)^s\left<F_i^{1-a_{ij}-s}F_jF_i^{s-1},A_i \right>-\\-\frac{1}{q_j-q_j^{-1}}\left<F_i^{1-a_{ij}},A_j \right>=0. \label{eq:Xcond2} \end{multline}
\end{lem}
%%%%%%%%%%%%%%%%%%%%%%%%%%%%%%%%%%%%%%%%%%%%%%%%
\begin{proof}

We may assume that $\mu=(1-a_{ij})\alpha_i+\alpha_j$  and that $\Theta(\mu)=-\mu$, as otherwise all terms in the above sum vanish. By Lemma \ref{lem:technical1} it suffices to consider the following two cases. 

\noindent{\bf Case 1:} $\Theta(\alpha_i)=-\alpha_j$ and $a_{ij}=0$. In this case $\mu=\alpha_i+\alpha_j$ and $s_i=s_j=0$ by definition \eqref{def:setS} of the parameter set $\mathcal{S}$. Hence $A_i=-s(j)(q_i-q_i^{-1}) \overline{c_i} E_j$ and $A_j=-s(i)(q_j-q_j^{-1})\overline{c_j}E_i$. Therefore the left hand side of (\ref{eq:Xcond2}) is equal to 
\begin{align}\label{eq:nearly}
\frac{1}{q_i-q_i^{-1}} \left<F_j,A_i \right>-\frac{1}{q_j-q_j^{-1}} \left<F_i,A_j \right>&=- s(j)\overline{c_i}\left<F_j,E_j \right>+ s(i)\overline{c_j}\left<F_i,E_i \right>.
\end{align}
Using $q_i=q_j$, the fact that $s(i)=s(j)$ by \eqref{eq:defs(i)2}, and the relation $c_i=c_j$ which holds by definition of the parameter set $\mathcal{C}$, one sees that the right hand side of \eqref{eq:nearly} vanishes.
 
\noindent{\bf Case 2:} $\Theta(\alpha_i)=-\alpha_i$ and $\Theta(\alpha_j)=-\alpha_j$. In this case by \eqref{def:setIns} one has $i,j \in I_{ns}$. Hence, by definition \eqref{def:setS} of the parameter set $\mathcal{S}$, one has either $s_j=0$ or $a_{ij}\in -2\mathbb{N}_0$. If $s_j=0$ then Lemma \ref{lem:technical2} implies that $(1-a_{ij})\alpha_i+\alpha_j= \mu\in \mathbb{N}_0 \alpha_i \oplus 2\mathbb{N}_0 \alpha_j$, which is not the case. If $-a_{ij}$ is even then the left hand side of \eqref{eq:Xcond2} can be written as
\begin{multline}
\frac{1}{q_i-q_i^{-1}}\left<F_jF_i^{-a_{ij}},A_i \right>-\frac{1}{q_j-q_j^{-1}}\left<F_i^{1-a_{ij}},A_j \right> \\
+\frac{-1}{q_i-q_i^{-1}} \sum_{s=1}^{-a_{ij}/2} {1-a_{ij} \brack s}_{q_i}(-1)^s \left<F_i^{1-a_{ij}-s}F_jF_i^{s-1}-F_i^{s}F_jF_i^{-a_{ij}-s},A_i \right>.\label{eq:Xcond2.1}
\end{multline}
By \eqref{eq:Ai} and \eqref{eq:iA} one has $A_j=-(q_j-q_j^{-1})^{-1}s_j \Xfrak_{\mu-\alpha_j} ={}_jA$ and hence
\begin{align*}
&\frac{1}{q_i-q_i^{-1}}\left<F_jF_i^{-a_{ij}},A_i \right>-\frac{1}{q_j-q_j^{-1}}\left<F_i^{1-a_{ij}},A_j \right> \\
&\qquad=-\frac{1}{q_i-q_i^{-1}}\frac{1}{q_j-q_j^{-1}}\left<F_i^{-a_{ij}},{}_jr(A_i)-r_i({}_jA)\right>.
\end{align*}
In view of Lemma \ref{lem:Xcond1} the above relation shows that the sum of the first two terms of \eqref{eq:Xcond2.1} vanishes. Each of the remaining summands in \eqref{eq:Xcond2.1} contains a factor of the form
\begin{align}
&\big<F_i^{1-a_{ij}-s}F_jF_i^{s-1}-F_i^{s}F_jF_i^{-a_{ij}-s},A_i \big> \notag \\
&\qquad=\frac{-1}{q_i-q_i^{-1}}\big<F_i^{-a_{ij}-s}F_jF_i^{s-1}-F_i^{s-1}F_jF_i^{-a_{ij}-s},{}_ir(A_i) \big>. \label{eq:Xcond2.2} 
\end{align}
Set $x=F_i^{-a_{ij}-s}F_jF_i^{s-1}-F_i^{s-1}F_jF_i^{-a_{ij}-s}$ and observe that $\sigma(x)=-x$.
Inserting the definition of $A_i$ into \eqref{eq:Xcond2.2} one obtains in view of $X_i=-E_i$ the relation
\begin{align*}
  \big<F_i^{1-a_{ij}-s}F_jF_i^{s-1}-F_i^{s}F_jF_i^{-a_{ij}-s},A_i \big> =
  \big<x, {}_ir(-\Xfrak_{\mu-2\alpha_i} \overline{c_i}E_i + \overline{s_i}  
  \Xfrak_{\mu-\alpha_i})\big>.
\end{align*}
Using the skew derivation property \eqref{def:ir} and the assumption that $\Xfrak_{\mu'}$ satisfies  \eqref{eq:irX} for all $\mu'<\mu$ one obtains
\begin{align*}
  &\big<F_i^{1-a_{ij}-s}F_jF_i^{s-1}-F_i^{s}F_jF_i^{-a_{ij}-s},A_i \big>\\
  &\qquad=-(q_i-q_i^{-1})\big<x,\overline{c_i}q^{(\alpha_i,\alpha_i)}c_i E_i\Xfrak_{\mu-4\alpha_i}E_i \big>\notag \\ 
& \qquad \quad +(q_i-q_i^{-1}) \big<x,s_i\overline{c_i}\Xfrak_{\mu-3\alpha_i}E_i \big>-\big<x,\overline{c_i}q^{(\alpha_i,\mu-2\alpha_i)}\Xfrak_{\mu-2\alpha_i} \big>\notag \\
&\qquad\quad +(q_i-q_i^{-1}) \big<x,s_iq^{(\alpha_i,\alpha_i)}c_iE_i\Xfrak_{\mu-3\alpha_i} \big>-(q_i-q_i^{-1})  \big<x,s_i^2 \Xfrak_{\mu-2\alpha_i} \big>
\end{align*}
Using the relations \eqref{eq:form-riE} and the property $\overline{c_i}=q^{(\alpha_i,\alpha_i)}c_i$ which holds by \eqref{eq:octau}, the above equation becomes
\begin{align}
&\big<F_i^{1-a_{ij}-s}F_jF_i^{s-1}-F_i^{s}F_jF_i^{-a_{ij}-s},A_i \big> \notag \\
&\qquad=-\overline{c_i}q^{(\alpha_i,\alpha_i)}c_i\frac{1}{q_i-q_i^{-1}} \big<r_i({}_ir(x)), \Xfrak_{\mu-4\alpha_i}\big>-\overline{c_i}q^{(\alpha_i,\mu-2\alpha_i)} \big<x,\Xfrak_{\mu-2\alpha_i} \big>  \label{eq:Xcond2.3}\\
&\qquad\quad  -s_i\overline{c_i} \big<r_{i}(x)+{}_ir(x),\Xfrak_{\mu-3\alpha_i} \big>-(q_i-q_i^{-1})s_i^2   \big<x,\Xfrak_{\mu-2\alpha_i} \big>.\notag
\end{align}
Using the fact that $\sigma(x)=-x$ we obtain from \eqref{eq:sigmari} and \eqref{eq:rijr} that
\begin{align*}
\sigma(r_i({}_ir(x)))&=-r_i({}_ir(x)), &  \sigma(r_{i}(x)+{}_ir(x))&=-(r_{i}(x)+{}_ir(x)).
\end{align*}
By Lemma \ref{lem:auxSerre} the above relations imply that all terms in \eqref{eq:Xcond2.3} vanish. Therefore all summands in \eqref{eq:Xcond2.1} vanish, which completes the proof of the Lemma in the second case.
\end{proof}
%%%%%%%%%%%%%%%%%%%%%%%%%%%%%%%%%%%%%
\begin{rema}
  If one restricts to quantum symmetric pair coideal subalgebras $\coid$ with $\bs=(0,0,\dots,0)$ then Case 2 in the proof of Lemma \ref{lem:Xcond2} simplifies significantly and Lemma \ref{lem:auxSerre} is not needed.
\end{rema}
%%%%%%%%%%%%%%%%%%%%%%%%%%%%%%%%%%%%%
\subsection{Constructing $\Xfrak$}\label{sec:constructingX}
%%%%%%%%%%%%%%%%%%%%%%%%%%%%%%%%%%%%%%
We are now ready to prove the main result of this section, namely the existence of the quasi K-matrix $\Xfrak$. Recall the assumptions from Section \ref{sec:Assumptions}.
%%%%%%%%%%%%%%%%%%%%%%%%%%%%%%%%%%%%%%%%%%%%%%%%
\begin{thm}\label{thm:Xfrak}
There exists a uniquely determined element $\Xfrak=\sum_{\mu\in Q^+}\Xfrak_\mu \in \widehat{U^+}$, with $\Xfrak_0=1$ and $\Xfrak_\mu \in U^+_\mu $, such that the equality 
\begin{equation}
\overline{x}^B \Xfrak=\Xfrak\, \overline{x} \label{eq:Xfrakthm}
\end{equation}
holds in $\mathscr{U}$ for all $ x \in \coid$.
\end{thm}
%%%%%%%%%%%%%%%%%%%%%%%%%%%%%%%%%%%%%%%%%%%%%%%%
\begin{proof}
We construct $\Xfrak_\mu$ by induction on  the height of $\mu$, starting from $\Xfrak_0=1$. If $\mu=\alpha_j$ then equations \eqref{eq:riX} and \eqref{eq:irX} are equivalent to
\begin{align*}
  r_i(\Xfrak_\mu) = {}_ir(\Xfrak_\mu)=\begin{cases} 0& \mbox{if $i\neq j$,}\\
                                                    -(q_i-q_i^{-1})s_i & \mbox{if $i=j$}
                                      \end{cases}              
\end{align*}
as $s_j=\overline{s}_j$ by \eqref{Mparameters2}. In this case $\Xfrak_{\alpha_j}= -(q_j-q_j^{-1})s_j E_j$ satisfies  \eqref{eq:riX} and \eqref{eq:irX}. This defines $\Xfrak_\mu$ in the case $\hght(\mu)=1$.
Assume now that $\hght(\mu)\ge 2$ and that the elements $\Xfrak_{\mu'}$ have been defined for all $\mu'$ with $\mathrm{ht}(\mu')<\mathrm{ht}(\mu)$ such that they satisfy \eqref{eq:riX} and \eqref{eq:irX} for all $i\in I$. The elements $A_i$ and ${}_iA$ given by (\ref{eq:Ai}) and (\ref{eq:iA}), respectively, are then well defined, and by Lemma \ref{lem:Xcond1} and Lemma \ref{lem:Xcond2} they satisfy the conditions of Proposition \ref{prop:Xi}.(2). By Proposition \ref{prop:Xi} the system of equations given by  \eqref{eq:riUnknown} for all $i\in I$ has a unique solution $\uX=\Xfrak_\mu\in U^+_\mu$. By definition of $A_i$ and ${}_iA$ the element $\Xfrak_\mu$ satisfies equations (\ref{eq:riX}) and (\ref{eq:irX}).

Set $\Xfrak=\sum_{\mu\in Q^+}\Xfrak_\mu\in \widehat{U^+}$. By Proposition \ref{prop:TFAE} the element $\Xfrak$ satisfies the relation (\ref{eq:Xfrakthm}) for all $x\in \coid$. 
The uniqueness of $\Xfrak$ follows by Propositions \ref{prop:TFAE} and \ref{prop:Xi} from the uniqueness of the solution of the system of equations given by \eqref{eq:riUnknown} for all $i\in I$.
\end{proof}
%%%%%%%%%%%%%%%%%%%%%%%%%%%%%%%%%%%%%%%%%%%%%%%

%%%%%%%%%%%%%%%%%%%%%%%%%%%%%%%%%%%%%%%%%%%%%%%%%%%%%%%%%%%%
\section{Construction of the universal K-matrix}\label{sec:K-construction}
%%%%%%%%%%%%%%%%%%%%%%%%%%%%%%%%%%%%%%%%%%%%%%%%%%%%%%%%%%%%
Using the quasi K-matrix $\Xfrak$ from the previous section we now construct a candidate $\cK\in\sU$ for a universal K-matrix as in Definition \ref{def:U-cylinder-braided}. Our approach is again inspired by the special case considered in \cite{a-BaoWang13p}. However, we are aiming for a comprehensive construction for all quantum symmetric Kac-Moody pairs. In this setting the Weyl group does not contain a longest element. We hence replace the Lusztig action in \cite[Theorem 2.18]{a-BaoWang13p} by a twist of the underlying module, see Section \ref{sec:pseudoT}. In Section \ref{sec:generalK'} we construct a $\coid$-module homomorphism between twisted versions of modules in $\Oint$. This provides the main step of the construction in the general Kac-Moody case. In Section \ref{sec:finitecase} we restrict to the finite case and obtain a $\cB$-$\tw$-automorphism $\cK$ for $\Oint$ as in Section \ref{sec:TwCylTw} with $\cB$ as in Example \ref{eg:Oint}. The coproduct of $\cK$ will be determined in Section \ref{sec:coproductK}. 
%%%%%%%%%%%%%%%%%%%%%%%%%%%%%%%%%%%%%%%%%%%%%%%%%%%%%%%%%%%%
\subsection{A pseudo longest element of $W$}\label{sec:pseudoT}
%%%%%%%%%%%%%%%%%%%%%%%%%%%%%%%%%%%%%%%%%%%%%%%%%%%%%%%%%%%%
If $\gfrak$ is of finite type then there exists $\tau_0\in \Aut(A)$ such that the longest element $w_0\in W$ satisfies 
\begin{align}\label{eq:w0alphai}
  w_0(\alpha_i)=-\alpha_{\tau_0(i)}\qquad\mbox{ for all $i\in I$.}
\end{align}
Moreover, in this case the Lusztig automorphism $T_{w_0}$ of $\uqg$ corresponding to $w_0$ can be explicitly calculated. Indeed, by  \cite[Proposition 8.20]{b-Jantzen96} or \cite[Lemma 3.4]{a-Kolb14} one has
\begin{equation}\label{eq:Tw0Fi}
\begin{aligned}
T_{w_0}(E_i)&=-F_{\tau_0(i)}K_{\tau_0(i)},&
T_{w_0}(F_i)&=-K_{\tau_0(i)}^{-1}E_{\tau_0(i)},& T_{w_0}(K_i)&=K_{\tau_0(i)}^{-1}, \\
T_{w_0}^{-1}(E_i)&=-K_{\tau_0(i)}^{-1}F_{\tau_0(i)},&
T_{w_0}^{-1}(F_i)&=-E_{\tau_0(i)}K_{\tau_0(i)},& T_{w_0}^{-1}(K_i)&=K_{\tau_0(i)}^{-1}. \\
\end{aligned}
\end{equation}
In the Kac-Moody case we mimic the inverse of the Lusztig automorphism corresponding to the longest element of the Weyl group as follows. Let $\tw:\uqg\rightarrow \uqg$ denote the algebra automorphism defined by
\begin{align*}
  \tw(E_i)= - K_i^{-1} F_i, \qquad \tw(F_i)=- E_i K_i, \qquad \tw(K_h)=K_{-h}
\end{align*}
for all $i\in I$, $h\in Q^\vee_{ext}$.
%%%%%%%%%%%%%%%%%%%%%%%%%%%%%%%%%%%%%%%%%%%%%%%%
\begin{lem}\label{lem:twTi}
  For all $i\in I$ one has $\tw\circ T_i=T_i\circ \tw$ on $\uqg$.
\end{lem}
%%%%%%%%%%%%%%%%%%%%%%%%%%%%%%%%%%%%%%%%%%%%%%%%
\begin{proof}
   For $h\in Q^\vee_{ext}$ one has $T_i \circ \tw(K_h)=K_{-s_i(h)}=\tw\circ T_i(K_h)$. It remains to check that 
\begin{align}\label{eq:Titw-goal}
   T_i\circ \tw(E_j)=\tw\circ T_i(E_j)\qquad \mbox{ and } \qquad T_i\circ \tw(F_j)=\tw\circ T_i(F_j)  
\end{align}   
for all $j\in I$. For $j=i$ relation \eqref{eq:Titw-goal} holds because $T_i^{-1}|_{\uqislzi}=\tw|_{\uqislzi}$. For $j\neq i$ relation \eqref{eq:Titw-goal} is verified by a direct calculation using the formulas
\begin{align*}
  T_i(E_j) &= \sum_{k=0}^{-a_{ij}} (-1)^k q_i^{-k} E_i^{(-a_{ij}-k)} E_j E_i^{(k)},\\
  T_i(F_j) &= \sum_{k=0}^{-a_{ij}} (-1)^k q_i^{k} F_i^{(k)} F_j F_i^{(-a_{ij}-k)}.\\
\end{align*}
which hold by \cite[37.1.3]{b-Lusztig94}.
\end{proof}
%%%%%%%%%%%%%%%%%%%%%%%%%%%%%%%%%%%%%%%%%%%%%%%%%%%%%%%%%%%%
To mimic the Lusztig action of the longest element in the Kac-Moody case we additionally need an automorphism $\tau_0\in \Aut(A,X)$. Recall our setting and assumptions from Section \ref{sec:Assumptions}. For the construction of the universal K-matrix we need to make minor additional assumptions on the parameters $\bc\in \cC$ and $\bs\in \cS$.

\medskip

\noindent{\bf Assumption ($\tau_0$):} We are given an additional involutive element $\tau_0\in\Aut(A,X)$ with the following properties:
  \begin{enumerate}
    \item $\tau\circ \tau_0 = \tau_0\circ \tau$.
    \item The parameters $\bc\in \cC$ and $\bs\in \cS$ satisfy the relations
      \begin{align}\label{eq:ctau0tau}
        c_{\tau_0\tau(i)}=c_i, \qquad s_{\tau_0(i)}=s_i \qquad \mbox{for all $i\in I\setminus X$.}
      \end{align}
    \item The function $s:I\rightarrow \field$ described by \eqref{eq:defs(i)1} and \eqref{eq:defs(i)2} satisfies the relation
      \begin{align}\label{eq:stautau0}
         s(\tau(i))=s(\tau_0(i)) \qquad \mbox{for all $i\in I$}.
      \end{align}   
  \end{enumerate}
%%%%%%%%%%%%%%%%%%%%%%%%%%%%%%%%%%%%%%%%%%%%%%%%%
\begin{rema}\label{rem:tw-finite}
  Assume that $\gfrak$ is of finite type. In this case we always choose $\tau_0$ to be the diagram automorphism determined by Equation \eqref{eq:w0alphai}. Then Property (1) is automatically satisfied as follows by inspection from the list of Satake diagrams in \cite{a-Araki62}. Moreover, by definition of the parameter set $\cS$ one can have $s_i\neq 0$ only if $\tau'(i)=i$ for all $\tau'\in \Aut(A)$. Hence Property (2) reduces to $c_{\tau_0\tau(i)}=c_i$ in the finite case. By \eqref{eq:defs(i)1} and \eqref{eq:defs(i)2} one can have $s(i)\neq 1$ only if $\tau(i)=\tau_0(i)$. Hence Property (3) is always satisfied in the finite case. 
  
If $\tau_0=\tau$ then Property (2) is an empty statement. It is possible that $\tau=\id$ and $\tau_0\neq \id$, see the list in \cite{a-Araki62}. In this case condition \eqref{eq:octau} implies that $c_i$ equals $c_{\tau_0(i)}$ up to multiplication by a bar invariant scalar. The new condition $c_{\tau_0\tau(i)}=c_i$ forces this scalar to be equal to $1$. Finally, only in type $D_{2n}$ is it possible that $\tau_0=\id$ and $\tau\neq \id$. In this case, however, the condition $\bc\in \cC$ implies that $c_{\tau_0\tau(i)}=c_i$. These arguments show that the new condition $c_{\tau_0\tau(i)}=c_i$ is consistent with the conditions imposed in Section \ref{sec:Assumptions} and that it is always possible to choose parameters $\bc$ and $\bs$ which satisfy all of the assumptions. 
\end{rema}
%%%%%%%%%%%%%%%%%%%%%%%%%%%%%%%%%%%%%%%%%%%%%%%%%%%%%%%%%%%%
The composition 
\begin{align*}
  \tw\circ \tau_0:\uqgp\rightarrow \uqgp
\end{align*}   
defines an algebra automorphism. By \eqref{eq:Tw0Fi} the automorphism $\tw\circ \tau_0$ is a Kac-Moody analog of the inverse of the Lusztig action on $\uqg$ corresponding to the longest element in the Weyl group in the finite case. As $\tau_0\in \Aut(A,X)$ one has $\tau_0\circ T_{w_X}=T_{w_X}\circ\tau_0$. By Lemma \ref{lem:twTi} this implies that
\begin{align}\label{eq:tau0twTwX}
  \tau_0\circ \tw\circ T_{w_X} = T_{w_X} \circ \tau_0\circ \tw.
\end{align}
To obtain an analog of this Lusztig action on modules in $\Oint$ we will twist the module structure. 
In the following section we construct a $\coid$-module homomorphism between twisted versions of modules in $\Oint$. As $\coid$ is a subalgebra of $\uqgp$ it suffices to consider objects in $\Oint$ as $\uqgp$-modules. With this convention, for any algebra automorphism $\varphi:\uqgp\rightarrow \uqgp$ and any $M\in Ob(\Oint)$ let $M^\varphi$ denote the vector space $M$ with the $\uqgp$-module structure $u\ot m\mapsto u\bullet_\varphi m$ given by
\begin{align*}
  u\bullet_\varphi m=\varphi(u)m\qquad \mbox{for all $u\in \uqgp$, $m\in M$.}
\end{align*}
We will apply this notation in particular in the case where $\varphi$ is one of $\tau_0\circ \tau$ and $\tw\circ \tau_0$, see Theorem \ref{thm:Bc-hom}. 
%%%%%%%%%%%%%%%%%%%%%%%%%%%%%%%%%%%%%%%
\begin{rema}
  If the algebra automorphism $\varphi:\uqgp\rightarrow \uqgp$ extends to a Hopf algebra automorphism of $\uqg$ then the notation $M^\varphi$ for $M\in Ob(\Oint)$ coincides with the notation in Example \ref{eg:Oint3}.
\end{rema}
%%%%%%%%%%%%%%%%%%%%%%%%%%%%%%%%%%%%%%%%%%%%%%%%%%%%%%%%%%%%
\subsection{The twisted universal K-matrix in the Kac-Moody case}\label{sec:generalK'}
%%%%%%%%%%%%%%%%%%%%%%%%%%%%%%%%%%%%%%%%%%%%%%%%%%%%%%%%%%%%
We keep our assumptions from Section \ref{sec:Assumptions} and Assumption $(\tau_0)$ from the previous subsection. To construct the desired $\coid$-module homomorphism we require one additional ingredient.
Consider the function $\gamma: I\to \field(q^{1/d})$ defined by
\begin{align}\label{eq:gamma(i)}
  \gamma(i)=\begin{cases}
                      1 & \mbox{if $i\in X$}\\
                      c_i s(\tau(i)) & \mbox{if $i\in I \setminus X$}
                   \end{cases}
\end{align}
and note that by \eqref{eq:ctau0tau} and \eqref{eq:stautau0} one has $\gamma(\tau\tau_0(i))=\gamma(i)$ for all $i\in I$.
Now assume that $\xi:P\rightarrow \mathbb{K}(q^{1/d})^\times$ is a function satisfying the following recursion
\begin{align}\label{eq:xi-recursion}
  \xi(\mu+\alpha_i) =\gamma(i) q^{-(\alpha_i,\Theta(\alpha_i)) - (\mu,\alpha_i+\Theta(\alpha_i))} \xi(\mu) \qquad \mbox{for all $\mu\in P$, $i\in I$.}
\end{align}
Such a function exists. Indeed, we may take an arbitrary map on any set of representatives of $P/Q$ and uniquely extend it to $P$ using \eqref{eq:xi-recursion}.
%%%%%%%%%%%%%%%%%%%%%%%%%%%%%%%%%%%%%
\begin{lem}
 Let $\xi:P\rightarrow \field(q^{1/d})^\times$ be any function which satisfies the recursion \eqref{eq:xi-recursion}. Then one has
  \begin{align}\label{eq:xi-QX-recursion}
    \xi(\mu+\lambda)= q^{-(\lambda,\lambda)-2(\mu,\lambda)} \xi(\mu)\qquad \mbox{for all $\mu\in P$, $\lambda\in Q_X$.}
  \end{align}
\end{lem}
%%%%%%%%%%%%%%%%%%%%%%%%%%%%%%%%%%%%%%%%
\begin{proof}
We prove this by induction on the height of $\lambda$. Assume that \eqref{eq:xi-QX-recursion} holds for a given $\lambda\in Q_X$. Then one obtains for any $i\in X$ the relation
 \begin{align*}
   \xi(\mu+\lambda+\alpha_i)&=q^{-(\alpha_i,\alpha_i)-2(\mu+\lambda,\alpha_i)}\xi(\mu+\lambda)\\
   &= q^{-(\alpha_i,\alpha_i)-2(\mu+\lambda,\alpha_i)-(\lambda,\lambda)-2(\mu,\lambda)}\xi(\mu)\\
   &= q^{-(\lambda+\alpha_i,\lambda+\alpha_i)-2(\mu,\lambda+\alpha_i)} \xi(\mu)
 \end{align*}
which completes the induction step. 
\end{proof}
%%%%%%%%%%%%%%%%%%%%%%%%%%%%%%%%%%%%%%
As in Example \ref{eg:xiinU} we may consider $\xi$ as an element of $\mathscr{U}$. The next theorem shows that the element $\Xfrak\, \xi\, T_{w_X}^{-1}\in \sU$ defines a $\coid$-module isomorphism between twisted modules in $\Oint$.
%%%%%%%%%%%%%%%%%%%%%%%%%%%%%%%%%%%%%%
\begin{thm}\label{thm:Bc-hom}
Let $\xi:P\rightarrow \field(q^{1/d})^\times$ be a function satisfying the recursion \eqref{eq:xi-recursion}. Then the element $\cK'=\Xfrak\, \xi\, T_{w_X}^{-1}\in \sU$ defines an isomorphism of $\coid$-modules 
\begin{align*}
  \mathcal{K}'_M:M^{\tw\circ \tau_0} \rightarrow M^{\tau_0\tau}, \quad m\mapsto \Xfrak_M \circ \xi_M \circ (T_{w_X}^{-1})_M(m)
\end{align*}
for any $M\in Ob(\Oint)$. In other words, the relation
\begin{align*}
  \cK'\, \tw( \tau_0(x))=\tau_0(\tau(x))\,\cK'
\end{align*}
holds in $\sU$ for all $x\in \coid$.
\end{thm}
%%%%%%%%%%%%%%%%%%%%%%%%%%%%%%%%%%%%%%%%%%%%%%%%%%%%%
\begin{proof}
 It suffices to check that 
   \begin{align}\label{eq:Txm}
     \mathcal{K}_M'\big(x\bullet_{\tw\circ \tau_0}m)= x\bullet_{\tau \tau_0}\mathcal{K}_M'(m)\qquad \mbox{for all $m\in M$}
   \end{align}
where $x$ is one of the elements $K_\lambda$, $E_i$, $F_i$, or $B_j$  for $\lambda\in Q^\Theta$, $i\in X$, and $j\in I\setminus X$. Moreover, it suffices to prove the above relation for a weight vector $m\in M_\mu$. In the following we will suppress the subscript $M$ for elements in $\sU$ acting on $M$.

\noindent{\bf Case 1: $x=K_\lambda$ for some $\lambda\in Q^\Theta$.} In this case $w_X(\lambda)=-\tau(\lambda)$. Moreover, as $\tau_0\in \Aut(A,X)$ one has $\tau_0 (w_X(\lambda)) = w_X(\tau_0(\lambda))$. Hence one obtains
\begin{align*}
  \mathcal{K}'(K_\lambda \bullet_{\tw\circ \tau_0} m) &=\Xfrak\circ \xi \circ T_{w_X}^{-1} (\tw( \tau_0(K_\lambda)) m)\\
  &=\Xfrak\circ \xi \circ T_{w_X}^{-1} (K_{-\tau_0(\lambda)} m)\\
  &=\Xfrak\big(K_{-w_X\tau_0(\lambda)}(\xi\circ T_{w_X}^{-1}(m))\big)\\
  &=K_{-w_X\tau_0(\lambda)} \Xfrak\circ\xi\circ T_{w_X}^{-1}(m)\\
  &=K_{\tau_0\tau(\lambda)} \mathcal{K}'(m)\\
  &=K_\lambda \bullet_{\tau_0\tau} \mathcal{K}'(m).
\end{align*}
\noindent{\bf Case 2: $x=E_i$ for some $i\in X$.} 
By relation \eqref{eq:Tw0Fi} applied to $\mathcal{M}_X$ we have $T_{w_X}^{-1}(F_i)=-E_{\tau(i)} K_{\tau(i)}$. Using this and the recursion \eqref{eq:xi-recursion} one obtains
\begin{align*}
  \mathcal{K}'&(E_i\bullet_{\tw\circ \tau_0} m)\stackrel{\phantom{\eqref{eq:xi-recursion}}}{=}
  \Xfrak\circ \xi \circ T_{w_X}^{-1} (-K_{\tau_0(i)}^{-1} F_{\tau_0(i)} m)\\
    &\stackrel{\phantom{\eqref{eq:xi-recursion}}}{=}  \Xfrak\circ \xi \big(q^{(\alpha_i,\alpha_i)}E_{\tau \tau_0(i)}K_{\tau \tau_0(i)}^2 T_{w_X}^{-1}(m)\big)\\
    &\stackrel{\phantom{\eqref{eq:xi-recursion}}}{=} \Xfrak \big(\xi(w_X(\mu){+}\alpha_{\tau \tau_0(i)}) \xi(w_X(\mu))^{-1}q^{(\alpha_i,\alpha_i)+2(w_X(\mu),\alpha_{\tau_0\tau(i)})} E_{\tau \tau_0(i)}\xi \circ T_{w_X}^{-1}(m)\big)\\
    &\stackrel{\eqref{eq:xi-recursion}}{=} E_{\tau \tau_0(i)}\Xfrak\circ\xi\circ T_{w_X}^{-1}(m)\\
    &\stackrel{\phantom{\eqref{eq:xi-recursion}}}{=} E_i \bullet_{\tau_0\tau} \mathcal{K}'(m).
\end{align*}
This confirms relation \eqref{eq:Txm} for $x=E_i$ where $i\in X$. The case $x=F_i$ for $i\in X$ is treated analogously.

\noindent{\bf Case 3: $x=B_j=F_j - \gamma(j)T_{w_X}(E_{\tau(j)})K_j^{-1} + s_j K_j^{-1}$ for some $j\in I\setminus X$.}  
We calculate
 \begin{align*}
  \mathcal{K}'(B_j\bullet_{\tw\circ \tau_0} m)
  &\stackrel{\phantom{\eqref{eq:tau0twTwX}}}{=} \Xfrak\circ \xi \circ T_{w_X}^{-1}\big(\tw\circ \tau_0(B_j)m\big)\\  
  &\stackrel{\eqref{eq:tau0twTwX}}{=} \Xfrak\circ \xi \big(\tw\circ \tau_0(T_{w_X}^{-1}(B_j)) T_{w_X}^{-1}(m)\big)\\  
 &\stackrel{\eqref{eq:tau0twTwX}}{=}\Xfrak\circ \xi \Big(\big(T_{w_X}^{-1}(- E_{\tau_0(j)} K_{\tau_0(j)}) + \gamma(j) K_{\tau\tau_0(j)}^{-1}F_{\tau \tau_0(j)} K_{w_X(\alpha_{\tau_0(j)})}\\
 & \qquad \qquad \qquad + s_j K_{w_X(\tau_0(j))}\big) T_{w_X}^{-1}(m)\Big)\\
 &=\Xfrak\circ \xi \Big(\big(F_{\tau \tau_0(j)} \gamma(j) q^{(\alpha_j,\alpha_j)- (w_X(\mu), \alpha_{\tau\tau_0(j)}+ \Theta(\alpha_{\tau\tau_0(j)}))} \\
 & \qquad\qquad \qquad -T_{w_X}^{-1}(E_{\tau_0(j)})K_{\tau_0\tau(j)} q^{-(w_X(\mu),\alpha_{\tau_0\tau(j)}+ \Theta(\alpha_{\tau_0\tau(j)}))} \\
 & \qquad \qquad \qquad + s_j K_{\tau_0\tau(j)}q^{-(w_X(\mu),\alpha_{\tau_0\tau(j)}+ \Theta(\alpha_{\tau_0\tau(j)}))}\big) T_{w_X}^{-1}(m)\Big).
  \end{align*}
To simplify the last term, recall from \eqref{def:setS} that $s_{\tau_0\tau(j)}=s_j=0$ unless $\Theta(\alpha_j)=-\alpha_j$, in which case $\alpha_{\tau_0\tau(j)} + \Theta(\alpha_{\tau_0\tau(j)})=0$. 
Additionally moving $\xi$ to the right one obtains
\begin{align}
  \mathcal{K}'&(B_j\bullet_{\tw\circ \tau_0} m)=\nonumber\\
 &=\Xfrak \Big(\big(F_{\tau \tau_0(j)} \gamma(j) q^{(\alpha_j,\alpha_j)- (w_X(\mu), \alpha_{\tau\tau_0(j)}+ \Theta(\alpha_{\tau\tau_0(j)}))} 
 \frac{\xi(w_X(\mu)-\alpha_{\tau_0\tau(j)})}{\xi(w_X(\mu))} \label{eq:TBj1}\\
 & \qquad -T_{w_X}^{-1}(E_{\tau_0(j)})K_{\tau_0\tau(j)} q^{-(w_X(\mu),\alpha_{\tau_0\tau(j)}+ \Theta(\alpha_{\tau_0\tau(j)}))} \frac{\xi(w_X(\mu){+}w_X\alpha_{\tau_0(j)})}{\xi(w_X(\mu))} \nonumber\\
 & \qquad \qquad \qquad + s_{\tau_0\tau(j)} K_{\tau_0\tau(j)}\big) \xi\circ T_{w_X}^{-1}(m)\Big).\nonumber
  \end{align}
To simplify the above expression observe that
\begin{align*}
  \xi(w_X(\mu) + w_X\alpha_i)&\stackrel{\eqref{eq:xi-QX-recursion}}{=}q^{-(w_X\alpha_i-\alpha_i, w_X\alpha_i-\alpha_i)-2 (w_X(\mu)+\alpha_i,w_X\alpha_i-\alpha_i)}\xi(w_X(\mu) +\alpha_i)\\
  &\stackrel{\phantom{\eqref{eq:xi-recursion}}}{=} q^{-2(w_X(\mu),w_X\alpha_i-\alpha_i)} \xi(w_X(\mu)+\alpha_i)\\
  &\stackrel{\eqref{eq:xi-recursion}}{=}\gamma(\tau_0\tau(i))q^{-(\alpha_i,\Theta(\alpha_i))+ (w_X(\mu),\alpha_{\tau(i)}+\Theta(\alpha_{\tau(i)})}\xi(w_X(\mu))
\end{align*}
for $i\in I\setminus X$. Inserting this formula for $i=\tau_0(j)$ into Equation \eqref{eq:TBj1} and applying the recursion \eqref{eq:xi-recursion} also to the first summand one obtains
\begin{align}
  \mathcal{K}'&(B_j\bullet_{\tw\circ \tau_0} m) =\Xfrak\Big( \big( F_{\tau \tau_0(j)}  
  -\gamma(\tau(j)) q^{-(\alpha_j,\Theta(\alpha_j))} 
    T_{w_X}^{-1}(E_{\tau_0(j)})K_{\tau_0\tau(j)} + \label{eq:TBj2}\\
  &\qquad\qquad \qquad\qquad \qquad + s_{\tau_0\tau(j)} K_{\tau_0\tau(j)}\big)\xi\circ 
    T_{w_X}^{-1}(m)\Big).\nonumber
\end{align}
Now set 
  \begin{align*}
    \beta_i &= (-1)^{2\alpha_i(\rho^\vee_X)}q^{(2\rho_X,\alpha_i)}
               \qquad \mbox{for $i\in I\setminus X$.}
  \end{align*}
In view of \cite[37.2.4]{b-Lusztig94} one has
\begin{align}\label{eq:TwXbar}
  \overline{T_{w_X}(E_i)}&= \beta_i^{-1} T_{w_X}^{-1}(E_i) \qquad \mbox{for all $i\in I\setminus X$}, 
\end{align}
see also the proof of \cite[Lemma 2.9]{a-BalaKolb14p}. Hence \eqref{eq:TBj2} gives 
  \begin{align*} 
 \mathcal{K}'&(B_j\bullet_{\tw\circ \tau_0} m)=\Xfrak \Big(\big(F_{\tau \tau_0(j)}-\\ 
 &-\gamma(\tau(j)) q^{-(\alpha_j,\Theta(\alpha_j))}  \beta_{\tau_0(j)} 
  \overline{T_{w_X}(E_{\tau_0(j)})K_{\tau_0\tau(j)}^{-1}} + s_j K_{\tau_0\tau(j))}\big) \xi \circ 
   T_{w_X}^{-1}(m)\Big). 
 \end{align*}
In view of the relation
\begin{align*}
  \gamma(\tau(j)) q^{-(\alpha_j,\Theta(\alpha_j))}  \beta_{\tau_0(j)}=s(\tau(j))\overline{c_j}
\end{align*} 
one now obtains
\begin{align*}
  \mathcal{K}'(B_j\bullet_{\tw\circ \tau_0} m)&=\Xfrak\left(\overline{B_{\tau_0\tau(j)}}\xi \circ 
   T_{w_X}^{-1}(m)\right)\\ 
   &= B_{\tau_0\circ\tau(j)} \mathcal{K}'(m)= B_j \bullet_{\tau \tau_0} \mathcal{K}'(m)
\end{align*}
which completes the proof of the Theorem.
\end{proof}
%%%%%%%%%%%%%%%%%%%%%%%%%%%%%%%%%%%%%%%%%%%%%%%%%%%%%%%%%%%
For later reference we note that relation \eqref{eq:TwXbar} implies that the element $X_i$ defined by \eqref{def:Xi} satisfies the relation
\begin{align}\label{eq:Xibar}
  \overline{X_i}= -s(i) q^{-(2\rho_X,\alpha_i)} T_{w_X}^{-1}(E_{\tau(i)}) \qquad \mbox{for all $i\in I\setminus X$},
\end{align}
see also \eqref{eq:defs(i)1}, \eqref{eq:defs(i)2} and property \eqref{adm3} in Definition \ref{def:admissible} of an admissible pair.
%%%%%%%%%%%%%%%%%%%%%%%%%%%%%%%%%%%%%%%%
\begin{rema}
  The function $\xi$ is an important ingredient in the construction of the twisted K-matrix $\mathcal{K}'$ and should be compared to the recursively defined function $f$ involved in the construction of the commutativity isomorphisms \cite[32.1.3]{b-Lusztig94}. The recursion \eqref{eq:xi-recursion} is a necessary and sufficient condition on $\xi$ for $\Xfrak \circ\xi\circ T_{w_X}^{-1}: M^{\tw\circ\tau_0}\rightarrow M^{\tau_0\tau}$ to be a $\coid$-module homomorphism. 
\end{rema}
%%%%%%%%%%%%%%%%%%%%%%%%%%%%%%%%%%%%%%%%%
%%%%%%%%%%%%%%%%%%%%%%%%%%%%%%%%%%%%%%%%%%%%%%%%%%%%%%%%%%%%%
\subsection{The universal K-matrix in the finite case}\label{sec:finitecase}
%%%%%%%%%%%%%%%%%%%%%%%%%%%%%%%%%%%%%%%%%%%%%%%%%%%%%%%%%%%%%
We now assume that $\gfrak$ is of finite type. In this case, following Remark \ref{rem:tw-finite}, we always choose $\tau_0\in \Aut(I,X)$ such that the longest element $w_0\in W$ satisfies $w_0(\alpha_i)=-\alpha_{\tau_0(i)}$ for all $i\in I$. By Equation \eqref{eq:Tw0Fi} this gives $T_{w_0}^{-1}=\tau_0\circ \tw$ on $\uqg$.

If $M$ a finite dimensional $\uqg$-module then the Lusztig action $T_{w_0}:M\rightarrow M$ satisfies
$T_{w_0}(um)= T_{w_0}(u) T_{w_0}(m)$ for all $m\in M$, $u\in \uqg$. In other words, the Lusztig action on $M$ defines an $\uqg$-module isomorphism
\begin{align*}
 T_{w_0}: M^{\tw\circ\tau_0}\rightarrow M.
\end{align*}
Composing the inverse of this isomorphism with the isomorphism $\mathcal{K}'$ from Theorem \ref{thm:Bc-hom}, we get the following corollary. 
%%%%%%%%%%%%%%%%%%%%%%%%%%%%%%%%%%%%%%
\begin{cor}\label{cor:K}
  Assume that $\gfrak$ is of finite type and let $\xi:P\rightarrow \field(q^{1/d})^\times$ be a function satisfying the recursion \eqref{eq:xi-recursion}. Then the element $\cK=\Xfrak\,\xi\, T_{w_X}^{-1}\, T_{w_0}^{-1}\in \mathscr{U}$ defines an isomorphism of $B_{\bc,\bs}$-modules
 \begin{align*}
    \cK_M:M \rightarrow M^{\tau \tau_0}, \qquad m\mapsto \Xfrak_M\circ \xi_M\circ (T_{w_X}^{-1})_M\circ (T_{w_0}^{-1})_M(m)
  \end{align*}
for any finite dimensional $\uqg$-module $M$. In other words, the relation
\begin{align*}
  \cK\,b=\tau_0(\tau(b))\,\cK
\end{align*}
holds in $\sU$ for all $b\in \coid$.
\end{cor}
%%%%%%%%%%%%%%%%%%%%%%%%%%%%%%%%%%%%%%
\begin{rema}
  As before let $\cB$ denote the category with objects in $\Oint$ and morphisms $\Hom_\cB(V,W)=\Hom_{\coid}(V,W)$. In the terminology of Section \ref{sec:TwCylTw} the above corollary states that $\cK=\Xfrak\,\xi\, T_{w_X}^{-1}\, T_{w_0}^{-1}$ is a $\cB$-$(\tau\circ \tau_0)$-automorphism of $\Oint$. Equivalently, the element $\cK$ satisfies relation \eqref{eq:phi-cylinder1} in Definition \ref{def:U-cylinder-braided} of a $\tau\tau_0$-universal K-matrix. 
\end{rema}
%%%%%%%%%%%%%%%%%%%%%%%%%%%%%%%%%%%%%%

%%%%%%%%%%%%%%%%%%%%%%%%%%%%%%%%%%%%%%%%%%%%%%%%
\section{A special choice of $\xi$}\label{sec:xi-construction}
%%%%%%%%%%%%%%%%%%%%%%%%%%%%%%%%%%%%%%%%%%%%%%%%
In the following we want to determine the coproduct of the element $\cK\in \sU$ from Corollary \ref{cor:K}. We aim to show that $\cK$ is a $\tau\tau_0$-universal K-matrix for $\coid$, that is that the coproduct $\kow(\cK)$ is given by \eqref{eq:phi-cylinder2}. This, however, will only hold true for a suitable choice of $\xi$.
%%%%%%%%%%%%%%%%%%%%%%%%%%%%%%%%%%%%%%55 
\subsection{Choosing $\xi$}\label{sec:choosing-xi}
%%%%%%%%%%%%%%%%%%%%%%%%%%%%%%%%%%%%%%%%
Recall that $\xi$ has to satisfy the recursion \eqref{eq:xi-recursion} which involves the function $\gamma:I\rightarrow \field(q^{1/d})$ given by \eqref{eq:gamma(i)}. Extend the function $\gamma$ to a group homomorphism $\gamma:P\to \mathbb{K}(q^{1/d})^\times$. Depending on the choice of coefficients $\bc\in \cC$, it may be necessary to replace $\field(q^{1/d})$ by a finite extension to do this. We will illustrate the situation and comment on the field extension in Subsection \ref{sec:lifting-gamma}.

%%%%%%%%%%%%%%%%%%%%%%%%%%%%%%%%%%%
 For any $\lambda\in P$ write
\begin{align*}
  \lambda^+ &= \frac{\lambda+\Theta(\lambda)}{2}, & \tilde{\lambda}=\frac{\lambda-\Theta(\lambda)}{2}.
\end{align*}
Observe that both $(\lambda^+,\lambda^+)$ and $(\tilde{\lambda},\tilde{\lambda})$ are contained in $\frac{1}{2d}\Z$ for all $\lambda\in P$. Recall from Section \ref{sec:RootDatum} that $\varpi^\vee_i$ for $i\in I$ denote the fundamental coweights. Now define a function $\xi:P\rightarrow \field(q^{1/d})^\times$ by 
  \begin{align}\label{eq:xi-def}
    \xi(\lambda) = \gamma(\lambda) q^{-(\lambda^+,\lambda^+)+ \sum_{k\in I}(\alphatil_k,\alphatil_k) \lambda(\varpi_k^\vee)}.
  \end{align}
%%%%%%%%%%%%%%%%%%%%%%%%%%%%%%%%%%%%%%%%%%%%%%%%%%%%%
\begin{rema}
  A priori one only has $-(\lambda^+,\lambda^+) + \sum_{k\in I}(\alphatil_k,\alphatil_k)\lambda(\varpi^\vee_k)\in \frac{1}{2d}\Z$. However, for all $\gfrak$ of finite type one can show by direct calculation that
\begin{align}\label{eq:in1dZ}
  -(\lambda^+,\lambda^+) + \sum_{k\in I}(\alphatil_k,\alphatil_k)
   \lambda(\varpi^\vee_k)\in \frac{1}{d}\Z \qquad \mbox{for all $\lambda \in P$.}
\end{align}
To this end it is useful to reformulate the above condition as
\begin{align*}
  -(\lambdatil,\lambdatil) + \sum_{k\in I}(\alphatil_k,\alphatil_k)
   \lambda(\varpi^\vee_k)\in \frac{1}{d}\Z \qquad \mbox{for all $\lambda \in P$.}
\end{align*}
and to work with the weight lattice $P(\Sigma)$ of the restricted root system $\Sigma$ of the symmetric pair $(\gfrak,\kfrak)$. The relation between $P(\Sigma)$ and $P$ is discussed in detail in \cite[Section 2]{a-Letzter-memoirs}. We expect \eqref{eq:in1dZ} also to hold for infinite dimensional $\gfrak$. If it does not hold, then the definition of $\xi$ requires an extension of $\field(q^{1/d})$ also for the $q$-power to lie in the field.
\end{rema}
%%%%%%%%%%%%%%%%%%%%%%%%%%%%%%%%%%%%%%%%%%%%%%%%%%%%%  
We claim that $\xi$ satisfies the recursion \eqref{eq:xi-recursion}.  
%%%%%%%%%%%%%%%%%%%%%%%%%%%%%%%%%%%%%%%%%%%%%%%%%%%%%
\begin{lem}\label{lem:xi}
The function $\xi:P\rightarrow \field(q^{1/d})^\times$ defined by \eqref{eq:xi-def} satisfies the relation 
\begin{align}
  \xi (\mu+\nu)= \xi (\mu) \xi(\nu) q^{-(\mu+\Theta(\mu),\nu)}  \label{eq:xi(mu+nu)}
\end{align}
for all $\mu, \nu\in P$. In particular, $\xi$ satisfies the recursion \eqref{eq:xi-recursion}.
\end{lem}
%%%%%%%%%%%%%%%%%%%%%%%%%%%%%%%%%%%%%%%%%%%%%%%%%%%%%
\begin{proof}
 For any $\mu,\nu\in P$ one calculates
\begin{align*}
  \xi (\mu+\nu)&= \gamma(\mu+\nu) q^{-((\mu+\nu)^+,(\mu+\nu)^+)+ \sum_{k\in I}(\alphatil_k,\alphatil_k) (\mu+\nu)(\varpi_k^\vee)}\nonumber\\
& = \gamma(\mu) \gamma(\nu) q^{-(\mu^+,\mu^+)-(\nu^+,\nu^+)-(\mu,\nu)-(\Theta(\mu),\nu)+\sum_{k\in I}(\alphatil_k,\alphatil_k) (\mu+\nu)(\varpi_k^\vee)}\nonumber\\
&=\xi (\mu) \xi(\nu) q^{-(\mu+\Theta(\mu),\nu)}  
\end{align*}
which proves \eqref{eq:xi(mu+nu)}. Choosing $\nu=\alpha_i$ one now obtains
\begin{align*}
  \xi(\mu+\alpha_i)&\stackrel{\eqref{eq:xi(mu+nu)}}{=}\xi(\mu)\xi(\alpha_i) q^{-(\mu+\Theta(\mu),\alpha_i)}\\
  &\stackrel{\eqref{eq:xi-def}}{=}\xi(\mu)\gamma(i) q^{-(\alpha_i^+,\alpha_i^+)+(\alphatil_i,\alphatil_i)-(\mu,\alpha_i+\Theta(\alpha_i))}
\end{align*}
As $(\alpha_i^+,\alpha_i^+)-(\alphatil_i,\alphatil_i)=(\alpha_i,\Theta(\alpha_i))$ the above formula implies that $\xi$ satisfies recursion \eqref{eq:xi-recursion}.
 \end{proof}
%%%%%%%%%%%%%%%%%%%%%%%%%%%%%%%%%%%%%%%%%%%%%%%%%%%

%%%%%%%%%%%%%%%%%%%%%%%%%%%%%%%%%%%%%%%%%%%%%%%%%%%
\subsection{The coproduct of $\xi$}\label{sec:xi-coproduct}
%%%%%%%%%%%%%%%%%%%%%%%%%%%%%%%%%%%%%%%%%%%%%%%%%%%
Recall the invertible element $\kappa\in\mathscr{U}^{(2)}$ defined in Example \ref{eg:kappa-def}. Let $f:P\rightarrow P$ be any map. For every $M,N \in Ob(\Oint)$ define a linear map
\begin{align}
  \kappa^{f}_{M,N}:M\ot N\rightarrow M\ot N, \quad (m\ot n) \mapsto 
  q^{(f(\mu),\nu)}m\ot n \quad \mbox{if $m{\in} M_\mu$, $n{\in} N_\nu$.}\label{def:kappa-f} 
\end{align}
As in Example \ref{eg:kappa-def} the collection $\kappa^{f}=(\kappa^{f}_{M,N})_{M,N\in Ob(\Oint)}$ defines an element in $\sU^{(2)}$. 
%%%%%%%%%%%%%%%%%%%%%%%%%%%%%%%%%%%%%%%%5
\begin{rema}
In the following we will apply this notion in the case $f=-\Theta=w_X\circ \tau$, see Section \ref{sec:SecondKind}. To this end we need to assume that the minimal realization $(\hfrak, \Pi,\Pi^\vee)$ is compatible with the involution $\tau\in \Aut(I,X)$ as in \cite[2.6]{a-Kolb14}. This means that the map $\tau:\Pi^\vee\rightarrow \Pi^\vee$ extends to a permutation $\tau:\Pi^\vee_{ext}\rightarrow \Pi^\vee_{ext}$ such that $\alpha_{\tau(i)}(d_\tau(s))=\alpha_i(d_s)$. In this case $\tau$ may be considered as a map $\tau:P\rightarrow P$. We will make this assumption without further comment. In the finite case, which is our only interest in Section \ref{sec:coproductK}, it is always satisfied.
\end{rema}
%%%%%%%%%%%%%%%%%%%%%%%%%%%%%%%%%%%%%%%%
Recall from Example \ref{eg:xiinU} that the function $\xi$ defined by \eqref{eq:xi-def} may be considered as an element in $\sU$ and hence we can take its coproduct, see Section \ref{sec:U-coproduct}. The coproduct $\kow(\xi)\in \sU^{(2)}$ can be explicitly determined.
%%%%%%%%%%%%%%%%%%%%%%%%%%%%%%%%%%%%%%%%%%%%%%%%%%%
\begin{lem}\label{lem:DeltaXi}
The element $\xi\in \sU$ defined by \eqref{eq:xi-def} satisfies the relation
\begin{equation}
     \kow(\xi)=(\xi\ot\xi)\cdot \kappa^{-1} \cdot \kappa^{-\Theta}. \label{DeltaXi}
\end{equation}
\end{lem}
%%%%%%%%%%%%%%%%%%%%%%%%%%%%%%%%%%%%%%%%%%%%%%%%%%%
\begin{proof}
Let $M,N\in Ob(\Oint)$ and $m\in M_\mu$, $n\in N_\nu$ for some $\mu,\nu\in P$. Then $m\otimes n$ lies in the weight space $(M\otimes N)_{\mu+\nu}$. Hence one gets
\begin{align*}
\kow(\xi)(m\otimes n )&=\xi (\mu+\nu)  m\otimes n\\
& \stackrel{\eqref{eq:xi(mu+nu)}}{=} \xi (\mu) \xi(\nu) q^{-(\mu,\nu)}  q^{-(\Theta(\mu),\nu) }  m\otimes n\\
&\stackrel{\phantom{\eqref{eq:xi(mu+nu)}}}{=} (\xi\ot\xi) \cdot \kappa^{-1} \cdot \kappa^{-\Theta} ( m\otimes n)
\end{align*}
which proves formula \eqref{DeltaXi}.
\end{proof}
%%%%%%%%%%%%%%%%%%%%%%%%%%%%%%%%%%%%%%%%%%
For the rest of this paper the symbol $\xi$ will always denote the function given by \eqref{eq:xi-def} and the corresponding element of $\sU$.

%%%%%%%%%%%%%%%%%%%%%%%%%%%%%%%%%%%%%%%%%
\subsection{The action of $\xi$ on $\uqg$}\label{sec:xi-adjoint}
%%%%%%%%%%%%%%%%%%%%%%%%%%%%%%%%%%%%%%%
Conjugation by the invertible element $\xi\in \sU$ gives an automorphism 
\begin{align*}
  \Ad(\xi):\mathscr{U}\to \mathscr{U},\qquad u\mapsto\Ad(\xi)(u)=\xi u \xi^{-1}.
\end{align*}  
For any $M\in Ob(\Oint)$ one has
    \begin{align}\label{eq:Adxi-property}
      \xi(u m)=\Ad(\xi)(u) \xi(m) \qquad \mbox{for all $u\in \sU$ and $m\in M$.}
    \end{align}
Recall that we consider $\uqg$ as a subalgebra of $\sU$.    
%%%%%%%%%%%%%%%%%%%%%%%%%%%%%%%%%%%%%%%%%%%
\begin{lem}\label{lem:Adxi}
The automorphism $\Ad(\xi):\mathscr{U}\to \mathscr{U}$ restricts to an automorphism of $\uqg$. More explicitly one has
    \begin{align}
      \Ad(\xi)(E_\nu)&= \xi(\nu)E_\nu K_{\nu+\Theta(\nu)}^{-1},\label{eq:AdxiE}\\
      \Ad(\xi)(F_\nu)&= \xi(\nu)^{-1}  K_{\nu+\Theta(\nu)} F_\nu,\label{eq:AdxiF}\\
      \Ad(\xi)(K_i)&=K_i \label{eq:AdxiK}
    \end{align}
for all $E_\nu\in U^+_\nu$ and $F_\nu\in U^-_{-\nu}$ and all $i\in I$.
\end{lem} 
%%%%%%%%%%%%%%%%%%%%%%%%%%%%%%%%%%%%%%%%%%%%%
\begin{proof}
By definition the elements $\xi$ and $K_i$ commute in $\sU$. This proves \eqref{eq:AdxiK}. To verify the remaining two formulas let $M\in Ob(\Oint)$ and $m\in M_\mu$ for some $\mu\in P$. Then one has
\begin{align*}
  \xi E_\nu \xi^{-1} m = \frac{\xi(\mu+\nu)}{\xi(\mu)} E_\nu m \stackrel{\eqref{eq:xi(mu+nu)}}{=} \xi(\nu) q^{-(\mu+\Theta(\mu),\nu)} E_{\nu} m=\xi(\nu) E_\nu K_{\nu+\Theta(\nu)}^{-1}m
\end{align*}
which proves formula \eqref{eq:AdxiE}. Formula \eqref{eq:AdxiF} is obtained analogously using the relation $\xi(\nu)^{-1}=\xi(-\nu)q^{(\nu+\Theta(\nu),\nu)}$ which also follows from \eqref{eq:xi(mu+nu)}.
\end{proof}
%%%%%%%%%%%%%%%%%%%%%%%%%%%%%%%%%%%%%%%%%%%%%%
For $\gfrak$ of finite type the above lemma allows us to identify the restriction of $\Ad(\xi)$ to the subalgebra $\cM_X U^0_\Theta{}'$ of $\uqg$. Recall the conventions for the diagram automorphisms $\tau_0$ in the finite case from Remark \ref{rem:tw-finite}.
%%%%%%%%%%%%%%%%%%%%%%%%%%%%%%%%%%%%%%%%%%%%%%
\begin{lem}\label{lem:Adxi-res}
  Assume that $\gfrak$ is of finite type. Then one has
  \begin{align*}
    \Ad(\xi)\big|_{\cM_XU^0_\Theta}= (T_{w_0} T_{w_X}\tau\tau_0)\big|_{\cM_XU^0_\Theta}.
  \end{align*}  
\end{lem}
%%%%%%%%%%%%%%%%%%%%%%%%%%%%%%%%%%%%%%%%%%%%%
\begin{proof}
  Consider $\mu\in Q^+_X$ and elements $E_\mu\in U^+_\mu$ and $F_{\mu}\in U^-_{-\mu}$. By Lemma \ref{lem:Adxi} and relations \eqref{eq:Tw0Fi} one has
  \begin{align*}
    \Ad(\xi)(E_\mu)&=q^{-(\mu,\mu)}E_\mu K_\mu^{-2}= T_{w_0} T_{w_X}\tau\tau_0(E_{\mu}),\\
    \Ad(\xi)(F_\mu)&=q^{(\mu,\mu)} K_\mu^2 F_\mu = T_{w_0} T_{w_X}\tau\tau_0(F_{\mu}).
  \end{align*}
Moreover, if $\nu \in Q^\Theta$ then Lemma \ref{lem:Adxi} implies that 
  \begin{align*}  
    \Ad(\xi)(K_\nu)&=K_\nu=T_{w_0}\tau_0 T_{w_X}\tau(K_\nu)
  \end{align*} 
which completes the proof of the lemma.
\end{proof}
%%%%%%%%%%%%%%%%%%%%%%%%%%%%%%%%%%%%%%%%%%%%%%%%%%%%
\subsection{Extending $\gamma$ from $Q$ to $P$}\label{sec:lifting-gamma}
%%%%%%%%%%%%%%%%%%%%%%%%%%%%%%%%%%%
In this final subsection, we illustrate how different choices of $\uc$ and $s$ influence the extension of the group homomorphism $\gamma:Q\rightarrow \field(q^{1/d})^\times$ to the weight lattice $P$. As an example consider the root datum of type $A_3$ with $I=\{1,2,3\}$, that is $\gfrak=\slfrak_4(\C)$, and the admissible pair $(X,\tau)$ given by $X=\{2\}$ and $\tau(i)=4-i$.
In this case the constraints \eqref{eq:octau} and \eqref{eq:defs(i)2} reduce to the relations 
\begin{align*}
  c_3=q^2\overline{c_1}, \qquad s(3)=-s(1).
\end{align*}  
The group homomorphism $\gamma:Q\rightarrow \field(q^{1/d})^\times$ is defined by
\begin{align*}
  \gamma(\alpha_1)=s(3)c_1, \qquad \gamma(\alpha_2)=1, \qquad \gamma(\alpha_3)=s(1)c_3.
\end{align*}
The weight lattice $P$ is spanned by the fundamental weights
\begin{align*}
  \varpi_1=\frac{3\alpha_1+2\alpha_2+\alpha_3}{4}, \qquad
  \varpi_2=\frac{\alpha_1+2\alpha_2+\alpha_3}{2}, \qquad
  \varpi_3=\frac{\alpha_1+2\alpha_2+3\alpha_3}{4}
\end{align*}  
and we want to extend $\gamma$ from $Q$ to $P$. 

\noindent {\bf Choice 1.} Let $c_1=c_3=q$, $s(1)=1$, $s(3)=-1$. Then $\gamma(\alpha_1)=-q$, $\gamma(\alpha_3)=q$, and $\gamma$ can be extended to $P$ by
\begin{align*}
  \gamma(\varpi_1)=e^{3\pi i /4}q, \qquad 
  \gamma(\varpi_2)=e^{\pi i /2}q, \qquad 
  \gamma(\varpi_3)=e^{\pi i /4}q.
\end{align*}  

\noindent{\bf Choice 2.} Let $c_1=1-q^2$, $c_3=q^2-1$, $s(1)=1$, $s(3)=-1$. Then $\gamma(\alpha_1)=\gamma(\alpha_3)=q^2-1$, and $\gamma$ can be extended to $P$ by
\begin{align*}
  \gamma(\varpi_1)=\gamma(\varpi_2)=\gamma(\varpi_3)=q^2-1.
\end{align*}  
The advantage of Choice 1 is that the parameters $c_i$ specialize to $1$ as $q\to 1$. 
This property is necessary to show that $\coid$ specializes to $U(\kfrak)$ for $q\rightarrow 1$, see \cite[Section 10]{a-Kolb14}. The drawback of Choice 1 is that, for $\gamma$ to extend to $P$, the field $\mathbb{K}$ must contain some $4$-th root of $-1$. Choice 2, one the other hand, has the advantage that $\gamma$ can be defined on $P$ with values in $\mathbb{Q}(q)^\times$. The drawback of Choice 2 is that $c_i\to 0$ as $q\to 1$ and hence $\coid$ does not specialize to $U(\kfrak)$.

For any quantum symmetric pair of finite type it is possible to find analogs of Choice 1 and Choice 2 above. We can choose $c_i=q^{a_i}$ for some $a_i\in \Z$, see \cite[Remark 3.14]{a-BalaKolb14p}. If $X=\emptyset$ or $\tau=\id$ then $\gamma$ extends to a group homomorphism $P\rightarrow \field(q^{1/d})$ and no field extension is necessary. Now assume that  $X\neq \emptyset$ and $\tau\neq \id$. If we keep the choice $c_i=q^{a_i}$ then the extension of $\gamma$ to $P$ requires the field to contain certain roots of unity. Alternatively, as in Choice 2, one can choose $c_i\in \{q^{a_i}, (1-q^{b_i})q^{a_i}\}$ for some $a_i \in \mathbb{Z}$, $b_i\in \Z$ and $s(i)=\pm 1$ in such a way that $\gamma$ can be extended from $Q$ to $P$ with values in $\Q(q^{1/d})^\times$.
%%%%%%%%%%%%%%%%%%%%%%%%%%%%%%%%%%%%%%%%%%%%%%%%%%%%%%%%%%%%%%%%%%%%

%%%%%%%%%%%%%%%%%%%%%%%%%%%%%%%%%%%%%%%%%%%%%%%%%%%%%%%%%%%%%%%%%%%%
\section{The coproduct of the universal K-matrix $\cK$}\label{sec:coproductK}
%%%%%%%%%%%%%%%%%%%%%%%%%%%%%%%%%%%%%%%%%%%%%%%%%%%%%%%%%%%%%%%%%%%5
For the remainder of this paper we assume that $\gfrak$ is of finite type. We keep the setting from Section \ref{sec:Assumptions} and Assumption ($\tau_0$) from Section \ref{sec:pseudoT}. Recall that in the finite case Assumptions (1) and (4) in Section \ref{sec:Assumptions} are always satisfied. Moreover, by Remark \ref{rem:tw-finite} Assumption ($\tau_0$) reduces to Equation \eqref{eq:ctau0tau} where $\tau_0$ is determined by \eqref{eq:w0alphai}. In this section we calculate the coproduct of the element 
  \begin{align*}
    \cK=\Xfrak\,\xi\, T_{w_0}^{-1}\,T_{w_X}^{-1}\in \sU
  \end{align*}   
given in Corollary \ref{cor:K}. This will show that $\cK$ is indeed a $\tau\tau_0$-universal K-matrix. As an essential step we determine the coproduct of the quasi-K-matrix $\Xfrak$ in Section \ref{sec:deltaX}. First, however, we perform some calculations which simplify later arguments. 
%%%%%%%%%%%%%%%%%%%%%%%%%%%%%%%%%%%%%%%%%%%%%%%%%%%%%%%%%%%%%%%%%%%
\subsection{Preliminary calculations with the quasi R-matrix}\label{sec:preliminaryR}
%%%%%%%%%%%%%%%%%%%%%%%%%%%%%%%%%%%%%%%%%%%%%%%%%%%%%%%%%%%%%%%%%%%%
Let $R_X$ denote the quasi-R-matrix corresponding to the semisimple Lie subalgebra $\gfrak_X$ of $\gfrak$. Recall that $w_0$ and $w_X$ denote the longest elements of $W$ and $W_X$, respectively. Choose a reduced expression $w_0=s_{i_1}\ldots s_{i_t}$ such that $w_X=s_{i_1}\ldots s_{i_s}$ for some $s<t$. As in Remark \ref{rem:R-finite} the quasi-R-matrices $R$ and $R_X$ can then be written as
\begin{align*}
R&=R^{[t]}\cdot R^{[t-1]}\cdots R^{[1]}, &
R_X&=R^{[s]}\cdot R^{[s-1]}\cdots R^{[1]}.
\end{align*}
In view of relation \eqref{eq:Rinv-Rbar} one obtains that 
\begin{align}
R\, \overline{R_X}=RR_X^{-1} &=R^{[t]}\cdot R^{[t-1]}\cdots R^{[s+1]}.\label{eq:RRX-aux1}
\end{align}
Define
\begin{align}\label{eq:RtX-def}
  R^{(\tau,X)}=(\Ad(\xi)T_{w_0}^{-1}T_{w_X}^{-1}\tau\tau_0\ot 1)(R \overline{R_X}) \,\, \in \sU^{(2)}_0.
\end{align}
We will see in Theorem \ref{thm:kowX} that the element $R^{(\tau,X)}$ is a major building block of the coproduct $\kow(\Xfrak)$.
%%%%%%%%%%%%%%%%%%%%%%%%%%%%%%%%%%%%%%%%%%%
\begin{lem}\label{lem:Adttttin}
  The following relation holds
  \begin{align*}
    \RtX \in \prod_{\mu\in w_XQ^+\cap Q^+} U^+_{-\Theta(\mu)} K_\mu\ot U^+_\mu.
  \end{align*}  
\end{lem}
%%%%%%%%%%%%%%%%%%%%%%%%%%%%%%%%%%%%%%%%%%
\begin{proof}
For any $j=1,\dots,t$ let $\gamma_j$ denote the corresponding root as in Remark \ref{rem:R-finite}. If $s+1\le j \le t$ then the factor $R^{[j]}$ defined by \eqref{eq:Rgammaj} satisfies
\begin{align*}
  R^{[j]}\in \prod_{n\ge 0}( T_{w_X}(U^-) \cap U^-_{-n \gamma_j}) \otimes U^+_{n \gamma_j}.
\end{align*}  
Using \eqref{eq:RRX-aux1} we get that
  \begin{align*}
    R\overline{R_X}\in \prod_{\mu\in w_XQ^+\cap Q^+} (T_{w_X}(U^-)\cap U^-_{-\mu})\ot U^+_\mu.
  \end{align*} 
 This implies that
  \begin{align*}
    (T_{w_X}^{-1}\tau\ot 1)(R\overline{R_X}) \in \prod_{\mu\in w_XQ^+\cap Q^+} U^-_{-w_X(\tau(\mu))}\ot U^+_\mu.
  \end{align*}
Using \eqref{eq:Tw0Fi} we get
  \begin{align*}
    (T_{w_0}^{-1}\tau_0 T_{w_X}^{-1}\tau\ot 1)(R\overline{R_X}) \in \prod_{\mu\in w_XQ^+\cap Q^+} U^+_{w_X(\tau(\mu))}K_{w_X(\tau(\mu))}\ot U^+_\mu.
  \end{align*}
  As $\tau_0$ and $T_{w_X}$ commute, relation \eqref{eq:AdxiE} gives 
  \begin{align*}
    (\Ad(\xi)T_{w_0}^{-1}T_{w_X}^{-1}\tau_0\tau\ot 1)(R\overline{R_X}) \in \prod_{\mu\in w_XQ^+\cap Q^+} U^+_{w_X(\tau(\mu))}K_{\mu}\ot U^+_\mu
  \end{align*} 
  which completes the proof of the lemma.
\end{proof}
%%%%%%%%%%%%%%%%%%%%%%%%%%%%%%%%%%%%%%%%%%%%%%%%%%%%%%%%%%%%%%%%%%%%%%%%%%%%%%%%%%%%%%
\begin{lem}\label{lem:1ri-long}
  For any $i\in I$ the following relation holds
  \begin{align*} 
    (1\ot r_i)\RtX &=\rho_i (q_i{-}q_i^{-1})\cdot \RtX \cdot R_X \cdot (T_{w_X}^{-1}(E_{\tau(i)})\ot 1) \cdot \overline{R_X}\cdot (K_i\ot 1)
  \end{align*}
 in $\sU^{(2)}_0$ where $\rho_i=c_{\tau(i)}s(i)q^{-(\alpha_i,\Theta(\alpha_i))}$, in particular $\rho_i=0$ if $i\in X$.
\end{lem}
%%%%%%%%%%%%%%%%%%%%%%%%%%%%%%%%%%%%%%%%%%%
\begin{proof}
  It follows from the intertwining property \eqref{eq:quasiR-property1} of the quasi R-matrix for $u=F_j$ and from relation \eqref{eq:riCommute} that
  \begin{align}\label{eq:riR}
    (1\ot r_j)R = - (q_j-q_j^{-1})\cdot R \cdot  (F_j\otimes 1) \qquad \mbox{for all $j\in I$.}
  \end{align}
and that $(1\ot r_i)(R\overline{R_X})=0$ for $i\in X$. This proves the lemma for $i\in X$. Now assume that $i\in I\setminus X$.  
In view of property \eqref{def:ri} for $U^+$ and the fact that $(1\ot r_i)R_X=0$ for $i\in I\setminus X$, relation \eqref{eq:riR} implies that
\begin{align*}
  (1\ot r_i)(R\overline{R_X})=-(q_i-q_i^{-1})R\cdot (F_iK_i^{-1}\ot 1) \cdot \overline{R_X} \cdot (K_i\ot 1).
\end{align*}
Hence one gets
  \begin{align} 
    (1\ot r_i)&(\Ad(\xi)T_{w_0}^{-1} T_{w_X}^{-1}\tau\tau_0\ot 1)(R \overline{R_X})= \nonumber\\   
    &=-(q_i-q_i^{-1})(\Ad(\xi)T_{w_0}^{-1}T_{w_X}^{-1}\tau\tau_0\ot 1)\big(R\cdot (F_iK_i^{-1}\ot 1) \cdot\overline{R_X}\cdot (K_i\ot 1)\big)\nonumber\\
    &=-(q_i-q_i^{-1})(\Ad(\xi)T_{w_0}^{-1}T_{w_X}^{-1}\tau\tau_0\ot 1)(R \overline{R_X})\cdot \nonumber\\
    &\qquad \qquad \qquad  \cdot(\Ad(\xi)T_{w_0}^{-1}T_{w_X}^{-1}\tau\tau_0\ot 1)\big(R_X \cdot (F_i K_i^{-1}\ot 1)\cdot \overline{R_X}\cdot (K_i\ot 1)\big).\label{eq:first-step}
  \end{align}
   Applying Lemma \ref{lem:Adxi-res} to the second factor in the above expression we get
  \begin{align}\label{eq:second-factor}
    (\Ad(\xi)T_{w_0}^{-1}T_{w_X}^{-1}&\tau\tau_0\ot 1)\big(R_X \cdot (F_i K_i^{-1}\ot 1) \cdot\overline{R_X}\cdot(K_i\ot 1)\big)= \nonumber\\
    &=R_X\cdot(\Ad(\xi)T_{w_0}^{-1}T_{w_X}^{-1}\tau\tau_0(F_i)\ot 1)\cdot (K_i^{-1}\ot 1)\cdot\overline{R_X}\cdot(K_i\ot 1)\nonumber\\
    &=-R_X\cdot \big(\Ad(\xi)(T_{w_X}^{-1}(E_{\tau(i)})) K_{w_X(\alpha_{\tau(i)})-\alpha_i}\ot 1\big)\cdot\overline{R_X}\cdot(K_i\ot 1).
  \end{align} 
    As $T_{w_X}^{-1}(E_{\tau(i)})\in U^+_{-\Theta(\alpha_i)}$ relation \eqref{eq:AdxiE} gives
    \begin{align*}
      \Ad(\xi)(T_{w_X}^{-1}(E_{\tau(i)}))=\gamma(\tau(i))q^{-(\alpha_i,\Theta(\alpha_i))} T_{w_X}^{-1}(E_{\tau(i)}) K_{\alpha_i+\Theta(\alpha_i)}.
    \end{align*}
  Inserting this expression into \eqref{eq:second-factor} gives
  \begin{align*}
    (\Ad(\xi)T_{w_0}^{-1}T_{w_X}^{-1}&\tau\tau_0\ot 1)\big(R_X \cdot (F_i K_i^{-1}\ot 1) \cdot\overline{R_X}\cdot (K_i\ot 1)\big)=\\
    &=-\gamma(\tau(i))q^{-(\alpha_i,\Theta(\alpha_i))} R_X\cdot \big(T_{w_X}^{-1}(E_{\tau(i)})\ot 1\big)\cdot\overline{R_X}\cdot (K_i\ot 1).
  \end{align*}
Finally, inserting the above formula into \eqref{eq:first-step} produces the desired result.
\end{proof}
%%%%%%%%%%%%%%%%%%%%%%%%%%%%%%%%%%%%%%%%%%%
\subsection{The coproduct of the quasi K-matrix $\Xfrak$}\label{sec:deltaX}
%%%%%%%%%%%%%%%%%%%%%%%%%%%%%%%%%%%%%%%%%%%
 As in Section \ref{sec:quasiK} write the quasi K-matrix as $\Xfrak=\sum_{\mu\in Q^+}\Xfrak_\mu \in \sU$ with $\Xfrak_\mu\in U^+_\mu$. Define
  \begin{align}\label{eq:XK2}
    \Xfrak_{K2}= \sum_{\mu\in Q^+} K_{\mu}\ot \Xfrak_\mu \in \mathscr{U}^{(2)}_0.
  \end{align}
This element will appear as a factor in the coproduct of $\Xfrak$. 
%%%%%%%%%%%%%%%%%%%%%%%%%%%%%%%%%%%%%%%%%%%
\begin{lem}\label{lem:1riXK2}
  For any $i\in I$ one has
  \begin{align*}
    (1\ot r_i)(\Xfrak_{K2})&=  (q_i-q_i^{-1}) \Xfrak_{K2} \cdot \big( \rho_i K_{\alpha_i-\Theta(\alpha_i)}\ot T_{w_X}^{-1}(E_{\tau(i)}) - s_i K_i\ot 1\big)
  \end{align*}
 in $\sU^{(2)}$ where $\rho_i=c_{\tau(i)}s(i) q^{-(\alpha_i,\Theta(\alpha_i))}$, in particular $\rho_i=s_i=0$ if $i\in X$.
\end{lem}
%%%%%%%%%%%%%%%%%%%%%%%%%%%%%%%%%%%%%%%%%%%
\begin{proof}
  By Equation \eqref{eq:riX} one has
  \begin{align}
    (1\ot r_i)&(\Xfrak_{K2})=\sum_{\mu\in Q^+} K_{\mu}\ot r_i(\Xfrak_\mu)\nonumber\\
      &=-(q_i{-}q_i^{-1})\sum_{\mu\in Q^+} K_{\mu}\ot\big(\oci \Xfrak_{\mu+\Theta(\alpha_i)-\alpha_i} \overline{X_i} + s_i \Xfrak_{\mu-\alpha_i}\big)\nonumber\\
      &=-(q_i{-}q_i^{-1}) \oci \sum_{\mu\in Q^+} (K_{\mu+\Theta(\alpha_i)-\alpha_i} \ot\Xfrak_{\mu+\Theta(\alpha_i)-\alpha_i})\,(K_{\alpha_i-\Theta(\alpha_i)}\ot \overline{X_i})\nonumber\\
      &\phantom{=}\, 
      -(q_i{-}q_i^{-1})s_i\sum_{\mu\in Q^+} (K_{\mu-\alpha_i}\ot \Xfrak_{\mu-\alpha_i})\,(K_i\ot 1)\nonumber\\
      &=-(q_i{-}q_i^{-1})\Xfrak_{K2}\big( \oci K_{\alpha_i-\Theta(\alpha_i)}\ot \overline{X_i} + s_i K_i\ot 1 \big).\label{eq:nearly2}
  \end{align}
By the explicit expression \eqref{eq:Xibar} for $\overline{X_i}$ and by relation \eqref{eq:octau} for $\overline{c_i}$ one has
\begin{align}\label{eq:overlineciXi}
  \overline{c_iX_i}=-\rho_i T_{w_X}^{-1}(E_{\tau(i)}).
\end{align}
Inserting \eqref{eq:overlineciXi} into \eqref{eq:nearly2} one obtains the desired formula.
\end{proof}
%%%%%%%%%%%%%%%%%%%%%%%%%%%%%%%%%%%%%%%%%%%
By Proposition \ref{prop:TFAE}.(4) the element $\Xfrak_{K2}$ commutes with $F_i\ot E_j$ for all $i,j\in X$ and hence 
\begin{align}\label{eq:RX-XK2}
  R_X\cdot \Xfrak_{K2}=\Xfrak_{K2}\cdot R_X.
\end{align}
Now we are ready to compute the coproduct $\Delta(\Xfrak)\in \sU^{(2)}$ of the quasi K-matrix $\Xfrak \in \mathscr{U}$ in terms of the elements $\RtX$ defined in \eqref{eq:RtX-def} and $\Xfrak_{K2}$ defined in \eqref{eq:XK2}.
%%%%%%%%%%%%%%%%%%%%%%%%%%%%%%%%%%%%%%%%%%%
\begin{thm}\label{thm:kowX}
  The intertwiner $\Xfrak$ satisfies the relation
  \begin{align}\label{eq:kowX}
    \kow(\Xfrak) = (\Xfrak\ot 1)\cdot  \RtX \cdot \Xfrak_{K2}
  \end{align}
  in $\sU^{(2)}_0$.
\end{thm}
%%%%%%%%%%%%%%%%%%%%%%%%%%%%%%%%%%%%%%%%%%%
\begin{proof}
  By the definition \eqref{eq:E-copr} of the coproduct of $\uqg$ the left hand side of Equation \eqref{eq:kowX} belongs to $\prod_{\mu\in Q^+} U^+ K_\mu\ot U^+_\mu$. The right hand side of \eqref{eq:kowX} also belongs to $\prod_{\mu\in Q^+} U^+ K_\mu\ot U^+_\mu$, as follows from  the definition \eqref{eq:XK2} of $\Xfrak_{K2}$ and from Lemma \ref{lem:Adttttin}.
By Lemma \ref{lem:yzXX'} it hence suffices to show that
\begin{align}\label{eq:kowX-goal}
  \left< y\ot z, \kow(\Xfrak)\right>=   \left< y\ot z, (\Xfrak\ot 1) \cdot \RtX \cdot\Xfrak_{K2}\right>
\end{align}
for all $y,z\in U^-$. By linearity it suffices to show this in the case where $z=F_{i_1} F_{i_2} \dots F_{i_r}$ is a monomial in the generators $F_i$ of $U^-$. We perform induction on $r$.

For $r=0$, we have $z=1$ and both sides of \eqref{eq:kowX-goal} equal $\left<y,\Xfrak\right>$. Now assume that \eqref{eq:kowX-goal} holds for all $y\in U^-$ and all monomials $z\in U^-$ of length shorter or equal than $r$. Then we get for any $i\in I$ the relation
\begin{align*}
  \left<y\ot z F_i, \kow \Xfrak\right>&\stackrel{\eqref{eq:formdef}}{=}\left<yzF_i,\Xfrak\right>\\
         &\stackrel{\eqref{eq:form-ri}}{=}\frac{-1}{q_i-q_i^{-1}} \left<yz,r_i(\Xfrak)\right>\\
         &\stackrel{\eqref{eq:riX}}{=}\left<yz, \Xfrak (\oci \overline{X_i} + s_i) \right> \\
         &\stackrel{\eqref{eq:overlineciXi}}{=} -\left<yz, \Xfrak  ( \rho_i T_{w_X}^{-1}(E_{\tau(i)}) - s_i) \right>,
\end{align*}
where as before $\rho_i=c_{\tau(i)} s(i) q^{-(\alpha_i,\Theta(\alpha_i))}$.
By induction hypothesis we obtain
\begin{align}
  \langle y\ot z F_i, \kow \Xfrak\rangle =&-\big<y\ot z, \rho_i(\Xfrak\ot 1)\cdot \RtX \cdot\Xfrak_{K2} \cdot\kow(T_{w_X}^{-1}(E_{\tau(i)}))\big> \label{eq:half-way-kowX}  \\
      &+\big<y\ot z, s_i (\Xfrak\ot 1)\cdot \RtX \cdot\Xfrak_{K2} \big>.\nonumber 
\end{align}
On the other hand, using Lemma \ref{lem:1ri-long} and Lemma \ref{lem:1riXK2} one gets
\begin{align*}
  - &\frac{1}{q_i-q_i^{-1}}(1\ot r_i)\big((\Xfrak\ot 1)\cdot R^{(\tau,X)}\cdot \Xfrak_{K,2}\big) \\
  =& -\rho_i (\Xfrak\ot 1)\cdot R^{(\tau,X)}\cdot R_X \cdot (T_{w_X}^{-1}(E_{\tau(i)})\ot 1)\cdot \overline{R_X}\cdot (K_i\ot K_i) \Xfrak_{K2}(1\ot K_i^{-1})\\
    &-\rho_i (\Xfrak\ot 1)\cdot R^{(\tau,X)} \cdot\Xfrak_{K2}\cdot (K_{\alpha_i-\Theta(\alpha_i)}\ot T_{w_X}^{-1}(E_{\tau(i)}))\\
    &+s_i (\Xfrak\ot 1)\cdot R^{(\tau,X)}\cdot \Xfrak_{K2}\cdot(K_i\ot 1).
\end{align*}
Using the above formula, Equation \eqref{eq:RX-XK2}, and again Proposition \ref{prop:TFAE}.(4) one obtains
\begin{align}
  \Big<y\ot z F_i,(\Xfrak\ot 1)& R^{(\tau,X)}  \Xfrak_{K2}  \Big> \nonumber\\
     =& - \frac{1}{q_i-q_i^{-1}} \left<y\otimes z, (1\ot r_i)((\Xfrak\ot 1)\cdot R^{(\tau,X)} \cdot \Xfrak_{K,2} )\right>\nonumber\\
     =& -\rho_i \left<y \ot z, (\Xfrak\ot 1)\cdot R^{(\tau,X)}\cdot \Xfrak_{K2}\cdot R_X\cdot (T_{w_X}^{-1}(E_{\tau(i)})\ot 1)\cdot \overline{R_X}\right>\label{eq:nearly-there}\\
    &-\rho_i \left<y \ot z, (\Xfrak\ot 1)\cdot R^{(\tau,X)} \cdot \Xfrak_{K,2}\cdot (K_{-\Theta(\alpha_i)}\ot T_{w_X}^{-1}(E_{\tau(i)}))\right>\nonumber\\
    &+s_i \left<y \ot z, (\Xfrak\ot 1)\cdot R^{(\tau,X)} \cdot \Xfrak_{K,2}\right>.\nonumber
\end{align}
Now we want to compare Equation \eqref{eq:half-way-kowX} with Equation \eqref{eq:nearly-there}. If $i\in X$ then both expressions vanish and hence coincide. Assume now that $i\in I\setminus X$. 
Applying $T_{w_X}^{-1}$ to the relation $E_{\tau(i)}F_{\tau(j)}K_{\tau(j)} = q^{-(\alpha_i, \alpha_j)} F_{\tau(j)} K_{\tau(j)} E_{\tau(i)}$ for $j\in X$ one sees that
\begin{align*}
  T_{w_X}^{-1}(E_{\tau(i)}) E_j = q^{(\alpha_j,w_X(\alpha_{\tau(i)}))} E_j T_{w_X}^{-1}(E_{\tau(i)}) \qquad \mbox{for all $j\in X$}
\end{align*}
and hence $K_{-\Theta(\alpha_i)}\ot T_{w_X}^{-1}(E_{\tau(i)})$ commutes with $R_X$.
Using \eqref{eq:kowTw01} one now gets
\begin{align*}
  \kow(T_{w_X}^{-1}(E_{\tau(i)})) &= R_X\cdot \big( T_{w_X}^{-1}(E_{\tau(i)}) \ot 1 + K_{-\Theta(\alpha_i)}\ot T_{w_X}^{-1}(E_{\tau(i)})\big) \cdot \overline{R_X}\\
    &=R_X \cdot \left(T_{w_X}^{-1}(E_{\tau(i)}) \ot 1 \right)  \overline{R_X} + K_{-\Theta(\alpha_i)}\ot T_{w_X}^{-1}(E_{\tau(i)}).
\end{align*}
Inserting the above relation into  \eqref{eq:half-way-kowX} and comparing the outcome with \eqref{eq:nearly-there} one obtains
\begin{align*}
  \left<y\ot z F_i,(\Xfrak\ot 1)\cdot R^{(\tau,X)} \cdot \Xfrak_{K2}\right> = \left<y\ot z F_i,\kow(\Xfrak)\right>
\end{align*}
also for $i\in I\setminus X$. This completes the induction step.
\end{proof}
%%%%%%%%%%%%%%%%%%%%%%%%%%%%%%%%%%%%%%%%%%%%%%%%%%
\subsection{The coproduct of $\mathcal{K}$}\label{sec:deltaK}
%%%%%%%%%%%%%%%%%%%%%%%%%%%%%%%%%%%%%%%%%%%%%%%%%%
We apply the construction given in \eqref{def:kappa-f} to the map $\tau\tau_0:P\rightarrow P$ and obtain elements  $\kappa^{\tau\tau_0}, \kappa^{-\tau\tau_0}\in\mathscr{U}^{(2)}_0$. 	
To simplify notation define
\begin{align*}
  R^{\tau\tau_0} = \sum_{\mu\in Q^+}(\tau\tau_0\ot 1)(R_\mu) = \sum_{\mu\in Q^+}(1\ot\tau\tau_0)(R_\mu) \quad \in \sU^{(2)}_{0}
\end{align*}
and
\begin{align*}
  \rh^{\tau\tau_0}=R^{\tau\tau_0}\cdot \kappa^{-\tau\tau_0}\cdot \flip \quad \in  \sU^{(2)}.
\end{align*}
Recall from Corollary \ref{cor:K} that $\cK=\Xfrak\, \xi\, T_{w_X}^{-1}\, T_{w_0}^{-1}$. 
%%%%%%%%%%%%%%%%%%%%%%%%%%%%%%%%%%%%%%%%%%%
\begin{thm}\label{thm:deltaK}
The coproduct of $\cK$ in $\sU^{(2)}$ is given by
  \begin{align}
    \kow(\cK) = (\cK\ot 1)\cdot \rh^{\tau\tau_0}\cdot(\cK\ot 1)\cdot \rh. \label{eq:deltaK}
  \end{align}
\end{thm}
%%%%%%%%%%%%%%%%%%%%%%%%%%%%%%%%%%%%%%%%%%%
\begin{proof}
Using \ref{prop:TFAE}.(4) one verifies that in $\sU^{(2)}_{0}$ one has
\begin{align}
    (T_{w_X}T_{w_0}\ot 1)\cdot \Xfrak_{K2}\cdot(T_{w_X}T_{w_0}\ot 1)^{-1}&= \kappa^{-\tau\tau_0}\cdot(1\ot \Xfrak)\cdot \kappa^{\tau\tau_0}.\label{eq:KX-1X}
\end{align}    
Moreover, the following relation holds
\begin{align}
    \kappa^{\tau\tau_0}\cdot(T_{w_0}T_{w_X}\ot 1) &= (T_{w_0}T_{w_X}\ot 1)\cdot\kappa^{\Theta}.\label{eq:kappaTT}  
  \end{align}
Combining the above two formulas with all our previous preparations we calculate
  \begin{align*}
    \kow(\cK)&\stackrel{\phantom{\eqref{eq:kowX}}}{=} \kow(\Xfrak)\cdot \kow(\xi) \cdot \kow(T_{w_X}^{-1})\cdot \kow(T_{w_0}^{-1})\\
    &\stackrel{\eqref{eq:kowX}}{=} (\Xfrak\ot 1)\cdot R^{(\tau,X)}\cdot\Xfrak_{K2}\cdot \kow(\xi) \cdot \kow(T_{w_X}^{-1})\cdot \kow(T_{w_0}^{-1})\\
    &\stackrel{\eqref{eq:RtX-def}}{=} (\cK\ot 1)\cdot R^{\tau\tau_0}\cdot(T_{w_X}T_{w_0}\xi^{-1}\ot 1)\cdot \overline{R_X}\cdot\Xfrak_{K2}\cdot \kow(\xi) \cdot \kow(T_{w_X}^{-1})\cdot \kow(T_{w_0}^{-1})\\
    &\stackrel{\eqref{eq:RX-XK2}}{=} (\cK\ot 1)\cdot R^{\tau\tau_0}\cdot(T_{w_X}T_{w_0}\xi^{-1}\ot 1) \cdot\Xfrak_{K2}\cdot \kow(\xi)\cdot\overline{R_X} \cdot \kow(T_{w_X}^{-1})\cdot \kow(T_{w_0}^{-1})\\
    &\stackrel{\eqref{eq:kowTw01}}{=} (\cK\ot 1)\cdot R^{\tau\tau_0}\cdot(T_{w_X}T_{w_0}\xi^{-1}\ot 1) \cdot\Xfrak_{K2}\cdot \kow(\xi) \cdot T_{w_X}^{-1}\ot T_{w_X}^{-1}\cdot \kow(T_{w_0}^{-1})\\
    &\stackrel{\eqref{eq:KX-1X}}{=} (\cK\ot 1)\cdot R^{\tau\tau_0}\cdot \kappa^{-\tau\tau_0}\cdot (1\ot \Xfrak)\cdot \kappa^{\tau\tau_0}\cdot(T_{w_X}T_{w_0}\xi^{-1}\ot 1)\cdot \kow(\xi) \\
    &\phantom{shakeitshakeitshakeit------------}\cdot T_{w_X}^{-1}\ot T_{w_X}^{-1}\cdot \kow(T_{w_0}^{-1})\\
    &\stackrel{\eqref{eq:kappaTT}}{=} (\cK\ot 1)\cdot R^{\tau\tau_0}\cdot \kappa^{-\tau\tau_0}\cdot (1\ot \Xfrak)\cdot (T_{w_X}T_{w_0}\xi^{-1}\ot 1)\cdot\kappa^{\Theta}\cdot \kow(\xi)\\
    &\phantom{shakeitshakeitshakeit------------} \cdot T_{w_X}^{-1}\ot T_{w_X}^{-1}\cdot \kow(T_{w_0}^{-1})\\
    &\stackrel{\eqref{DeltaXi}}{=} (\cK\ot 1)\cdot R^{\tau\tau_0}\cdot \kappa^{-\tau\tau_0}\cdot (1\ot \Xfrak\xi)\cdot (T_{w_X}T_{w_0}\ot 1)\cdot\kappa^{-1}\\
    &\phantom{shakeitshakeitshakeit------------}\cdot T_{w_X}^{-1}\ot T_{w_X}^{-1}\cdot \kow(T_{w_0}^{-1})\\
    &\stackrel{\eqref{eq:kowTw02}}{=} (\cK\ot 1)\cdot R^{\tau\tau_0}\cdot \kappa^{-\tau\tau_0}\cdot (1\ot \Xfrak\xi)\cdot (T_{w_X}T_{w_0}\ot 1)\cdot T_{w_X}^{-1}\ot T_{w_X}^{-1}\\
   &\phantom{shakeitshakeitshakeit------------} \cdot T_{w_0}^{-1}\ot T_{w_0}^{-1}\cdot R_{21}\cdot \kappa^{-1}\\
    &\stackrel{\phantom{\eqref{eq:kowTw02}}}{=} (\cK\ot 1)\cdot R^{\tau\tau_0}\cdot \kappa^{-\tau\tau_0}\cdot (1\ot \Xfrak\xi T_{w_0}^{-1}T_{w_X}^{-1})\cdot R_{21}\cdot \kappa^{-1}\\
    &\stackrel{\phantom{\eqref{eq:kowTw02}}}{=} (\cK\ot 1)\cdot \rh^{\tau\tau_0}\cdot (\cK\ot 1) \cdot \rh
  \end{align*}
which gives the desired formula.  
\end{proof}
%%%%%%%%%%%%%%%%%%%%%%%%%%%%%%%%%%%%%%%%%%%%%%%%%%%%
Recall Definition \ref{def:U-cylinder-braided} of a $\varphi$-universal K-matrix. Combining Corollary \ref{cor:K} with Theorem \ref{thm:deltaK} we obtain the first statement of the following corollary. 
%%%%%%%%%%%%%%%%%%%%%%%%%%%%%%%%%%%%%%%%%%%%%%%%%%%%
\begin{cor}
  The element $\cK=\Xfrak\, \xi\, T_{w_X}^{-1}\, T_{w_0}^{-1}\in \sU$ is a $\tau\tau_0$-universal K-matrix for the quantum symmetric pair coideal subalgebra $\coid$ of $\uqg$. In particular, $\cK$ satisfies the reflection equation
  \begin{align}\label{eq:RE-final}
     (\cK\ot 1)\cdot \rh^{\tau\tau_0}\cdot(\cK\ot 1)\cdot \rh=\rh\cdot(\cK\ot 1)\cdot \rh^{\tau\tau_0}\cdot(\cK\ot 1).
  \end{align}
in $\sU^{(2)}$.  
\end{cor} 
%%%%%%%%%%%%%%%%%%%%%%%%%%%%%%%%%%%%%%%%%%%%%%%%%%%%
\begin{proof}
  The second statement follows from Remark \ref{rem:uniK=tw-cyl-tw} and Equation \eqref{eq:RE-phi}. Here one needs to observe that $\rh_{N^{\tau\tau_0},M^{\tau\tau_0}}=\rh_{N,M}$ considered as a map $N\otimes M\rightarrow M\otimes N$ for all $M,N\in \Oint$. 
\end{proof}
%%%%%%%%%%%%%%%%%%%%%%%%%%%%%%%%%%%%%%%%%%%
\begin{rema}
  The approach to quantum symmetric pairs in the papers \cite{a-Noumi96}, \cite{a-NS95}, \cite{a-NDS97}, \cite{a-Dijk96} is based on explicit solutions of the reflection equation. In \cite{a-Noumi96} Noumi first found a reflection equation for the symmetric pairs of type 
  \begin{align*}
  AI: (\slfrak_N,\mathfrak{so}_N), \qquad AII:(\slfrak_{2N+1}, \mathfrak{sp}_{2N}).
  \end{align*} 
For the symmetric pairs of type 
  \begin{align*}
    AIII:(\slfrak_{M+N}, \mathfrak{s}(\mathfrak{gl}_M \times \mathfrak{gl}_N))
  \end{align*}
a different reflection equation appeared in \cite{a-NDS97}. The differing reflection equations are unified by Equation \eqref{eq:RE-final}. Indeed, the diagram automorphism $\tau\tau_0$ is non-trivial in types AI and AII, while $\tau\tau_0$ is the identity in type $AIII$. 
\end{rema}
%%%%%%%%%%%%%%%%%%%%%%%%%%%%%%%%%%%%%%%%%%%
\providecommand{\bysame}{\leavevmode\hbox to3em{\hrulefill}\thinspace}
\providecommand{\MR}{\relax\ifhmode\unskip\space\fi MR }
\providecommand{\MRhref}[2]{%
  \href{http://www.ams.org/mathscinet-getitem?mr=#1}{#2}
}
\providecommand{\href}[2]{#2}

%%%%%%%%%%%%%%%%%%%%%%%%%%%%%%%%%%%%%%%%%%%

\end{document}